\documentclass[11pt,reqno]{booktemplate}
\usepackage[utf8]{inputenc}
\usepackage[unicode=true,bookmarks=true]{hyperref}
\usepackage[a4paper, hmargin={2.7cm,2.7cm},vmargin={3.3cm,2.3cm}]{geometry}
\usepackage[english]{babel}
\usepackage{color}
\usepackage[nospace,noadjust]{cite}
\usepackage{amsmath}
\usepackage{amssymb}
\usepackage{amsfonts}
\usepackage{booktabs}
\usepackage{amsthm}
\usepackage[T1]{fontenc}
\usepackage{dsfont}
\usepackage{etoolbox}
\usepackage{mathrsfs}
\usepackage{enumitem}
\usepackage{lmodern}
\usepackage{todonotes}
\usepackage{cleveref}

\DeclareRobustCommand{\gobblefive}[5]{}

\hypersetup{
    bookmarksnumbered=true,
    bookmarksopen=true,
    bookmarksopenlevel=1,
    bookmarksdepth=1
}

\makeatletter
\providecommand*{\cupdot}{%
  \mathbin{%
    \mathpalette\@cupdot{}%
  }%
}
\newcommand*{\@cupdot}[2]{%
  \ooalign{%
    $\m@th#1\cup$\cr
    \sbox0{$#1\cup$}%
    \dimen@=\ht0 %
    \sbox0{$\m@th#1\cdot$}%
    \advance\dimen@ by -\ht0 %
    \dimen@=.5\dimen@
    \hidewidth\raise\dimen@\box0\hidewidth
  }%
}

\providecommand*{\bigcupdot}{%
  \mathop{%
    \vphantom{\bigcup}%
    \mathpalette\@bigcupdot{}%
  }%
}
\newcommand*{\@bigcupdot}[2]{%
  \ooalign{%
    $\m@th#1\bigcup$\cr
    \sbox0{$#1\bigcup$}%
    \dimen@=\ht0 %
    \advance\dimen@ by -\dp0 %
    \sbox0{\scalebox{2}{$\m@th#1\cdot$}}%
    \advance\dimen@ by -\ht0 %
    \dimen@=.5\dimen@
    \hidewidth\raise\dimen@\box0\hidewidth
  }%
}
\makeatother

\makeatletter
\patchcmd{\ttlh@hang}{\parindent\z@}{\parindent\z@\leavevmode}{}{}
\patchcmd{\ttlh@hang}{\noindent}{}{}{}
\makeatother

\renewcommand*{\theequation}{%
  \ifnum\value{chapter}=0 %
  \arabic{equation}%
  \else
    \ifnum\value{section}=0 %
      \thechapter
    \else
      \thesection
    \fi
    .\arabic{equation}%
  \fi
}

\newcommand\numberthis{\addtocounter{equation}{1}\tag{\theequation}}

\usepackage[ddmmyyyy,hhmmss]{datetime}
\newdateformat{daymonthyeardate}{%
  \THEDAY.\THEMONTH.\THEYEAR}
\newdateformat{mydateformat}{%
  \THEDAY.\THEMONTH.\THEYEAR}
\usepackage{lastpage}
\usepackage{etoolbox}

\theoremstyle{plain}
\newtheorem{theorem}{Theorem}[chapter]
\newtheorem{lemma}[theorem]{Lemma}
\newtheorem{proposition}[theorem]{Proposition}
\newtheorem{corollary}[theorem]{Corollary}

\theoremstyle{definition}
\newtheorem{definition}[theorem]{Definition}
\newtheorem{assumption}[theorem]{Assumption}
\newtheorem{examplex}[theorem]{Example}
\newenvironment{example}
  {\pushQED{\qed}\renewcommand{\qedsymbol}{$\blacksquare$}\examplex}
  {\popQED\endexamplex}

\theoremstyle{remark}
\newtheorem{remark}[theorem]{Remark}
\newtheorem*{remark*}{Remark}

\counterwithin{theorem}{chapter}
\numberwithin{equation}{chapter}
\numberwithin{section}{chapter}

\let\emptyset\varnothing

\DeclareMathOperator*{\dom}{dom}

\DeclareMathOperator*{\loc}{loc}

\DeclareMathOperator{\Co}{Co}
\DeclareMathOperator*{\rel}{Rel}
\DeclareMathOperator*{\Span}{span}
\DeclareMathOperator*{\supp}{supp}

\DeclareMathOperator*{\esssup}{ess\,sup}

\newcommand{\R}{\mathbb{R}}
\newcommand{\N}{\mathbb{N}}

\newcommand{\T}{\mathbb{T}}
\newcommand{\CC}{\mathbb{C}}
\newcommand{\gramian}{\mathscr{G}}
\newcommand{\synthesis}{\mathscr{D}}
\newcommand{\frameop}{\mathscr{S}}
\newcommand{\analysis}{\mathscr{C}}
\newcommand{\identity}{\mathrm{id}}
\newcommand{\goodMatricesUnweighted}{\mathcal{C}}
\newcommand{\goodMatrices}{\goodMatricesUnweighted^p_w}
\newcommand{\dominated}{\prec}

\newcommand{\CalH}{\mathcal{H}}

\newcommand{\CalU}{\mathcal{U}}
\newcommand{\CalK}{\mathcal{K}}

\newcommand{\Hpi}{\mathcal{H}_{\pi}}
\newcommand{\Aw}{\mathcal{A}_{w}}
\newcommand{\Bwp}{\mathcal{B}_{w}^p}
\newcommand{\Bw}{\mathcal{B}_{w}}
\newcommand{\Hw}{\mathcal{H}^1_{w}}

\newcommand{\indicator}{\mathds{1}}
\newcommand{\eps}{\varepsilon}

\newcommand{\normm}[1]{{\left\vert\kern-0.25ex\left\vert\kern-0.25ex\left\vert #1
    \right\vert\kern-0.25ex\right\vert\kern-0.25ex\right\vert}}

\newcommand{\group}{G}
\newcommand{\haarMeasure}{\mu_G}
\newcommand{\dd}[1]{d \haarMeasure(#1)}
\newcommand{\hilbert}{\mathcal{H}_\pi}

\newcommand{\maxR}{M^R}
\newcommand{\maxL}{M^L}
\newcommand{\maxSt}{M}

\newcommand{\wienerLNeutral}{W^L}
\newcommand{\wienerL}[1]{\wienerLNeutral(#1)}
\newcommand{\wienerR}[1]{W^R(#1)}
\newcommand{\wienerSt}[1]{W(#1)}
\newcommand{\wienerStC}[1]{W_C(#1)}
\newcommand{\diff}{d}

\newcommand{\Reservoir}{\mathcal{R}}

\newcommand{\cardinality}[1]{\# #1}

\newcommand{\translationL}[1]{L_{#1}}
\newcommand{\translationR}[1]{R_{#1}}

\newcommand{\twisttranslationL}[1]{L^{\sigma}_{#1}}
\newcommand{\twisttranslationR}[1]{R^{\sigma}_{#1}}

\binoppenalty=\maxdimen
\relpenalty=\maxdimen
\begin{document}
\frontmatter

\title[Coorbit spaces associated to quasi-Banach function spaces]{Coorbit spaces associated to quasi-Banach function spaces and their molecular decomposition}
\author{Jordy Timo van Velthoven}

\address{Delft University of Technology,
Mekelweg 4, Building 36,
2628 CD Delft, The Netherlands.}

\address{Faculty of Mathematics, University of Vienna, Oskar-Morgenstern-Platz 1, 1090 Vienna, Austria}
\email{jordy.timo.van.velthoven@univie.ac.at}

\author{Felix Voigtlaender}

\address{Department of Mathematics,
Technical University of Munich, 
Boltzmann-str. 3,
85748 Garching bei München,
Germany
}

\address{Mathematical Institute for Machine Learning and Data Science (MIDS),
  Catholic University of Eichst\"att–Ingolstadt (KU), Auf der Schanz 49, 85049 Ingolstadt, Germany
}

\email{felix.voigtlaender@ku.de, felix@voigtlaender.xyz}

\thanks{The authors sincerely thank the Mathematisches Forschungsinstitut Oberwolfach (MFO)
for the support through the ``Research in Pairs'' program.
For Jordy Timo van Velthoven, this research was funded in whole or in part by the Austrian
Science Fund (FWF): 10.55776/J4555.
F.\ Voigtlaender acknowledges support by the German Research Foundation (DFG)
in the context of the Emmy Noether junior research group VO 2594/1--1.}


%22A10: Analysis on general topological groups
%42C15: General harmonic expansions, frames
%42B35: Function spaces arising in harmonic analysis
%43A15: $L^p$-spaces and other function spaces on groups, semigroups, etc.
%46B15: Summability and bases; functional analytic aspects of frames in Banach and Hilbert spaces [See also 46A35, 42C15]
%46E22: Hilbert spaces with reproducing kernels (= (proper) functional Hilbert spaces, including de Branges-Rovnyak and other structured spaces)


\begin{abstract}
\pagestyle{empty}
This paper provides a self-contained exposition of coorbit spaces
associated to integrable group representations and quasi-Banach function spaces,
and at the same time extends and simplifies previous work.
The main results provide an extension of the theory in [Studia Math., 180(3):237–253, 2007]
from groups admitting a compact, conjugation-invariant unit neighborhood to arbitrary
(possibly nonunimodular) locally compact groups.
In addition, the present paper establishes the existence of molecular dual frames
and Riesz sequences as in [J.\ Funct.\ Anal., 280(10):56, 2021]
for the full scale of quasi-Banach function spaces. 
The theory is developed for possibly projective and reducible unitary representations
in order to be easily applicable to well-studied function spaces not satisfying the classical assumptions of coorbit theory.
Compared to the existing literature on quasi-Banach coorbit spaces,
all our results apply under significantly weaker integrability conditions on the analyzing vectors,
which allows for obtaining sharp results in concrete settings. 
\end{abstract}

\maketitle

{\hypersetup{linkcolor=black}\tableofcontents}

\mainmatter

\addtocontents{toc}{\protect\setcounter{tocdepth}{0}}

\chapter{Introduction}
\label{sec:Introduction}

For a continuous representation $(\pi, \mathcal{A})$ of a locally compact group $G$
on a Banach or Fr\'echet space $\mathcal{A}$ and a vector space $Y$ of measurable functions on $G$,
the coorbit method provides a procedure for constructing an associated distribution space
$\Co(Y)$, namely
\begin{align} \label{eq:coorbit_intro}
  \Co (Y)
  := \big\{
      f \in \mathcal{A}^* \; : \; V_g f \in Y
     \big\}
  \quad \text{with} \quad
  V_g f (x) = \langle f, \pi(x) g \rangle,
\end{align}
where $g \in \mathcal{A} \setminus \{0\}$ is a fixed vector and where $\mathcal{A}^*$
denotes the associated (anti)-linear dual space.
Common choices for $\mathcal{A}$ are the space of integrable vectors (resp.\ smooth vectors)
whenever $\pi$ is an integrable representation (resp.\ $G$ is a Lie group).

The influential series of papers
\cite{feichtinger1988unified,feichtinger1989banach1,feichtinger1989banach2,groechenig1991describing}
introduced the spaces $\Co (Y)$ for an irreducible, integrable unitary representation $\pi$
and a Banach function space $Y$ and established general properties of these spaces.
Among others, it was shown that the spaces \eqref{eq:coorbit_intro} are independent
of the defining vector $g$, that each space $\Co (Y)$ admits an atomic decomposition
in terms of a suitable subsystem $\big( \pi(x_i) g \big)_{i \in I}$ of the orbit $\pi(G) g$,
and that properties such as inclusions, embeddings and minimality/maximality
can be completely characterized by the corresponding properties of associated sequence spaces.
In addition, the papers
\cite{feichtinger1988unified,feichtinger1989banach1,feichtinger1989banach2,groechenig1991describing}
revealed that many classical function spaces in complex and harmonic analysis---%
such as Bergman spaces, Fock spaces, Hardy spaces, and (homogeneous) Besov and Sobolev spaces---%
can be realized as a coorbit space.

The purpose of the present paper is to provide a self-contained exposition
of coorbit spaces $\Co (Y)$ associated with integrable representations $\pi$
and \emph{quasi}-Banach function spaces $Y$, i.e., spaces where the triangle inequality
of a norm is replaced by $\| F_1 + F_2 \|_{Y} \leq C  (\| F_1 \|_Y + \|F_2 \|_{Y})$.
To some extent this has already been done in \cite{rauhut2007coorbit},
but only for the restrictive setting of groups
with a conjugation-invariant compact unit neighborhood (IN groups).
The present paper removes this assumption and, in addition, establishes the existence of dual coorbit molecules
as in \cite{MoleculePaper} for the full scale of \emph{quasi}-Banach spaces.
It also simplifies and unifies several results obtained in previous papers, even for the case of genuine Banach spaces.

\section*{Motivation}

Our motivation for considering quasi-Banach coorbit spaces is two-fold.

The first motivation stems from the fact that all of the above mentioned classical function spaces
have aside their Banach range also a natural range of parameters yielding quasi-Banach spaces.
For example, the Hardy spaces $H^p (\mathbb{R}^d)$, where $p\in (0,\infty]$,
are Banach spaces for $p \in [1,\infty]$, but merely quasi-Banach spaces for $p \in (0,1)$.
Therefore, in order to treat these examples for the full range of parameters
in the setting of coorbit theory, it is essential for the latter theory
to also apply to quasi-Banach spaces.
It should be mentioned that the Hardy spaces can be identified with coorbit spaces
associated with the (nonunimodular) affine group, and hence they cannot be treated
in the setting of \cite{rauhut2007coorbit} as it only applies to IN groups.

A second motivation stems from applications of coorbit theory to \emph{nonlinear approximation}.
Here, given a family $(f_i)_{i \in I} \subseteq \Hpi \vphantom{\sum_{j_i}}$ of elements of a Hilbert (or Banach) space $\Hpi$,
the objective is to seek for given $f \in \Hpi$ a good approximation
$\vphantom{\sum^{T^i}} \widetilde{f} = \sum_{i \in {I_f}} c_i \, f_i$ to $f$
under the restriction $|I_f| \leq K$, i.e., using only a fixed number of elements.
The associated map $f \mapsto \widetilde{f}$ is not necessarily linear.
In the context of the finite-dimensional space $\Hpi = \mathbb{C}^N$
and with $(f_i)_i = (e_i)_{i = 1}^N$ being the standard basis,
the best approximation (with respect to any $\ell^p$-norm on $\CC^N$)
to $v \in \CC^N$ is obtained by $\widetilde{v} = v \cdot \indicator_{I_v}$,
with $\indicator_{I_v}$ being the indicator of the set $I_v$ containing the indices
of the $K$ largest entries of $v$ (in absolute value).
The associated minimal approximation error is denoted by
\[
  \sigma_K (v)_p
  = \min
    \big\{
      \| v - u \|_{\ell^p}
      \; : \;
      u \in \mathbb{C}^N \;\; \text{is $K$-sparse}
    \big\},
\]
where $u \in \mathbb{C}^N$ is called \emph{$K$-sparse} if at most $K$ entries of $u$ are nonzero.
This error obeys the following bound, sometimes referred to as \emph{Stechkin's inequality}:
\begin{align} \label{eq:error_bound}
  \sigma_K (v)_q
  \leq K^{\frac{1}{q} - \frac{1}{p}}  \| v \|_{\ell^p}
  \quad \text{for} \quad
  0 < p \leq q \leq \infty,
\end{align}
see, e.g., \cite[Proposition~2.3]{foucart2013mathematical}.
Hence, $\sigma_K (v)_q$ decays fast for those $v \in \CC^N$ that have $\|v\|_{\ell^p}$ of reasonable size
with a \emph{small} value of $p > 0$.
In the setting of coorbit spaces, one can similarly show (based on the discretization theory of coorbit spaces)
that elements of $\Co(L^p)$ with small $p > 0$ can be well approximated by ``sparse vectors''
(i.e., by linear combinations of the family $\big( \pi(x_i) g \big)_{i \in I}$ with at most $K$ terms)
as elements of $\Co(L^q)$ for $q \gg p$, where it is important to note $\Co(L^2) = \hilbert$.
Hence, the \emph{quasi}-Banach spaces $\Co(L^p)$ with $p \in (0,1)$ play an important role
in nonlinear approximation.
See also the motivating discussions in \cite{fuehr2016vanishing, rauhut2007coorbit} for considering coorbit spaces $\Co(L^p)$ with $p < 1$
for the purpose of nonlinear approximation.

\section*{Related work}

The theory of coorbit spaces $\Co (Y)$ with a Banach function space $Y$
as developed in \cite{feichtinger1988unified,feichtinger1989banach1,feichtinger1989banach2,groechenig1991describing}
crucially relies on convolution relations of the form
\begin{align}\label{eq:convolution_intro}
  Y \ast L^1_w (G) \hookrightarrow Y
  \quad \text{and} \quad
  L^1_w (G) \ast Y \hookrightarrow Y
\end{align}
for a suitable weight function $w : G \to [1,\infty)$, often called a \emph{control weight} for $Y$.
If $Y$ is merely a \emph{quasi}-Banach space,
then relations such as \eqref{eq:convolution_intro} cannot be expected to hold,
e.g., for $Y = L^p (G)$ with $p \in (0,1)$.

The paper \cite{rauhut2007coorbit} considered coorbit spaces $\Co (Y)$
associated with quasi-Banach spaces $Y$ and used
instead of the embeddings \eqref{eq:convolution_intro} certain convolution relations
for so-called (left) \emph{Wiener amalgam spaces} $\wienerL{Y}$
(see, e.g., \cite{feichtinger1983banach,busby1981product,holland1975harmonic,fournier1985amalgams}
for such Banach spaces), as proven in \cite{rauhut2007wiener}.
In \cite{rauhut2007coorbit}, the coorbit theory is developed
only for invariant neighborhood (IN) groups, i.e., groups $G$ admitting a compact unit neighborhood
$U \subseteq G$ satisfying $x U x^{-1} = U$ for all $x \in G$.
In the setting of IN groups, several of the convolution relations in \cite{rauhut2007wiener}
for general locally compact groups possess simpler versions, but this setting is very restrictive.
For example, it excludes simply connected nilpotent Lie groups (e.g.,
the nonreduced Heisenberg group) and connected locally compact groups
with exponential volume growth (e.g., the affine group), see \cite{iwasawa1951topological,palmer1978classes}.

It should be pointed out that the preprint version \cite{rauhutcoorbitpreprint}
of \cite{rauhut2007coorbit} also proposed a coorbit theory valid on general
(possibly, non-IN) locally compact groups.
However, as observed in the PhD thesis \cite{FelixPhD} of the second named author,
an essential convolution relation asserted in \cite{rauhut2007wiener}
(which is used in \cite{rauhut2007coorbit, rauhutcoorbitpreprint})
is incorrect and fails for general groups.
More precisely, two examples presented in Section~\ref{sec:counterexamples} below
(both taken from \cite{FelixPhD}) show that
\[
  \wienerL{L^p_v} \ast \bigl[\wienerL{Y^{\vee}}\bigr]^{\vee}
  \nsubseteq \wienerL{Y}
\]
with weights $v (x) = \|L_{x^{-1}} \|_{\wienerL{Y} \to \wienerL{Y}}$
and $v(x) = \| L_{x} \|_{\wienerL{Y} \to \wienerL{Y}}$,
for the left-trans\-la\-tion operator $L_x F (y) = F(x^{-1} y)$ and involution $F^{\vee} (x) = F(x^{-1})$.
These examples show that the asserted \cite[Theorem~5.2]{rauhut2007wiener} fails in general.

A corrected and modified version of the theory proposed
in \cite{rauhut2007coorbit,rauhutcoorbitpreprint}, valid for general locally compact groups,
is contained in the thesis \cite{FelixPhD}.
The theory in \cite{FelixPhD} shows that essentially all the basic properties
of coorbit spaces known in the Banach space setting
\cite{feichtinger1988unified,feichtinger1989banach1,feichtinger1989banach2,groechenig1991describing}
remain valid for quasi-Banach spaces, if essential modifications are made at appropriate places,
both in the definition of coorbit spaces and in the proofs of their basic properties.

Lastly, a very general theory of \emph{quasi}-Banach coorbit spaces has been developed in \cite{kempka2017general},
where also coorbit spaces associated with general reproducing formulae
not necessarily arising from group representations can be treated,
see also \cite{fornasier2005continuous, rauhut2011generalized}.
Some of the basic properties of coorbit spaces for the group case can be extracted
from the general theory \cite{kempka2017general}, but this can be a rather daunting task
(especially for interested nonexperts) due to the many technical assumptions
required in \cite{kempka2017general}.
In addition, as the present article shows, many of these technicalities
can simply be avoided or considerably simplified in the group case considered here.
There are also significant parts of the present article that are currently not available
outside the group setting, most notably the theory of coorbit molecules.

\section*{Contributions}

The present paper provides a self-contained, greatly simplified and generalized,
exposition of coorbit spaces associated with integrable representations and quasi-Banach function spaces.
The exposition is similar to \cite{FelixPhD}, but contains important further simplifications
and improvements, which will be commented on at the relevant places throughout the text.
As one of the biggest simplifications, it is not assumed here
(in contrast to \cite{rauhutcoorbitpreprint,rauhut2007wiener,FelixPhD})
that the so-called control weight $w$ satisfies $w(x) \geq \| R_x \|_{\wienerL{Y} \to \wienerL{Y}}$;
instead, it is only required that $w(x) \geq \| R_x \|_{Y \to Y}$.
Although it is true that the Wiener amalgam space $\wienerL{Y}$
is right-invariant whenever $Y$ is (cf.\ \cite[Corollary~4.2]{rauhut2007wiener}),
it does not seem possible to \emph{readily} estimate $\|R_x \|_{\wienerL{Y} \to \wienerL{Y}}$
in terms of $\|R_x \|_{Y \to Y}$ whenever $G$ is a non-IN group, see, e.g., \cite[Lemma~2.3.18]{FelixPhD}.
This simplification of using $\| R_x \|_{Y \to Y}$ instead of $\| R_x \|_{\wienerL{Y} \to \wienerL{Y}}$
is important, since having a control weight of only moderate growth
is essential for obtaining sharp explicit conditions on atoms yielding atomic decompositions.
To see that just estimating $\|R_x \|_{Y \to Y}$ instead of $\|R_x \|_{\wienerL{Y} \to \wienerL{Y}}$
is advantageous, we mention that large parts of the recent PhD thesis \cite{AchimPhD}
are concerned with deriving usable bounds for the operator norm $\| R_x \|_{\wienerL{Y} \to \wienerL{Y}}$
for settings in which $\| R_x \|_{Y \to Y}$ can be readily estimated.

The biggest improvements in the present article compared to \cite{rauhut2007coorbit,FelixPhD}
are the results on atomic decompositions.
Instead of adapting the classical sampling techniques from
\cite{feichtinger1989banach1,groechenig1991describing} as done in \cite{rauhut2007coorbit,FelixPhD},
the present article provides an extension of the recent results on dual molecules
as in \cite{MoleculePaper} to the setting of quasi-Banach spaces.
This approach is more easily accessible and at the same time yields much stronger conclusions
on the localization of the dual system by showing that it also forms a family of molecules
satisfying similar envelope conditions as the basic atoms.
In particular, the obtained results on molecules close a gap
between what was known for general (quasi-Banach) coorbit spaces \cite{FelixPhD}
and the concrete setting of Besov-Triebel-Lizorkin spaces on $\mathbb{R}^d$, \cite{frazier1990discrete, memoir}.

Lastly, several arguments presented in this article are different
and often simpler than the classical arguments
in \cite{feichtinger1988unified,feichtinger1989banach1,feichtinger1989banach2,groechenig1991describing},
even for the setting of Banach spaces.
For example, the present article avoids the use of an atomic decomposition
to prove that the space of integrable vectors (so-called analyzing vectors) is dense,
and instead provides a direct proof using Bochner integration, cf.\ Section~\ref{sec:integrable};
in particular, \Cref{lem:basic_Hw}.

\section*{Extensions}

To ensure a wide applicability, all the results in the present article are proven
for possibly \emph{reducible} and/or \emph{projective} group representations.
This flexibility is highly beneficial for treating several key examples.
For example, reducible representations occur naturally in high-dimensional wavelet theory
\cite{fuehr2019coorbit,fuehr2015coorbit,koppensteiner2023anisotropic, koppensteiner2023anisotropic2},
whereas the use of (possibly noncontinuous) projective representations
is convenient for treating weighted Bergman spaces
\cite{christensen2019coorbits, christensen2017new, christensen2021new}
or representations of nilpotent Lie groups that are only square-integrable
\emph{modulo a central subgroup} \cite{grochenig2021new, fischerpreprint}.

As possible extensions of the theory
\cite{feichtinger1988unified,feichtinger1989banach1,feichtinger1989banach2,groechenig1991describing},
it was already mentioned in \cite{feichtinger1988unified} that the case of quasi-Banach spaces,
reducible and/or projective representations would be desirable for treating several key examples.
For coorbit spaces associated with Banach spaces, extensions to reducible and projective representations
can be found in \cite{christensen2011coorbit, dahlke2017coorbit}
and \cite{christensen2019coorbits, christensen1996atomic}, respectively.
For \emph{quasi}-Banach spaces, the present article seems to be the first to develop these extensions.

\section*{Overview}

\Cref{sec:prelim} introduces general notation used throughout the paper
and provides background on quasi-Banach function spaces and local maximal functions.
Convolution relations for Wiener amalgam spaces are proved in \Cref{sec:WienerAmalgam}.
In addition, \Cref{sec:counterexamples} contains two counterexamples
to convolution relations asserted in \cite{rauhut2007wiener}.
\Cref{sec:coorbit} is devoted to the definition of coorbit spaces
and to studying their basic properties as quasi-Banach spaces.
Convolution-dominated integral operators and matrices form the subject of \Cref{sec:CD}.
Among others, it will be shown that these classes of operators and matrices
form algebras and possess a local holomorphic functional calculus.
The results on convolution-dominated operators will be exploited in \Cref{sec:molecules}
to prove the existence of dual coorbit molecules of frames and Riesz sequences
and to derive associated decompositions of coorbit spaces.
\Cref{sec:applications} presents several applications of the obtained results
to the boundedness of operators.
Simplified statements of the main results for irreducible, square-integrable representations are given in \Cref{sec:discrete_series}.
The proofs of several miscellaneous and technical results are postponed to the appendix.

\addtocontents{toc}{\protect\setcounter{tocdepth}{1}}

\chapter{Function spaces and local maximal functions}
\label{sec:prelim}

This chapter provides background on solid function spaces,
twisted convolution and quasi-Banach spaces.

\section{General notation}
\label{sec:notation}

Throughout this article, $G$ denotes a second countable locally compact group
and $Q \subseteq G$ denotes a fixed open, relatively compact symmetric neighborhood
of the identity element $e_G \in G$.
In addition, a left Haar measure $\mu_G$ on $G$ is fixed and the associated modular function
is denoted by $\Delta : \group \to (0,\infty)$.

We write $\mathbb{T} := \{ z \in \CC \colon |z| = 1 \}$.
For $x \in \R^d$ or $x \in \CC$, we denote by $B_r (x)$ the open ball
(with respect to the usual Euclidean norm) of radius $r > 0$ around $x$.
For a subset $V$ of a ``base set'' $X$ that is usually implied by the context,
the notation $\indicator_V$ denotes the indicator function of $V$, i.e.,
$\indicator_V (x) = 1$ if $x \in V$ and $\indicator_V (x) = 0$ otherwise.
A pairwise disjoint union of sets will be denoted by the symbol $\bigcupdot$.

For functions $F_1, F_2 : X \to [0,\infty)$ on a set $X$
(which in most cases will be implied by the context), the notation $F_1 \lesssim F_2$
means that there exists a constant $C > 0$ such that $F_1 (x) \leq C  F_2 (x)$ for all $x \in X$.
The notation $F_1 \asymp F_2$ is used if $F_1 \lesssim F_2$ and $F_2 \lesssim F_1$.
For a function $F : G \to \mathbb{C}$, its involution $F^{\vee} : G \to \mathbb{C}$
is defined by $F^{\vee} (x) = F(x^{-1})$ for $x \in G$.
A function $w : G \to (0,\infty)$ will be called \emph{submultiplicative}
if $w(x y) \leq w(x) w(y)$ for all $x,y \in G$.

The Lebesgue space of $r$-integrable functions is denoted by $L^r (\group)$ for $0 < r \leq \infty$
and defined relative to $\mu_G$.
For a (measurable) function $w : G \to (0,\infty)$,
the associated weighted spaces $L^r_w (G)$ are equipped
with the norm $\| F \|_{L^r_w} := \| F \cdot w \|_{L^r}$.
For $r \in [1, \infty)$, the bi-linear dual pairing between $L^r_w (G)$
and its dual space $L^{r'}_{1/w} (G)$ is denoted $(\cdot , \cdot)_{L^r_w, L^{r'}_{1/w}}$,
while the sesquilinear pairing (which is anti-linear in the second component)
is denoted by $\langle \cdot , \cdot \rangle_{L^r_w, L^{r'}_{1/w}}$.
If the involved spaces are clear from the context, these pairings will also
sometimes simply be denoted by $(\cdot, \cdot)$ and $\langle \cdot , \cdot \rangle$.

\section{Function spaces}
\label{sec:QBF}

A vector space $Y$ is called a \emph{quasi-normed space}
if it is equipped with a map $\| \cdot \| : Y \!\to\! [0,\infty)$ that is positive definite
(i.e., $\| f \| > 0$ for $f \neq 0$),
absolute homogeneous (i.e., $\| \alpha \, f \| = |\alpha| \, \| f \|$)
and such that there exists a constant $C > 0$ satisfying
$\|f + g \| \leq C  (\| f \| + \|g\|)$ for all $f, g \in Y$;
such a map $\| \cdot \|$ is called a \emph{quasi-norm} on $Y$
with \emph{triangle constant} $C > 0$.
For $p \in (0,1]$, a map $\|\cdot\| : Y \to [0,\infty)$
is a \emph{$p$-norm} if $\| f + g \|^p \leq \| f \|^p + \|g \|^p$.
If $\|\cdot\| : Y \to [0,\infty)$ satisfies the $p$-norm property,
then it is a quasi-norm with triangle constant $C = 2^{\frac{1}{p} - 1}$;
see, e.g.,\, \cite[Exercise~1.1.4]{GrafakosClassicalFourier}.

By the Aoki-Rolewicz theorem (see, e.g., the proof of \cite[Chapter~2, Theorem~1.1]{devore1993constructive}),
given a quasi-norm $\| \cdot \| : Y \to [0,\infty)$, there exists some $p \in (0,1]$ such that
\begin{equation}
  \normm{f}
  := \inf
       \bigg\{
         \bigg(
           \sum_{i = 1}^n \| f_i \|^p
         \bigg)^{\frac{1}{p}}
         \; : \;
         n \in \mathbb{N},
         f = \sum_{i = 1}^n f_i, \;
         f_1, ..., f_n \in Y
       \bigg\}
  \label{eq:AokiRolewiczNorm}
\end{equation}
defines a $p$-norm on $Y$ which is equivalent to $\| \cdot \|$,
i.e., $\| \cdot \| \asymp \normm{\cdot}$.
A vector space $Y$ with quasi-norm $\| \cdot \|$ is a \emph{quasi-Banach space} if it is complete
with respect to the metric $d(f,g) = \normm{f-g}^p$, where $\normm{\cdot}$ is any $p$-norm
(for some $p \in (0,1]$) equivalent to $\| \cdot \|$.

Let $L^0 (G)$ be the space of ($\mu_G$-equivalence classes of) measurable functions $f : G \to \mathbb{C}$.
A \emph{quasi-Banach function space} $(Y, \| \cdot \|_Y)$ is a quasi-Banach space
satisfying $Y \subseteq L^0 (G)$.%
\footnote{The definition of a quasi-Banach function space is not consistent throughout the literature;
see \cite{lorist2023banach} for a discussion.
The definition used in the present paper is taken from \cite{FelixPhD}
and is inspired by the definition of a Banach function space in \cite{ZaanenIntegration}.
This definition contains the minimal assumptions needed for the present paper.}
It is called \emph{solid} if for each measurable $f : G \to \mathbb{C}$
satisfying $|f| \leq |g|$ $\mu_G$-a.e.\ for some $g \in Y$,
it follows that $f \in Y$ with $\| f \|_Y \leq \|g \|_Y$.

It is not difficult to see (cf. \cite[Corollary~2.2.12]{FelixPhD} for a proof)
that if $(Y, \| \cdot \|_Y)$ is a solid quasi-Banach function space,
then the $p$-norm $\normm{\cdot}$ defined in \Cref{eq:AokiRolewiczNorm} is solid as well.
Hence, we can (and will) always assume that solid quasi-Banach function spaces
are equipped with a solid $p$-norm, for some $p \in (0,1]$ depending on $Y$.

The space $Y^{\vee} = \{ F^{\vee} \colon F \in Y \}$ associated with a quasi-normed space $Y \subseteq L^0(G)$
is equipped with the quasi-norm $\| F \|_{Y^{\vee}} = \| F^{\vee} \|_Y$.

\section{Discrete sets}
\label{sec:discrete}

Let $\Lambda = (\lambda_i)_{i \in I}$ be a countable family in $G$.
Then $\Lambda$ is called \emph{relatively separated} if
\[
  \rel(\Lambda)
  := \sup_{x \in G}
       \cardinality{ \{ i \in I \colon \lambda_i \in x Q \}}
   = \sup_{x \in G}
       \sum_{i \in I}
         \indicator_{\lambda_i Q} (x)
  < \infty.
\]
Let $U \subseteq G$ be a relatively compact unit neighborhood.
The family $\Lambda$ is said to be \emph{$U$-dense} if $G = \bigcup_{i \in I} \lambda_i U$,
and is called \emph{$U$-separated} if $\lambda_i U \cap \lambda_{i'} U = \emptyset$
for all $i, i' \in I$ with $i \neq i'$.
If $\Lambda$ is separated, then it is also relatively separated.
A family $\Lambda$ is called \emph{relatively dense} if it is $V$-dense for some unit neighborhood $V$.
A relatively separated and $U$-dense family exists for any chosen $U$, cf.\ \cite{AnkerWellSpreadExistence}.

For a measurable unit neighborhood $U$, and any countable, $U$-dense family
$\Lambda = (\lambda_i)_{i \in I}$ in $G$, there exists a family $(U_i)_{i \in I}$
of measurable sets $U_i \subseteq \lambda_i U$ satisfying $G = \bigcupdot_{i \in I} U_i$;
see, e.g., \cite[Lemma~2.1]{MoleculePaper}.
Any such family $(U_i)_{i \in I}$ is called a \emph{disjoint cover} associated with $\Lambda$ and $U$.

\section{Sequence spaces}
\label{sub:SequenceSpaces}

Let $Y$ be a solid quasi-Banach function space on $G$ with $p$-norm $\| \cdot \|_Y$.
For a relatively separated family $\Lambda$ in $G$,
the \emph{discrete sequence space} associated with $Y$ and $\Lambda$ is defined as
\[
  Y_d (\Lambda)
  := Y_d (\Lambda, Q)
  := \bigg\{
       c = (c_{i})_{i \in I} \in \mathbb{C}^{I}
       :
       \sum_{i \in I} |c_{i}| \indicator_{\lambda_i Q} \in Y
     \bigg\}
\]
and equipped with the $p$-norm
\begin{equation}
  \| c \|_{Y_d (\Lambda)}
  := \bigg\|
       \sum_{i \in I}
         |c_{i} | \indicator_{\lambda_i Q}
     \bigg\|_{Y},
  \quad c = (c_{i})_{i \in I} \in Y_d (\Lambda).
  \label{eq:SequenceSpaceNormDefinition}
\end{equation}
The space $Y_d (\Lambda)$ is a quasi-Banach space.
If $Y$ is right-invariant, then $Y_d (\Lambda, Q)$ is independent
of the choice of the neighborhood $Q$, with equivalent quasi-norms for different choices.
In fact, the implied constant of the norm equivalence for different choice of $Q$
is independent of $\Lambda$.
To see this, note that if $Q'$ is a relatively compact unit neighborhood, then
$Q \subseteq \bigcup_{j=1}^n Q' x_j^{-1}$ for a certain $n \in \N$ and $x_1,\dots,x_n \in G$.
Hence, for any family $\Lambda = (\lambda_i)_{i \in I} \subseteq G$,
\begin{equation}
\begin{split}
  \bigg\|
    \sum_{i \in I}
      |c_{i} | \indicator_{\lambda_i Q}
  \bigg\|_{Y}
  & \leq \bigg\|
           \sum_{i \in I}
             |c_{i} | \sum_{j=1}^n \indicator_{\lambda_i Q' x_j^{-1}}
         \bigg\|_{Y}
    \lesssim_{n,Y} \sum_{j=1}^n
                     \bigg\|
                       R_{x_j}
                       \sum_{i \in I}
                         |c_{i} | \indicator_{\lambda_i Q'}
                     \bigg\|_{Y} \\
  & \lesssim_{n,Y} \bigg(
                     \sum_{j=1}^n
                       \| R_{x_j} \|_{Y \to Y}
                   \bigg)
                   \bigg\|
                     \sum_{i \in I}
                       |c_i| \, \indicator_{\lambda_i Q'}
                   \bigg\|_Y
    ,
\end{split}
\label{eq:SequenceSpaceIndependentOfBaseSet}
\end{equation}
where the implied constants only depend on $Q,Q',Y$ and are independent of $\Lambda$.

If $Y = L^p_v (G)$ for $p \in (0, \infty]$ and a (measurable) weight $v : G \to (0,\infty)$
satisfying $v(xy) \leq w(x) v(y)$ and $v(x y) \leq v(x) w(y)$ for a submultiplicative,
measurable $w : G \to (0, \infty)$, then
$Y_d (\Lambda) = \ell^p_u (I)$, where $u(i) := v(\lambda_i)$ for $i \in I$.
The space of finite sequences on $\Lambda$ will be denoted by $c_{00} (\Lambda)$.

For further properties and proofs, cf.\ \cite[Section~2]{rauhut2007wiener}
and \cite[Section~2.3.2]{FelixPhD}.

\section{Cocycles and twisted convolution}
\label{sec:cocycle_twisted}

A \emph{cocycle} or \emph{multiplier} on $G$ is a Borel measurable function
$\sigma : G \times G \to \mathbb{T}$ satisfying the properties
\begin{enumerate}[label=(\arabic*)]
 \item For all $x,y, z \in G$, it holds $\sigma(x,yz) \sigma(y,z) = \sigma(xy, z) \sigma(x,y)$;
 \item For any $x \in G$, it holds $\sigma(x, e_G) = \sigma(e_G, x) = 1$.
\end{enumerate}

Given a cocycle $\sigma$, the associated \emph{twisted translation operators}
$\twisttranslationL{x}$ and $\twisttranslationR{x}$, where $x \in G$,
act on a measurable function $F : G \to \mathbb{C}$ as
\begin{align} \label{eq:twisted_translation}
 (\twisttranslationL{x} F)(y) = \sigma(x, x^{-1} y) F(x^{-1} y)
 \quad \text{and} \quad
 (\twisttranslationR{x} F)(y) = \overline{\sigma(y, x)} F(y x),
 \quad y \in G.
\end{align}
If $\sigma \equiv 1$, then $\twisttranslationL{x}$ (resp.\ $\twisttranslationR{x}$)
will simply be denoted by $\translationL{x}$ (resp.\ $\translationR{x}$)
and called a translation operator.
The \emph{twisted convolution} of two measurable functions $F_1, F_2 : G \to \mathbb{C}$ is defined by
\begin{align} \label{eq:twisted_convolution}
  F_1 \ast_{\sigma} F_2 (x)
  = \int_G
      F_1(y) [\twisttranslationL{y} F_2] (x)
    \; d\mu_G (y)
  = \int_G
      F_1(y) \sigma(y, y^{-1} x) F_2 (y^{-1} x)
    \; d\mu_G (y)
\end{align}
whenever the integral converges.
As for twisted translation, if $\sigma \equiv 1$,
the ordinary convolution product will be denoted by $\ast$.
Note that
\[
  |F_1 \ast_{\sigma} F_2 (x)| \leq (|F_1| \ast |F_2|) (x)
\]
for all $x \in G$.

\section{Local maximal functions}
\label{sec:AmalgamPrelims}

For a measurable (resp.\ continuous) function  $F : G \to \mathbb{C}$,
the \emph{left} and \emph{right (local) maximal functions} defined by
\begin{align} \label{eq:maximal}
  \maxL F (x) = \esssup_{y \in xQ} | F(y)|
  \quad \text{resp.} \quad
  \maxR F(x) = \esssup_{y \in Qx} |F(y)|,
\end{align}
are measurable (resp.\ continuous) on $G$. The maximal functions satisfy the commutation relations
\begin{align} \label{max_commute}
  \maxL [\translationL{x} F] = \translationL{x} [\maxL F]
  \qquad \text{and} \qquad
  \maxR [\translationR{x} F] = \translationR{x} [\maxR F]
\end{align}
for arbitrary $x \in G$.

The notation $\maxL_{Q'}$ (resp.\ $\maxR_{Q'}$) will be used if maximal functions
are defined as in \eqref{eq:maximal}, but relative to a (relatively compact) unit neighborhood $Q'$
that is possibly different from the canonical neighborhood $Q$ (cf.\ Section~\ref{sec:notation}).

By symmetry of $Q$, the left and right maximal functions are related
by $(\maxL F)^{\vee} = \maxR F^{\vee}$.
For any measurable functions $F_1,F_2 : G \to \CC$ for which $F_1 \ast_\sigma F_2$ is
(almost everywhere) well-defined, the estimates
\begin{equation}
  \maxL (F_1 \ast_\sigma F_2) (x) \leq (|F_1| \ast \maxL F_2 ) (x)
  \quad \text{and} \quad
  \maxR (F_1 \ast_\sigma F_2) (x) \leq (\maxR F_1 \ast |F_2|)(x)
  \label{eq:ConvolutionMaximalFunctionEstimate}
\end{equation}
hold for all $x \in G$.

For a solid quasi-Banach function space $Y$ that is invariant under left- and right translation
and with $p$-norm $\| \cdot \|_Y$, the associated left and right \emph{Wiener amalgam spaces}
$\wienerL{Y}$ and $\wienerR{Y}$ are defined by
\begin{align*}
  \wienerL{Y}  = \big\{ F \in L^{\infty}_{\loc} (G) : \maxL F \in Y \}
  \quad \text{and} \quad
  \wienerR{Y}  = \big\{ F \in L^{\infty}_{\loc} (G) : \maxR F \in Y \},
\end{align*}
respectively.
These spaces are equipped with the canonical $p$-norms $\| F \|_{\wienerL{Y}} = \| \maxL F \|_Y$
and $\| F \|_{\wienerR{Y}} = \| \maxR F \|_{Y}$.
Note that $\wienerR{Y} = [\wienerL{Y^{\vee}}]^{\vee}$,
as can be deduced from the identity $(\maxL F)^{\vee} = \maxR (F^{\vee})$.
The spaces $\wienerL{Y}$ and $\wienerR{Y}$ are complete and continuously embedded into $Y$;
in notation, $\wienerL{Y}, \wienerR{Y} \hookrightarrow Y$.
In addition, the spaces $\wienerL{Y}$ and $\wienerR{Y}$ are independent
of the choice of defining neighborhood $Q$, with equivalent norms for different choices.
Furthermore, $\wienerL{Y}, \wienerR{Y}$ are invariant under left- and right-translations,
and satisfy the following estimates
\begin{equation}
  \| \translationL{x} \|_{\wienerL{Y} \to \wienerL{Y}} \leq \| \translationL{x} \|_{Y \to Y}
  \qquad \text{and} \qquad
  \| \translationR{x} \|_{\wienerR{Y} \to \wienerR{Y}} \leq \| \translationR{x} \|_{Y \to Y}
  .
  \label{eq:WienerAmalgamTranslationNormEstimate}
\end{equation}
These estimates are easy consequences of \Cref{max_commute}.

The following simple pointwise estimates for continuous functions will be used repeatedly.
For a proof, see, e.g., \cite[Lemma~2.4]{MoleculePaper}.

\begin{lemma}\label{lem:StandardShiftedSeriesEstimates}
Let $\Lambda = (\lambda_i)_{i \in I} \subseteq \group$ be relatively separated.
  If $F_1, F_2 : \group \to [0,\infty)$ are continuous functions, then
  \begin{align}
    \sum_{i \in I}
      F_1 (\lambda_i^{-1} x) F_2 (y^{-1} \lambda_i)
    & \leq \frac{\rel(\Lambda)}{\haarMeasure(Q)}
            \bigl(\maxL F_2 \ast \maxR F_1 \bigr) (y^{-1} x)
    \label{eq:ShiftedSeriesEstimateTwoTerms}
    \\
    \sum_{i \in I}
      F_1 (y^{-1} \lambda_i)
    & \leq \frac{\rel(\Lambda)}{\haarMeasure(Q)} \| F_1 \|_{\wienerL{L^1}}
    \label{eq:ShiftedSeriesEstimateOneTerm}
  \end{align}
  for all $x,y \in G$.

  In addition, if $F_1 \in L^1(G)$ is continuous and satisfies $F_1^{\vee} \in \wienerL{L^1}$,
  then the mapping
  \[
    D_{F_1, \Lambda} : \quad
    \ell^2 (\Lambda) \to L^2 (G), \quad
    (c_i)_{i \in I} \mapsto \sum_{i \in I} c_i L_{\lambda_i} F_1
  \]
  is well-defined and bounded, with absolute convergence of the defining series a.e.\ on $G$.
  Its operator norm satisfies
  \(
    \|D_{F_1, \Lambda} \|^2_{\ell^2 \to L^2}
    \leq \frac{\rel(\Lambda)}{\mu_G (Q)} \| F_1 \|_{L^1} \|F_1^{\vee} \|_{\wienerL{L^1}}
    .
  \)
\end{lemma}

The \emph{(two-sided) maximal function} is defined by $\maxSt := \maxL \maxR = \maxR \maxL$.
The associated two-sided Wiener amalgam space is the subspace of $\wienerL{Y} \cap \wienerR{Y}$ defined by
\[
  \wienerSt{Y} := \big\{ F \in L_{\loc}^{\infty} (G) : \maxSt F \in Y \big\}.
\]
The (closed) subspace of $\wienerSt{Y}$ consisting of continuous functions is denoted by $\wienerStC{Y}$.
Both $\wienerSt{Y}$ and $\wienerStC{Y}$ are quasi-Banach function spaces
when equipped with the $p$-norm $\| F \|_{\wienerSt{Y}} = \| \maxSt F\|_{Y}$.
These properties follow from $\wienerSt{Y} = \wienerR {\wienerL {Y}}$. In addition, note that
\[
  [\wienerSt{Y}]^{\vee}
  = \wienerL{[\wienerL{Y}]^{\vee}}
  = \wienerL{\wienerR{Y^{\vee}}}
  = \wienerSt{Y^{\vee}}
  .
\]
For further properties and proofs, see \cite[Sections 2 and 3]{rauhut2007wiener}
and \cite[Section~2.3]{FelixPhD}.

\chapter{Convolution relations for Wiener amalgam spaces}
\label{sec:WienerAmalgam}

The purpose of this chapter is to provide several convolution relations and simple embeddings
of Wiener amalgam spaces into Lebesgue spaces.
These results will play an essential role in the development
of coorbit theory in the subsequent sections.

\section{Embeddings into Lebesgue spaces}
The following conditions on a weight function will often be assumed in the sequel.

\begin{definition}\label{def:PWeight}
  Let $p \in (0,1]$.
  A function $w : \group \to (0,\infty)$ will be called a \emph{$p$-weight} if it satisfies
  the following conditions:
  \begin{enumerate}[label=(w\arabic*)]
    \item \label{enu:WeightBoundedBelow}
          $w$ is measurable and satisfies $w \geq 1$,
    \item \label{enu:WeightSubmultiplicative}
          $w$ is \emph{submultiplicative}, i.e.,
          $w(x y) \leq w(x) w(y)$ for all $x,y \in \group$,
    \item \label{enu:WeightPSymmetric}
          $w(x) = w(x^{-1})  \Delta^{1/p}(x^{-1})$ for all $x \in \group$.
  \end{enumerate}
  A $p$-weight is called a \emph{control weight} for a quasi-Banach function space $(Y, \| \cdot \|_Y)$
  with $p$-norm $\| \cdot \|_Y$ if $Y$ is \emph{translation-invariant}
  (i.e., invariant under left- and right translations) and $w$ satisfies
  \begin{enumerate}[resume,label=(w\arabic*)]
      \item \label{enu:ControlWeightCondition}
            $\|R_y \|_{Y \to Y} \leq w(y)$ for all $y \in G$.
  \end{enumerate}
  A control weight $w$ for $Y$ satisfying additionally
  \begin{enumerate}[resume,label=(w\arabic*)]
    \item \label{enu:StrongControlWeightCondition}
          $\|L_{y^{-1}} \|_{Y \to Y} \leq w(y)$ for all $y \in G$
  \end{enumerate}
  is called a \emph{strong control weight} for $Y$.
  \end{definition}

\begin{remark}\label{rem:SubmultiplicativeWeightsLocallyBounded}
  A measurable, submultiplicative weight is automatically locally bounded;
  see, e.g., \cite[Lemma~1.3.3]{kaniuth2009course}.
\end{remark}

\begin{lemma}\label{lem:AmalgamWeightedLInftyEmbedding}
  Let $r \in (0,\infty]$, let $w : \group \to (0,\infty)$ be measurable and submultiplicative,
  and let $Y$ be a solid, translation-invariant quasi-Banach function space on $\group$.
  Then the following hold:
  \begin{enumerate}[label=(\roman*)]
    \item If $v : \group \to (0,\infty)$ is measurable and satisfies
          $\| \translationL{y^{-1}} \|_{Y \to Y} \leq v(y)$ for all $y \in \group$,
          then $\wienerL{Y} \hookrightarrow L_{1/v}^\infty$.

    \item $\wienerL{L_w^r} \hookrightarrow L_{w}^s$ for all $s \in [r,\infty]$,

    \item If $p \in (0,1]$ and $w$ is a $p$-weight, then
          $\wienerR{L_w^p} \hookrightarrow [L_{w}^s]^{\vee}$ for all $s \in [p,\infty]$.
  \end{enumerate}
\end{lemma}

\begin{remark}
  In fact, the lemma implies something slightly stronger:
  If $F \in \wienerL{L_w^r}$, then $\maxL F \in \wienerL{L_w^r} \hookrightarrow L_w^s$
  and hence $F \in \wienerL{L_w^s}$.
  This argument easily implies that $\wienerL{L_w^r} \hookrightarrow \wienerL{L_w^s}$
  for $s \in [r,\infty]$.
\end{remark}

\begin{proof}
  For proving assertions (i)-(iii), it will be used that
  for a measurable $F : \group \to \CC$, there exists a null-set $N \subseteq \group$
  such that for every $x \in \group \setminus N$, it holds that
  \begin{equation}
    \indicator_{x Q^{-1}}  |F(x)|
    \leq \maxL F
    \quad \text{a.e.}
    \label{eq:MaximalFunctionPointwiseDomination}
  \end{equation}

  To show this, let
  \(
    \Omega
    := \big\{
         (x,z) \in \group \times \group
         \colon
         \indicator_{x Q^{-1}} (z)  |F(x)| > \maxL F (z)
       \big\} .
  \)
  Since $\maxL F$ is measurable (cf.\ Section~\ref{sec:AmalgamPrelims}), it is easy to see
  that $\Omega$ belongs to the Borel $\sigma$-algebra on $\group \times \group$.
  Since $G$ is second countable, the Borel $\sigma$-algebra on $G \times G$ coincides
  with the product $\sigma$-algebra on $G \times G$; see e.g.\ \cite[Theorem~7.20]{FollandRA}.
  Let $\Omega_x = \{ z \in \group \colon (x,z) \in \Omega \}$
  and $\Omega^{z} = \{ x \in \group \colon (x,z) \in \Omega \}$.
  Then Tonelli's theorem shows that
  \[
    \int_{\group} \haarMeasure(\Omega_x) \; \dd{x}
    = \int_{\group} \haarMeasure(\Omega^z) \; \dd{z}.
  \]
  Therefore, it suffices to show that $\Omega^z$ is a null-set for arbitrary $z \in \group$.
  For this, let $z \in \group$.
  By definition of $\maxL$, there exists a null-set $N_z \subseteq Q$ such that
  $\maxL F(z) \geq |F(z q)|$ for all $q \in Q \setminus N_z$.
  For $x \in z (\group \setminus N_z)$ there are two cases.
  First, if $x \notin z Q$, then $\indicator_{x Q^{-1}}(z) = 0$ and hence trivially
  $\indicator_{x Q^{-1}}(z)  |F(x)| \leq \maxL F (z)$.
  Second, if $x \in z Q$, then $q := z^{-1} x \in Q \setminus N_z$
  and hence $\maxL F (z) \geq |F(z q)| = |F(x)| \geq |F(x)| \indicator_{x Q^{-1}}(z)$.
  Overall, this shows that $\Omega^z \subseteq z N_z$ is a null-set,
  and establishes inequality \eqref{eq:MaximalFunctionPointwiseDomination}.

  (i)
  Define $C := 1 / \| \indicator_{Q^{-1}} \|_Y$, with the understanding that $C = 0$
  if $\indicator_{Q^{-1}} \notin Y$.
  Let $F \in \wienerL{Y}$.
  It will be shown that $\| F \|_{L_{1/v}^\infty} \leq C \, \| F \|_{\wienerL{Y}}$.
  For this, let $N = N(F) \subseteq \group$ be a null set
  such that \eqref{eq:MaximalFunctionPointwiseDomination} holds for all $x \in \group \setminus N$.
  For $x \in \group \setminus N$, it follows then by the solidity of $Y$ that
  \[
    |F(x)|  \| \indicator_{x Q^{-1}} \|_Y
    \leq \| \maxL F \|_Y
    =    \| F \|_{\wienerL{Y}}
    <    \infty .
  \]
  If $|F(x)| \neq 0$, then this implies $\indicator_{x Q^{-1}} \in Y$, and
  \[
    \| \indicator_{Q^{-1}} \|_{Y}
    = \| \translationL{x^{-1}} \indicator_{x Q^{-1}} \|_Y
    \leq v(x)  \| \indicator_{x Q^{-1}} \|_Y,
  \]
  so that
  \(
    |F(x)| / v(x)
    \leq C  \| F \|_{\wienerL{Y}} .
  \)
  If $|F(x)| = 0$, then this trivially holds.
  In conclusion, it follows that $|F(x)| / v(x) \leq C  \| F \|_{\wienerL{Y}}$
  for all $x \in \group \setminus N$.

  (ii)
  It will be shown that $\wienerL{L_w^r} \hookrightarrow L_w^\infty$.
  Since $\wienerL{L_w^r} \hookrightarrow L_w^r$ (cf.\ Section~\ref{sec:AmalgamPrelims}),
  this will then imply that $\wienerL{L_w^r} \hookrightarrow L_w^s$ for all $s \in [r,\infty]$.
  Let $F \in \wienerL{L_w^r}$.
  Choose a null-set $N = N(F) \subseteq \group$ such that
  \Cref{eq:MaximalFunctionPointwiseDomination} holds for all $x \in \group \setminus N$.
  Since $L_w^r$ is solid, it follows that, for all $x \in \group \setminus N$,
  \[
    |F(x)|  \| \indicator_{x Q^{-1}} \|_{L_w^r}
    \leq \| \maxL F \|_{L_w^r}
    = \| F \|_{\wienerL{L_w^r}}
    < \infty
    .
  \]
  There exists  $C = C(Q,w) > 0$ such that  $w(q) \leq C$
  for all $q \in Q$ (cf.\ \Cref{rem:SubmultiplicativeWeightsLocallyBounded}).
  Hence, if $z \in Q^{-1}$, then $w(x) \leq w(x z) w(z^{-1}) \leq C  w(x z)$ for all $x \in G$.
  Therefore,
  \[
    \| \indicator_{x Q^{-1}} \|_{L_w^r}^r
    = \int_{x Q^{-1}}
        [w(y)]^r
      \dd{y}
    = \int_{Q^{-1}}
        [w(x z)]^r
      \dd{z}
    \geq C^{-r}  [w(x)]^r  \mu_{\group} (Q^{-1})
  \]
  for all $x \in \group$.
  Overall, this shows  $|F| \cdot  w \lesssim \| F \|_{\wienerL{L_w^r}}$
  a.e.\ on $G$, as required.

 (iii)
  Condition \ref{enu:WeightPSymmetric} combined with the identity
  \[
    \int_{\group} F(x) \, d \mu_{\group}(x)
    = \int_{\group} F(x^{-1}) \Delta(x^{-1}) \, d \mu_{\group}(x)
  \]
  easily implies that $[L_w^p]^{\vee} = L_w^p$,
  with identical quasi-norms.
  Hence,
  \[
    \wienerR{L_w^p}
    = [\wienerL{(L_w^p)^{\vee}}]^{\vee}
    = [\wienerL{L_w^p}]^{\vee}
    ,
  \]
  with identical quasi-norms (cf.\ Section~\ref{sec:AmalgamPrelims}).
  Therefore, Part (ii) implies for all $s \in [p,\infty]$ that
  $\wienerR{L_w^p} = [\wienerL{L_w^p}]^{\vee} \hookrightarrow [L_w^s]^{\vee}$.
\end{proof}

\section{Convolution relations} \label{sec:convolution_relations}

This subsection is devoted to convolution relations for Wiener amalgam spaces.
The following compatibility condition between a quasi-Banach function space
and a weight function will play an important role in the sequel.

\begin{definition} \label{def:compatible}
  Let $w : G \to [1,\infty)$ be a $p$-weight for some $p \in (0,1]$
  and let $Y$ be a solid quasi-Banach function space $Y$ on $\group$ with $p$-norm $\| \cdot \|_Y$.
  The space $Y$ is said to be \emph{$L_w^p$-compatible} if
  \begin{enumerate}[label=(c\arabic*)]
      \item \label{enu:CompatibleTranslationInvariant}
            $Y$ is translation-invariant.

      \item \label{enu:CompatibleConvolution}
            The convolution relation
            $\wienerL{Y} \ast \wienerSt{L^p_w} \hookrightarrow \wienerL{Y}$ holds.

      \item \label{enu:CompatibleLInfty}
            The embedding $\wienerL{Y} \hookrightarrow L_{1/w}^{\infty} (G)$ holds.
  \end{enumerate}
\end{definition}

\begin{remark}
  In \Cref{cor:UserFriendlyConvolutionBounds}, it will be shown that $Y$ is $L_w^p$-compatible
  if $w$ is a strong control weight for $Y$.
  The reason for introducing the concept of $L_w^p$-compatibility is that there are some
  (important) cases in which $L_w^p$-compatibility holds even if $w$ is not
  a strong control weight for $Y$.
  A case in point are the spaces $Y = L_w^p$, $Y = \wienerL{L_w^p}$, and $Y = \wienerR{L_w^p}$,
  as well as $Y = \wienerSt{L_w^p}$, which are always $L_w^p$-compatible if $w$ is a $p$-weight;
  see \Cref{cor:Lpw_nostrong}.
\end{remark}

The following result provides a general convolution relation for Amalgam spaces.
The most important consequences of this result
are stated in \Cref{cor:UserFriendlyConvolutionBounds}.

\begin{theorem}\label{thm:AmalgamMainConvolutionRelation}
  Let $p \in (0,1]$ and let $w : \group \to (0,\infty)$ be a $p$-weight.
  Let $Y$ be a solid, translation-invariant quasi-Banach function space
  and assume that $w$ is a control weight for $Y$.
  Suppose there exists a submultiplicative
  measurable weight $v : \group \to [1,\infty)$ satisfying
  $\wienerL{Y} \hookrightarrow L_{1/v}^\infty$ and $\wienerR{L_w^p} \hookrightarrow [L_v^1]^{\vee}$.
  Then the convolution relation
  \[
    \wienerL{Y} \ast \wienerR{L_w^p} \hookrightarrow Y
  \]
  holds.
  More precisely, there exists a constant $C = C(p,w,Q) > 0$ such that
  \[
    \big\| F_1 \ast F_2 \big\|_{Y}
    \leq \big\| |F_1| \ast |F_2| \big\|_{Y}
    \leq C  \| F_1 \|_{\wienerL{Y}}  \| F_2 \|_{\wienerR{L_w^p}}
  \]
  for all $F_1 \in \wienerL{Y}$ and $F_2 \in \wienerR{L_w^p}$.
  In addition, $|F_1| \ast |F_2| (x) < \infty$ for \emph{all} $x \in \group$
  and $F_1 \ast F_2 : \group \to \CC$ is continuous.
\end{theorem}

\begin{proof} The proof will be split into three steps.

  \medskip{}

  \textbf{Step~1.}
  Since $F_1 \ast F_2 (x) = \big( F_1, \translationL{x} [F_2^{\vee}] \big)_{L_{1/v}^\infty, L_v^1}$,
  it follows that $|F_1| \ast |F_2| (x) < \infty$ for all $x \in \group$.
  Furthermore, the map $x \mapsto \translationL{x} [F_2^{\vee}]$ is continuous from $G$
  into $L_v^1$, see, e.g., \cite[Lemma~1.3.6]{kaniuth2009course}
  or \mbox{\cite[Proposition~3.7.6]{ReiterClassicalHarmonicAnalysis}}.
  Hence, $F_1 \ast F_2 : \group \to \CC$ is continuous as well.

  \medskip{}

  \textbf{Step~2.}
  In this step, it will be shown that there exists a countable family $(x_i)_{i \in I} \subseteq \group$
  and a measurable partition of unity $(\varphi_i)_{i \in I}$ on $\group$
  with $ \varphi_i^{-1} (\CC \setminus \{0\}) \subseteq x_i Q$ and a constant $C = C(p,Q,w) > 0$ such that
  \begin{equation}
    \Big\|
      \Big(
        w(x_i)  \| F \cdot \varphi_i \|_{L^\infty}
      \Big)_{i \in I}
    \Big\|_{\ell^p}
    \leq C  \| F \|_{\wienerL{L_w^p}}
    \label{eq:WienerAmalgamDiscreteDescription}
  \end{equation}
  for all $F \in \wienerL{L_w^p}$.

  For constructing $(\varphi_i)_{i \in I}$, an application of \cite[Lemma~1]{AnkerWellSpreadExistence} yields a subset
  $X \subseteq \group$ and some $N \in \N$ such that $\group = \bigcup_{x \in X} x Q$
  and such that each $x \in \group$ belongs to at most $N$ of the sets $x Q$ for $x \in X$.
  Since $\group$ is second-countable and hence $\sigma$-compact,
  it is possible to extract a countable subset of $X$
  that still satisfies these conditions, so that it may be assumed that $X$ is countable.
  Thus, for a suitable $M \in \N \cup \{ \infty \}$ and for $I = \{ i \in \N \colon i < M \}$,
  we can write $X = \{ x_i \colon i \in I \}$ with $x_i \neq x_j$ for $i \neq j$.
  Define $\Omega_i := x_i Q \setminus \bigcup_{\ell = 1}^{i-1} x_\ell Q$ inductively
  and set $\varphi_i := \indicator_{\Omega_i}$.
  Then $ \varphi_i^{-1} (\CC \setminus \{0\}) = \Omega_i \subseteq x_i Q$ and $\sum_{i \in I} \varphi_i \equiv 1$
  on $\group$.

  For showing \eqref{eq:WienerAmalgamDiscreteDescription},
  define $Q' := Q^{-1} Q$ and let $F \in \wienerL{L_w^p}$.
  Note that if $x \in x_i Q$, then  $\varphi_i^{-1} (\CC \setminus \{0\}) \subseteq x_i Q \subseteq x Q^{-1} Q = x Q'$.
  This implies that
  $\| F \cdot \varphi_i \|_{L^\infty} \leq  \maxL_{Q'} F (x). $
  On the other hand, we have $w \lesssim 1$ everywhere on $Q$
  (cf.\ \Cref{rem:SubmultiplicativeWeightsLocallyBounded}),
  and thus there exists $C_1 > 0$ such that $w(x_i)  \leq w(x) w(x^{-1} x_i) \leq C_1  w(x)$.
  Overall, these considerations show that
  \[
    \big[ w(x_i)  \| F \cdot \varphi_i \|_{L^\infty} \big]^p \cdot \indicator_{x_i Q}
    \lesssim \big[ w \cdot \maxL_{Q'} F  \big]^p \quad \text{for all} \quad i \in I.
  \]
  By the choice of $\{ x_i \colon i \in I \} = X$,
  it holds that $\sum_{i \in I} \indicator_{x_i Q} \leq N$, and hence
  \begin{align*}
    \sum_{i \in I}
      \big[ w(x_i)  \| F \cdot \varphi_i \|_{L^\infty} \big]^p \cdot \indicator_{x_i Q}
     \lesssim \big[ w \cdot \maxL_{Q'} F  \big]^p
            \sum_{i \in I}
                   \indicator_{x_i Q}
     \lesssim
           \ \big[ w \cdot \maxL_{Q'} F  \big]^p .
  \end{align*}
  Integrating this estimate over $\group$ gives
  \begin{align*}
    \sum_{i \in I}
      \big[ w(x_i)  \| F \cdot \varphi_i \|_{L^\infty} \big]^p
    & =    \frac{1}{\mu_G(Q)}
           \sum_{i \in I}
             \int_{\group}
               \big[ w(x_i)  \| F \cdot \varphi_i \|_{L^\infty} \big]^p
                \indicator_{x_i Q} (x)
             \dd{x} \\
    & \lesssim  \| \maxL_{Q'} F \|_{L_w^p}^p
      \lesssim \| F \|_{\wienerL{L_w^p}}^p,
  \end{align*}
  where the last estimate follows from the independence of
  $\wienerL{L_w^p} = W^L_Q (L_w^p)$ from the choice of the set $Q$ (cf.\ \Cref{sec:AmalgamPrelims}.)
  The preceding estimate easily shows that \Cref{eq:WienerAmalgamDiscreteDescription} holds.

  \medskip{}

  \textbf{Step~3.}
  With notation as in Step~2, set $\psi_i := \varphi_i^{\vee}$
  and $F_2^{(i)} := F_2 \cdot \psi_i$ for $i \in I$.
  Note that $(F_2^{(i)})^{\vee} = F_2^{\vee} \cdot \varphi_i$ and $\psi_i^{-1} (\CC \setminus \{0\}) \subseteq Q^{-1} x_i^{-1}$.
  Furthermore, note that $L_w^p = [L_w^p]^{\vee}$,
  which follows from condition \ref{enu:WeightPSymmetric} of \Cref{def:PWeight}.
  Hence, applying the estimate \eqref{eq:WienerAmalgamDiscreteDescription}
  to the function $F_2^{\vee} \in \wienerL{ [L_w^p]^{\vee} } = \wienerL{L_w^p}$ shows that
  \begin{equation}
    \begin{split}
      \Big\|
        \Big(
          w(x_i)  \| F_2^{(i)} \|_{L^\infty}
        \Big)_{i \in I}
      \Big\|_{\ell^p}
       \lesssim  \| F_2^{\vee} \|_{\wienerL{L_w^p}}
        =     \| F_2 \|_{[\wienerL{L_w^p}]^{\vee}}
        =    \| F_2 \|_{\wienerR{L_w^p}} .
    \end{split}
    \label{eq:ConvolutionProofDiscretePartsEstimate}
  \end{equation}

  If $i \in I$ and $F_2^{(i)}(y^{-1} x) \neq 0$, then
  $y^{-1} x \in (F_2^{(i)})^{-1}(\CC \setminus \{0\}) \subseteq Q^{-1} x_i^{-1}$ and hence $\indicator_{x x_i Q} (y) = 1$.
  Thus, $F_2^{(i)}(y^{-1} x) = F_2^{(i)}(y^{-1} x) \indicator_{x x_i Q} (y)$, and
  \begin{align*}
    \big(
      |F_1| \ast |F_2^{(i)}|
    \big) (x)
    & = \int_{\group}
          |F_1(y)|  \bigl|F_2^{(i)} (y^{-1} x)\bigr|
        \dd{y} \\
    & \leq \big\| F_2^{(i)} \big\|_{L^\infty}
           \| F_1 \cdot \indicator_{x x_i Q} \|_{L^\infty}
            \mu_G (x x_i Q) \\
    & =    \mu_G(Q)
            \big\| F_2^{(i)} \big\|_{L^\infty}
            (\maxL F_1) (x x_i) \\
    &  =    \mu_G(Q)
            \big\| F_2^{(i)} \big\|_{L^\infty}
            \translationR{x_i} \bigl[\maxL F_1\bigr] (x)
  \end{align*}
  for all $x \in \group$.
  By the solidity of $Y$ and since $\| \translationR{x_i} \|_{Y \to Y} \leq w(x_i)$
  by assumption (cf.\ condition \ref{enu:ControlWeightCondition} of \Cref{def:PWeight}), this implies
  \[
    \big\|\,
      |F_1| \ast \bigl|F_2^{(i)}\bigr|
    \,\big\|_{Y}
    \lesssim
          \big\| F_2^{(i)} \big\|_{L^\infty}
          \big\| \translationR{x_i} \big[ \maxL F_1 \big] \big\|_{Y}
    \leq  w(x_i)
          \big\| F_2^{(i)} \big\|_{L^\infty}
          \| F_1 \|_{\wienerL{Y}}, \quad i \in I.
  \]
  Using the estimate \eqref{eq:ConvolutionProofDiscretePartsEstimate}
  and \Cref{lem:QuasiBanachSolidAbsoluteConvergence}, it follows that
  $F := \sum_{i \in I} |F_1| \ast \bigl|F_2^{(i)}\bigr| \in Y$ and
  \[
    \| F \|_Y
    \lesssim
          \| F_1 \|_{\wienerL{Y}}
          \Big\|
                 \Big(
                   w(x_i)  \big\| F_2^{(i)} \big\|_{L^\infty}
                 \Big)_{i \in I}
               \Big\|_{\ell^p}
    \lesssim
          \| F_1 \|_{\wienerL{Y}}
          \| F_2 \|_{\wienerR{L_w^p}} .
  \]
  Since $F_2 = \sum_{i \in I} F_2^{(i)}$, an application of the monotone convergence theorem
  yields
  \[
    |F_1 \ast F_2|
    \leq |F_1| \ast |F_2|
    \leq \sum_{i \in I}
           |F_1| \ast \bigl|F_2^{(i)}\bigr|
    = F,
  \]
  which easily implies the claim.
\end{proof}

\begin{remark}
  The assumptions $\wienerL{Y} \hookrightarrow L_{1/v}^\infty$
  and $\wienerR{L_w^p} \hookrightarrow [L_v^1]^{\vee}$ in \Cref{thm:AmalgamMainConvolutionRelation}
  are not actually needed for deriving the convolution relation
  $\wienerL{Y} \ast \wienerR{L_w^p} \hookrightarrow Y$;
  these assumptions are only used to ensure that $F_1 \ast F_2$ is well-defined everywhere
  and continuous.
\end{remark}

As a consequence of Theorem~\ref{thm:AmalgamMainConvolutionRelation}, it follows,
in particular, that $Y$ is $L^p_w$-compatible whenever $w$ is a strong control weight for $Y$.
This sufficient condition for compatibility is particularly convenient for applications.
The precise statement is as follows.

\begin{corollary}\label{cor:UserFriendlyConvolutionBounds}
   Let $w : \group \to [1,\infty)$ be a $p$-weight for some $p \in (0,1]$.
  \begin{enumerate}[label=(\roman*)]
  \item The following convolution relations hold:
        \[
          \wienerL{L_w^p} \ast \wienerR{L_w^p} \hookrightarrow L_w^p
          \qquad \text{and} \qquad
          \wienerSt{L_w^p} \ast \wienerSt{L_w^p} \hookrightarrow \wienerStC{L_w^p}.
        \]

 \item If $w$ is a strong control weight for a solid, translation-invariant quasi-Banach
       function space $Y$ in the sense of \Cref{def:PWeight},
       then the following convolution relations hold
       \[
         \wienerL{Y} \ast \wienerR{L_w^p} \hookrightarrow Y
         \qquad \text{and} \qquad
         \wienerL{Y} \ast \wienerSt{L_w^p} \hookrightarrow \wienerL{Y} .
       \]
       In particular, $Y$ is $L^p_w$-compatible.
  \end{enumerate}
\end{corollary}

\begin{proof}
  (i)
  We apply \Cref{thm:AmalgamMainConvolutionRelation} with $Y = L_w^p$ and $v \equiv 1$.
  For this, note that applying the identity
  $\int \translationR{y} F \, d \mu_\group = \Delta(y^{-1}) \int F \, d \mu_\group$
  for measurable $F : G \to [0,\infty]$ implies for $f \in L_w^p$ that
  \begin{align*}
    \| \translationR{y} f \|_{L_w^p}^p
    & = \int_{\group}
          \bigl|w(x y y^{-1}) f(x y)\bigr|^p
        \dd {x}
    \leq \bigl[w(y^{-1})\bigr]^p
         \int_{\group}
           R_y \big[ |w \cdot f|^p \big]
         \, d \mu_{\group} \\
    & =    \bigl[
             w(y^{-1}) \, \Delta^{1/p}(y^{-1})
           \bigr]^p
            \| f \|_{L_w^p}^p ,
  \end{align*}
  and hence $\| \translationR{y} \|_{Y \to Y} \leq w(y^{-1})  \Delta^{1/p}(y^{-1})$.
  Since $w(y^{-1})  \Delta^{1/p}(y^{-1}) = w(y)$ by condition \ref{enu:WeightPSymmetric} of \Cref{def:PWeight},
  this shows that $w$ is a control weight for $Y$.
  Furthermore, \Cref{lem:AmalgamWeightedLInftyEmbedding} shows
  because of $w \geq 1 \equiv v = \frac{1}{v}$ and $p \leq 1$
  that $\wienerL{L_w^p} \hookrightarrow L_w^\infty \hookrightarrow L_v^\infty = L_{1/v}^\infty$.
  Similarly, \Cref{lem:AmalgamWeightedLInftyEmbedding}
  shows that $\wienerR{L_w^p} \hookrightarrow [L_w^1]^{\vee} \hookrightarrow [L_{v}^1]^{\vee}$.
  %In combination, this easily shows $\wienerR{L_w^p} \hookrightarrow L_v^1$.
  This shows that all assumptions of \Cref{thm:AmalgamMainConvolutionRelation} are satisfied,
  and thus $\wienerL{L_w^p} \ast \wienerR{L_w^p} \hookrightarrow L_w^p$.

  Finally, for $F_1,F_2 \in \wienerSt{L_w^p}$,
  an application of \Cref{eq:ConvolutionMaximalFunctionEstimate} yields
  \begin{align*}
    \big\| F_1 \ast F_2 \big\|_{\wienerSt{L_w^p}}
    & = \big\| \maxL \maxR (F_1 \ast F_2) \big\|_{L_w^p}
      \leq \big\| (\maxR F_1) \ast (\maxL F_2) \big\|_{L_w^p} \\
    & \lesssim \big\| \maxR F_1 \big\|_{\wienerL{L_w^p}}  \big\| \maxL F_2 \big\|_{\wienerR{L_w^p}}
      \leq \| F_1 \|_{\wienerSt{L_w^p}}  \| F_2 \|_{\wienerSt{L_w^p}} ,
  \end{align*}
  as asserted, where \Cref{thm:AmalgamMainConvolutionRelation} shows that $F_1 \ast F_2$
  is continuous for $F_1 \in \wienerSt{L_w^p} \hookrightarrow \wienerL{L_w^p}$
  and $F_2 \in \wienerSt{L_w^p} \hookrightarrow \wienerR{L_w^p}$.

  (ii)
  By assumption, $\| \translationL{y^{-1}} \|_{Y \to Y} \leq w(y)$ for all $y \in G$,
  and thus \Cref{lem:AmalgamWeightedLInftyEmbedding}
  shows that $\wienerL{Y} \hookrightarrow L_{1/w}^\infty$.
  Furthermore, \Cref{lem:AmalgamWeightedLInftyEmbedding} also shows
  that $\wienerR{L_w^p} \hookrightarrow [L_w^1]^{\vee}$.
  Thus, \Cref{thm:AmalgamMainConvolutionRelation} is applicable and shows
  $\wienerL{Y} \ast \wienerR{L_w^p} \hookrightarrow Y$.
  Combining this convolution relation with \Cref{eq:ConvolutionMaximalFunctionEstimate}
  and the solidity of $Y$, we see for $F_1 \in \wienerL{Y}$ and $F_2 \in \wienerSt{L_w^p}$ that
  \begin{align*}
    \| F_1 \ast F_2 \|_{\wienerL{Y}}
    & = \| \maxL (F_1 \ast F_2) \|_Y
      \leq \| |F_1| \ast \maxL F_2 \|_Y \\
    & \lesssim \| F_1 \|_{\wienerL{Y}}  \| \maxL F_2 \|_{\wienerR{L_w^p}}
      =        \| F_1 \|_{\wienerL{Y}}  \| F_2 \|_{\wienerSt{L_w^p}} ,
  \end{align*}
  completing the proof.
\end{proof}

\begin{corollary} \label{cor:Lpw_nostrong}
  Let $p \in (0,1]$ and let $w : G \to (0, \infty)$ be a $p$-weight.
  Then each of the spaces $L_w^p$, $\wienerL{L^p_w}, \wienerR{L^p_w}$,
  and $\wienerSt{L^p_w}$ is $L^p_w$-compatible.
\end{corollary}

\begin{proof}
  This easily follows by combining \Cref{cor:UserFriendlyConvolutionBounds}
  and \Cref{lem:AmalgamWeightedLInftyEmbedding}.
\end{proof}

The last result of this subsection concerns genuine \emph{Banach} function spaces
satisfying the so-called \emph{Fatou property}, see, e.g., \cite[Chapter~15, §~65]{ZaanenIntegration}.
Among others, it applies to weighted Lebesgue spaces.

\begin{lemma}\label{lem:SolidBanachSpaceConvolutionRelation}
  Let $Y$ be a solid Banach function space that is right-translation-invariant and satisfies
  the \emph{weak Fatou property}, meaning that there exists a constant $C > 0$ such that
  \[
    \Big\| \liminf_{n \to \infty} F_n \Big\|_{Y}
    \leq C  \liminf_{n \to \infty} \| F_n \|_Y
  \]
  for every sequence $(F_n)_{n \in \N}$ of nonnegative functions $F_n \in Y$
  with $\liminf_{n \to \infty} \| F_n \|_Y < \infty$.

  If $w : G \to (0,\infty)$ is measurable with $w(x) \geq \| \translationR{x} \|_{Y \to Y}$
  for all $x \in \group$, then $F_1 \ast F_2$ is almost-everywhere well-defined for every
  $F_1 \in Y$ and $F_2 : G \to \CC$ with $F_2^{\vee} \in L_w^1$, and it holds
  \[
    \| F_1 \ast F_2 \|_{Y}
    \lesssim \| F_1 \|_Y  \| F_2^{\vee} \|_{L_w^1}
    .
  \]
  In particular, if $w$ is a $1$-weight, then the following convolution relations hold:
  \[
    Y \ast L_w^1 \hookrightarrow Y
    \qquad \text{and} \qquad
    \wienerL{Y} \ast \wienerSt{L_w^1}
    \hookrightarrow Y \ast \wienerL{L_w^1}
    \hookrightarrow \wienerL{Y}
    .
  \]
\end{lemma}

\begin{proof}
  Let $M^+$ denote the set of all nonnegative measurable functions $F : \group \to [0,\infty]$,
  where functions are identified whenever they agree almost everywhere.
  The norm $\| \cdot \|_Y$ will be extended to a map $M^+ \to [0,\infty]$
  by setting $\| F \|_Y := \infty$ if $F \notin Y$.
  It is then straightforward to verify that $\| \cdot \|_Y : M^+ \to [0,\infty]$ is a function norm
  in the sense of \cite[Section~63]{ZaanenIntegration}.
  By our assumptions, it also follows that $\| \cdot \|_Y$ satisfies the weak Fatou property
  as defined in \cite[Section~65]{ZaanenIntegration};
  see \cite[Theorem~3 in Section~65]{ZaanenIntegration}.

  In accordance with \cite[Section~68]{ZaanenIntegration}, the \emph{associated seminorms}
  $\| \cdot \|^{(n)}_Y : M^+ \to [0,\infty]$ are inductively defined
  by $\| \cdot \|^{(0)}_Y := \| \cdot \|_Y$ and
  \[
    \| F \|_{Y}^{(n)}
    := \sup
       \bigg\{
         \int_{\group}
           F \cdot H
         \, d \haarMeasure
         \,\,\colon\,\,
         H \in M^+ \text{ with } \| H \|_{Y}^{(n-1)} \leq 1
       \bigg\}
    \quad \text{for } n \geq 1.
  \]
  Since $\| \cdot \|_Y$ satisfies the weak Fatou property, an application of
  \cite[Theorem~3 in Section~71]{ZaanenIntegration} shows that
  $\| \cdot \|_Y \asymp \| \cdot \|_{Y}^{(2)}$.

  Let $F_1,F_2,H \in M^+$ with $F_1 \in Y$, $F_2^{\vee} \in L_w^1$, and $\| H \|_{Y}^{(1)} \leq 1$.
  Tonelli's theorem shows $F_1 \ast F_2 \in M^+$.
  Further,
  \(
    F_1 \ast F_2 (x)
    = \int_{\group}
        F_1(y) F_2(y^{-1} x)
      \dd{y}
    = \int_{\group}
        (R_z F_1)(x) \, F_2 (z^{-1})
      \dd{z}
    .
  \)
  Based on this identity, another application of Tonelli's theorem yields that
  \[
    \int_{\group}
      (F_1 \ast F_2)(x)
      H(x) \;
    \dd{x}
    = \int_{\group}
        F_2(z^{-1})
        \int_{\group}
          (R_z F_1)(x) \, H(x) \;
        \dd{x}
      \dd{z}
    .
  \]
  By definition of $\| \cdot \|_Y^{(1)}$, it follows
  \(
    \int_{\group}
      (R_z F_1)(x) \, H(x)
    \dd{x}
    \leq \| H \|_Y^{(1)} \cdot \| R_z F_1 \|_Y
    \leq w(z) \cdot \| F_1 \|_Y
    .
  \)
  Hence,
  \[
    \int_{\group}
      (F_1 \ast F_2)(x)
      H(x)
    \dd{x}
    \leq \| F_1 \|_Y \!
         \int_{\group} w(z)  F_2^{\vee}(z) \dd{z}
    =    \| F_1 \|_Y \| F_2^{\vee} \|_{L_w^1}
    .
    %=    \| F_1 \|_Y \| F_2 \|_{L_w^1} .
  \]
  Since this holds for every $H \in M^+$ with $\| H \|_Y^{(1)} \leq 1$, this implies
  \[
    \| F_1 \ast F_2 \|_Y
    \lesssim \| F_1 \ast F_2 \|_Y^{(2)}
    \leq \| F_1 \|_Y  \| F_2^{\vee} \|_{L_w^1}.
  \]
  Thus, the convolution relation $Y \ast [L_w^1]^{\vee} \hookrightarrow Y$ holds for nonnegative
  functions; by solidity, this easily implies the general case.

  Finally, if $w$ is a $1$-weight, it follows that $L_w^1 = [L_w^1]^{\vee}$ with identical norms,
  so that the preceding convolution relation implies that $Y \ast L_w^1 \hookrightarrow Y$.
  Moreover, using the estimate \eqref{eq:ConvolutionMaximalFunctionEstimate},
  it follows for $F_1 \in Y$ and $F_2 \in \wienerL{L_w^1}$ that
  \begin{align*}
    \| F_1 \ast F_2 \|_{\wienerL{Y}}
    & = \| \maxL (F_1 \ast F_2) \|_Y
      \leq \big\| |F_1| \ast \maxL F_2 \big\|_Y \\
    & \lesssim \big\| |F_1| \big\|_Y  \| \maxL F_2 \|_{L_w^1}
      =  \| F_1 \|_Y  \| F_2 \|_{\wienerL{L_w^1}},
  \end{align*}
  as asserted.
\end{proof}

\section{Counterexamples}
\label{sec:counterexamples}

In this section two examples (both taken from \cite[Section~2.3]{FelixPhD}) are provided,
showing that in general
\[
 \wienerL{L^p_v} \ast [\wienerL{Y^{\vee}}]^{\vee}
 \nsubseteq \wienerL{Y}
\]
with $v (x) = \| L_{x^{-1}} \|_{\wienerL{Y} \to \wienerL{L}}$
and/or $v(x) = \| L_{x}\|_{\wienerL{Y} \to \wienerL{Y}}$.
These examples show that the convolution relation asserted in \cite[Theorem~5.2]{rauhut2007wiener}
fails in general.
This (incorrect) convolution relation is used several times in the coorbit theory
developed in \cite{rauhutcoorbitpreprint, rauhut2007coorbit}.

The first example shows that one cannot use the weight
$v(x) = \| L_{x^{-1}} \|_{\wienerL{Y} \to \wienerL{Y}}$.

\begin{example}\label{exa:RauhutFirstCounterexample}
Let $G = \R$ and consider the weight $w : \R \to (0,\infty), \; x \mapsto e^{x}$.
Let $Y := L_w^1 (\R)$.
For $f \in Y$ and $x \in \R$, it follows that
\[
  \left\Vert L_{x}f\right\Vert_{Y}
  = \int_{\R}
      |f(y - x)| \, e^y
    \, d y
  = \int_{\R}
      |f(z)| \, e^{z+x}
    \, d z
  = e^{x}  \left\Vert f \right\Vert_{Y},
\]
and hence $\| L_{x} \|_{Y\to Y} = e^{x}$.
As seen in \Cref{eq:WienerAmalgamTranslationNormEstimate}, this implies
\[
  v(x)
  := \| L_{x^{-1}}\|_{\wienerL{Y} \to \wienerL{Y}}
  \leq \| L_{x^{-1}}\|_{Y\to Y}
  = e^{-x},
\]
where we have written $G = \R$ multiplicatively, as well as additively, i.e., $x^{-1} = -x$.
Moreover note, for any measurable $f : \R \to \mathbb{C}$ that
\[
  \left\Vert f^{\vee} \right\Vert_{Y}
  = \int_{\R}
      |f(y)| \, e^{-y}
    \, d y
  = \left\Vert f \right\Vert_{L_{u}^{1}}
\]
for $u : \R \to (0,\infty), \;x \mapsto e^{-x}$.
Thus, $Y^{\vee} = L_{u}^{1}(\R)$.

Let $T > 0$ be arbitrary and define $f := \indicator_{(T, T+1)}$
as well as $g := \indicator_{(-T-1, -T)}$.
Using the (open, relatively compact, symmetric) unit neighborhood $Q := (-1,1)$, it follows that
\[
  \left(\maxL \indicator_{A}\right)(x)
  = \esssup_{y\in x+Q}
      \indicator_{A}(y)
  \leq \indicator_{A+Q}(x)
\]
for each measurable $A \subseteq \R$, since $\indicator_{A}(y) \neq 0$
for some $y \in x + Q$ implies $x + q = y\in A$ for a suitable $q \in Q$,
which in turn yields $x = y - q \in A-Q = A + Q$.
%Since $Y = L_{w}^{1}(G)$ is a Banach space, we can choose a $p$-norm with $p=1$.
Note that $\| \cdot \|_{L_w^1} = \| \cdot \|_Y$ is a $p$-norm with $p = 1$.
Thus,
\begin{align*}
  \left\Vert f\right\Vert_{\wienerL{ L_{v}^{p}}}
  = \big\| \maxL \, \indicator_{(T, T+1)} \big\|_{L_v^1}
  \leq \left\Vert \indicator_{(T-1,T+2)} \right\Vert_{L_{v}^{1}}
  \leq \int_{T-1}^{T+2}
          e^{-x}
        \, d x
  = e^{-T}  (e^1 - e^{-2})
  \leq e  e^{-T}
\end{align*}
and
\(
  \| g \|_{[\wienerL{Y^{\vee}}]^{\vee}}
  = \| g^{\vee} \|_{\wienerL{L_u^1}}
  = \| \maxL \indicator_{(T, T+1)} \|_{L_u^1}
  \leq \| \indicator_{(T-1, T+2)} \|_{L_u^1}
  \leq e  e^{-T} .
  %\left\Vert g \right\Vert_{\wienerL{Y^{\vee}}^{\vee}}
  %\leq \left\Vert \indicator_{(M-1,M+2)} \right\Vert_{L_{u}^{1}}
  %\leq 3 e \cdot e^{-M}.
\)

Since $f,g\in L^{2}(\R)$, it is easy to see that $f\ast g$ is continuous with
\[
  f\ast g(0)
  = \int_{\R}
      \indicator_{(T,T+1)}(x)  \indicator_{(-T-1, -T)} (0 - x) \;
    \, d x
  = 1
  .
\]
Hence, for arbitrary $x \in (-1,1)$, we have $\maxL[ f\ast g](x) \geq f\ast g(0) = 1$,
which yields that
\(
  \left\Vert f\ast g \right\Vert_{\wienerL{Y}}
  \geq \int_{(-1,1)} e^{x} \; \diff x
  \geq 1.
\)

Assume towards a contradiction that $\wienerL{L_v^p} \ast [\wienerL{Y^{\vee}}]^{\vee} \subseteq \wienerL{Y}$.
As an easy consequence of the closed graph theorem, this implies existence of
a constant $C = C(w,Q) > 0$ satisfying
\(
  \left\Vert f\ast g \right\Vert_{\wienerL{Y}}
  \leq C
       \left\Vert f \right\Vert_{\wienerL{L_{v}^{p}}}
       \left\Vert g \right\Vert_{[\wienerL{Y^{\vee}}]^{\vee}}
  .
\)
Then the above estimates yield that
\begin{align*}
  1
  & \leq \left\Vert f\ast g \right\Vert_{\wienerL{Y}}
    \leq C
         \left\Vert f \right\Vert_{\wienerL{L_{v}^{p}}}
         \left\Vert g \right\Vert_{[\wienerL{Y^{\vee}}]^{\vee}}
  \leq C  e^{2}
          e^{-2 T}
  \rightarrow 0
  \quad \text{as} \quad T \to \infty,
\end{align*}
a contradiction.
Thus, $\wienerL{L^p_v} \ast [\wienerL{Y^{\vee}}]^{\vee}  \nsubseteq \wienerL{Y}$.
\end{example}

The next example shows that \cite[Theorem~5.2]{rauhut2007wiener} also fails%
\footnote{The weight $v(x) = \| L_{x} \|_{\wienerL{Y} \to \wienerL{Y}}$
does not occur in the statement of \cite[Theorem~5.2]{rauhut2007wiener},
but it is this weight that is used in its proposed proof.}
for $v(x) = \| L_{x} \|_{\wienerL{Y} \to \wienerL{Y}}$.

\begin{example}\label{exa:RauhutCounterexample}
Let $G$ denote the affine group, i.e., $G := \R \times (0,\infty)$ with multiplication given by
\[
  (x,a) \cdot (y,b)
  = (x+ay, ab).
\]
Neutral element and inverse in $G$ are given by $e_{G} = (0,1) \in G$
and $(x,a)^{-1} =\left(-\frac{x}{a},a^{-1}\right)$, respectively.
The left Haar integral on $G$ is given by
\[
  \int_{G}
    f(g)
  \, d \mu_G (g)
  = \int_{0}^{\infty}
      \int_{\mathbb{R}}
        f(x,a)
      \diff x
    \,\frac{\diff a}{a^{2}}
\]
with modular function $\Delta \bigl( (x,a) \bigr) = a^{-1}$.

Define the weight
\[
  w : \quad
  G\to(1,\infty), \quad
  (x,a) \mapsto 1 + a = 1 + \Delta\left((x,a)^{-1}\right).
\]
Then $w$ is clearly continuous and submultiplicative.
Set $Y := L_{w}^{1}(G)$.
Since $L^{1}(G)$ is right invariant and isometrically left invariant,
it follows that $\wienerL{Y}$ is left and right invariant
(cf.\ Section~\ref{sec:AmalgamPrelims})
with
\begin{align}
  v(x,a)
   := \| {L_{\left(x,a\right)}} \|_{\wienerL{Y} \to \wienerL{Y}}
       \nonumber
    \leq \| L_{(x,a)} \|_{Y\to Y}
     \leq w(x,a)
     =    1 + a .
  \label{eq:RauhutCounterexampleVDominatesW}
\end{align}
Moreover, as seen in \Cref{sec:AmalgamPrelims}, $\wienerLNeutral_{Q'}(Y)$ is independent of the
choice of the (open, relatively compact) unit neighborhood $Q' \subseteq \group$.
In addition, a direct calculation using the identity
$\int_{\group} \! F(z) \; d \mu_G (z) \!=\! \int_{\group} F(z^{-1}) \Delta(z^{-1}) d \mu_G(z)$
shows that
\begin{align*}
  \left\Vert f^{\vee} \right\Vert_{L_{w}^{1}}
  & = \int_{G}
        \left|f^{\vee}(z)\right|
         w(z)
      \, d \mu_G (z)
    = \int_{G}
         \left|f(z^{-1})\right|
      \, d \mu_G (z)
      + \int_{G}
          \left|f(z^{-1})\right|
           \Delta(z^{-1})
      \, d \mu_G (z) \\
   & = \int_{G}
         |f(z)|
          \Delta(z^{-1})
       \, d \mu_G (z)
       + \int_{G}
           |f(z)|
         \, d \mu_G(z)
     = \left\Vert f \right\Vert_{L_{w}^{1}},
\end{align*}
so that $Y = Y^{\vee}$.

Let $\alpha \in (1,\infty)$ and $\beta \in (0,1)$
and set $\delta := \max\left\{ \alpha,\beta\right\} = \alpha$.
Define $f : G \to (0, \infty)$
by $f\bigl( (x,a) \bigr) = e^{-|x|}  \min\left\{ a^{\alpha}, a^{-\beta} \right\}$.
Then $f$ is clearly continuous and is
easily seen to satisfy $\left\Vert f \right\Vert_{\sup} \leq 1$.
In the following, we use the (open, relatively compact) unit neighborhoods
$Q_0 := (-1,1) \times (\frac{1}{2}, 2)$ and $Q' := Q_0^{-1}$.
Then the substitutions $(z,c) = (y,b)^{-1}$ and $(\mu,\nu) = (z,c) \cdot (x,a)$ yield
\begin{align*}
  \left( \maxL_{Q'} f^{\vee} \right)(x,a)
  & \leq \sup_{(y,b) \in (x,a) Q'}
           \left|f^{\vee}(y,b)\right|
    =    \sup_{(z,c) \in Q_0 (x,a)^{-1}}
           \left|f(z,c)\right| \\
  & = \sup_{(\mu,\nu) \in Q_0}
         \left|
           f\left((\mu,\nu) \cdot (x,a)^{-1} \right)
         \right|
   = \sup_{(\mu,\nu) \in Q_0}
        \left|
          f \left( (\mu,\nu) \cdot\left(-\tfrac{x}{a}, a^{-1}\right) \right)
        \right| \\
  & = \sup_{(\mu,\nu) \in Q_0}
        \left|
          f \left( \mu - \tfrac{\nu}{a}x,\,  \tfrac{\nu}{a} \right)
        \right|
   = \sup_{(\mu,\nu) \in Q_0}
        e^{-\left|\mu-\frac{\nu}{a}x\right|}
         \min \left\{ \left(\nu/a\right)^{\alpha},  \left(\nu/a\right)^{-\beta} \right\} .
\end{align*}
Using that $(\mu,\nu) \in Q_0 = (-1,1) \times (\frac{1}{2}, 2)$, we see
\[
  \left| \mu - \frac{\nu}{a}x \right|
  \geq \frac{\nu}{a}|x| - |\mu|
  \geq \frac{\nu}{a}|x| - 1
  \geq \frac{|x|}{2a} - 1
\]
and thus
\(
  e^{-\left|\mu-\frac{\nu}{a}x\right|}
  \leq e^{-\left(\frac{\left|x\right|}{2a}-1\right)}
  =    e e^{-\left|x\right|/2a}.
\)
Furthermore, one has $\frac{1}{2a} \leq \frac{\nu}{a} \leq \frac{2}{a}$, which entails
\(
  (\frac{\nu}{a})^{\alpha}
  \leq (\frac{2}{a})^{\alpha}
  \leq 2^{\delta} a^{-\alpha}
\)
and
\(
  \left(\frac{\nu}{a}\right)^{-\beta}
  = \left(\frac{a}{\nu}\right)^{\beta}
  \leq 2^{\beta} a^{\beta}
  \leq 2^{\delta} a^{\beta}.
\)
Combining these estimates yields
\begin{equation}
  \left(\maxL_{Q'} f^{\vee}\right)(x,a)
  \leq 2^{\delta} e
        e^{-\left|x\right|/2a}
        \min \left\{ a^{-\alpha},  a^{\beta} \right\}
  \quad \text{for all } (x,a) \in G .
  \label{eq:RauhutCounterexampleMaximalFunctionUpperEstimate}
\end{equation}
For $a \in (0,\infty)$, using the identity
\(
  C_{1}
  := \int_{\R}
       e^{-\left|y\right|}
     \, d y
  =  \int_{\R}
       e^{-|x| / 2a}
     \frac{d x}{2a},
\)
it follows that
\begin{align*}
  \left\Vert f^{\vee} \right\Vert_{W^L_{Q'} (Y)}
  & \leq 2^{\delta} e
         \int_{G}
           e^{-|x| / 2a}
            \min \left\{ a^{-\alpha},  a^{\beta} \right\}
            w(x,a)
         \, d \mu_G (x,a) \\
  & = 2^{\delta} e
      \int_{0}^{\infty}
          \frac{\min\left\{ a^{-\alpha},  a^{\beta} \right\}  (1+a)}
               {a^{2}}  2a \int_{\R} e^{-|x| / 2a}
        \frac{\diff x}{2a}
      \diff a \\
  & = 2^{\delta + 1} e C_{1}
      \int_{0}^{\infty}
        \frac{\min \left\{ a^{-\alpha},  a^{\beta} \right\}  (1+a)}
             {a}
      \diff a \\
  & \leq 2^{\delta + 1} e C_{1}
          \left[
                 \int_{0}^{1}
                   a^{\beta}  (1 + a)
                 \frac{\diff a}{a}
                 + \int_{1}^{\infty}
                     a^{-\alpha}  (1 + a)
                   \frac{\diff a}{a}
               \right]
  .
\end{align*}
Using $1 + a \leq 2$ for $a \in (0,1)$ and $1 + a \leq a + a = 2a$ for $a \in [1,\infty)$,
this gives
\[
  \left\Vert f^{\vee} \right\Vert_{W^L_{Q'} (Y)}
  \leq 2^{\delta + 2} e C_{1}
        \left[
               \int_{0}^{1}
                 a^{\beta}
               \frac{\diff a}{a}
               + \int_{1}^{\infty}
                   a^{1-\alpha}
                 \frac{\diff a}{a}
             \right]
  < \infty,
\]
because of $\beta > 0$ and $\alpha > 1$.
This implies $f^{\vee}\in \wienerL{Y} = \wienerL{Y^{\vee}}$
and thus $f \in \bigl[\wienerL{Y^{\vee}}\bigr]^{\vee}$, since
\(
  f^{\vee} \in \wienerL{Y}
  = \wienerL{L_{w}^{1}}
  \hookrightarrow \wienerL{L_{v}^{1}},
\)
since $v\leq w$.

Note that $Y = L_w^1(G)$ is a Banach space, so that $\| \cdot \|_Y$ is a $p$-norm for $p = 1$.
Now, if the convolution relation $\wienerL{L^p_v} \ast [\wienerL{Y^{\vee}}]^{\vee} \subseteq \wienerL{Y}$
would hold, then the above would yield $f^{\vee} \ast f \in \wienerL{Y}$.
We will now show that in fact $f^{\vee} \ast f \notin \wienerL{Y}$.
For this, note first that $f^{\vee}\in \wienerL{L_{w}^{1}} \hookrightarrow L^{1}$
since $w \geq 1$.
Furthermore, $f$ is bounded, so that $f^{\vee} \ast f : G \to \mathbb{C}$
is a well-defined, continuous, bounded function.
A direct calculation gives
\begin{align*}
    \left(f^{\vee}\ast f\right)(x,a)
  & = \int_{G}
        f^{\vee}(y,b)
         f\left((y,b)^{-1}  (x,a)\right)
      \, d \mu_G (y,b) \\
  & = \int_{0}^{\infty}
        \int_{\R}
          f\left(-\frac{y}{b},\frac{1}{b}\right)
           f\left(\frac{x-y}{b},\frac{a}{b}\right)
        \diff y
      \frac{\diff b}{b^{2}} \\
  & = \int_{0}^{\infty}\!\!
         \min \left\{ b^{-\alpha},b^{\beta}\right\}
         \min \left\{ \left(\frac{a}{b}\right)^{\alpha},\left(\frac{a}{b}\right)^{-\beta}\right\}
          \frac{b}{b^{2}}
         \int_{\R}
           e^{-\left|-\frac{y}{b}\right|}
           e^{-\left|\frac{x-y}{b}\right|}
         \frac{\diff y}{b}
       \diff b \\
  & = \int_{0}^{\infty}\!\!
         \min\left\{ b^{-\alpha},b^{\beta}\right\}
          \min
               \left\{
                 \left(\frac{a}{b}\right)^{\alpha},
                 \left(\frac{b}{a}\right)^{\beta}
               \right\}
          \int_{\R}
                  e^{-\left|z\right|}
                  e^{-\left|z-\frac{x}{b}\right|}
                \diff z
      \frac{\diff b}{b}.
\end{align*}
Let $x = 0$ and $a \geq 1$.
Then, for each $b \in (0,1)$, it follows that $b^{\beta} \leq 1 \leq b^{-\alpha}$
as well as $b < 1 \leq a$,
and thus $\left(b/a\right)^{\beta} \leq 1^{\beta} = 1 \leq \left(a/b\right)^{\alpha}$.
With the abbreviation
\(
  C_{2}
  := \int_{\mathbb{R}}
       ( e^{-|z|} )^{2}
     \diff z
  \in (0,\infty)
\),
these considerations and the above calculation imply
\begin{align*}
  \left( f^{\vee} \ast f \right)(0,a)
  & \geq C_{2}
          \int_{0}^{1}
                 b^{\beta}
                  \left(b/a\right)^{\beta}
               \frac{\diff b}{b}
  = C_{2}
       a^{-\beta}
       \int_{0}^{1} b^{2\beta-1} \diff b
      = C_{2}
       a^{-\beta}
       \frac{b^{2\beta}}{2\beta}\bigg|_{0}^{1} \\
  & = \frac{C_{2}}{2\beta}
       a^{-\beta}
 \numberthis \label{eq:RauhutCounterexampleConvolutionLowerEstimate}
\end{align*}
for all $a \geq 1$. For using this to obtain a lower bound on $\maxL (f^{\vee}\ast f)$,
first note for $(y,b) \in G$ that
\begin{align*}
  (y,b) Q_0
   =  \left[y + (-b,\, b)\right]
      \times \left( \frac{b}{2} , 2b \right)
   =: B_{b}(y) \times \left(\frac{b}{2} , 2b \right).
\end{align*}
For $b \in (1,\infty)$ and $y \in(-b,b)$, we thus have $(0,b) \in (y,b) Q_0$.
Since $f^{\vee} \ast f$ is continuous and $(y,b) Q_0$ is open, it follows that,
for $b \in (1,\infty)$ and $y \in (-b,b)$,
\begin{align*}
  \maxL_{Q_0} [f^{\vee}\ast f](y,b)
   \geq ( f^{\vee} \ast f )(0,b)
   \geq \frac{C_{2}}{2\beta}  b^{-\beta},
\end{align*}
where the estimate \eqref{eq:RauhutCounterexampleConvolutionLowerEstimate} was used in the last step.
Combining the obtained estimates gives
\begin{align*}
  \left\Vert f^{\vee} \ast f \right\Vert_{\wienerLNeutral_{Q_0}(Y)}
  & \geq \frac{C_{2}}{2\beta}
          \int_{G}
                 \indicator_{(1,\infty)}(b)
                  \indicator_{(-b,b)}(y)
                  b^{-\beta}
                  w(y,b)
               \, d \mu_G (y,b) \\
  & =    \frac{C_{2}}{2\beta}
          \int_{1}^{\infty}
                 \frac{b^{-\beta}}{b^{2}}
                  (1+b)
                  \int_{-b}^{b} \, d y
               \, d b \\
  & \geq \frac{C_{2}}{2\beta}
          \int_{1}^{\infty}
                 b^{-\beta}
               \, d b
    =    \infty ,
\end{align*}
because of $\beta \in (0,1)$.
Thus, $f^{\vee} \ast f \notin \wienerL{Y}$.
\end{example}

\begin{remark}
  The above example also shows that
  \begin{equation}
    \wienerL{L_{w}^{p}}
    \ast \bigl[\wienerL{L_{w^{\ast}}^{p}} \bigr]^{\vee}
    \nsubseteq \wienerL{L_{w}^{p}}
    \label{eq:RauhutInvalidLpConvolutionRelation}
  \end{equation}
  for the weight $w^{\ast}(x) = \Delta(x^{-1})  w(x^{-1})$.
  Hence, the convolution relation that is
  stated in \cite[Corollary~5.4]{rauhut2007wiener} is also false in general (even for $p = 1$).
  To see that the above example indeed implies \Cref{eq:RauhutInvalidLpConvolutionRelation},
  note that the weight $w$ in the example satisfies $w^{\ast} = w$
  and that $f^{\vee} \in \wienerL{L_{w}^{1}}$
  as well as $f \in \bigl[\wienerL{L_{w^{\ast}}^{1}}\bigr]^{\vee}$,
  but $f^{\vee} \ast f \notin \wienerL{Y} = \wienerL{L_{w}^{1}}$.
\end{remark}

\chapter{Coorbit spaces associated with integrable group representations}
\label{sec:coorbit}

This chapters develops the basic theory of coorbit spaces
associated with (possibly projective) integrable representations
and quasi-Banach function spaces.

\section{Admissible vectors and reproducing formulae}
\label{sec:admissible}

Let $(\pi, \Hpi)$ be a unitary $\sigma$-re\-pre\-sen\-ta\-tion of $G$
on the separable Hilbert space $\Hpi \neq \{ 0 \}$, i.e., a strongly measurable%
\footnote{This means that for each $f \in \Hpi$,
the map $G \to \Hpi, x \mapsto \pi(x) f$ is measurable.}
map $\pi : G \to \mathcal{U}(\Hpi)$ into the set $\CalU(\Hpi)$ of unitary operators on $\Hpi$
satisfying $\pi(e_G) = I_{\Hpi}$ and
\begin{equation}
  \pi(x) \pi(y)
  = \sigma(x,y) \pi(xy),
  \quad \text{for all } x,y \in G,
  \label{eq:SigmaProjectionDefinition}
\end{equation}
for some function $\sigma : G \times G \to \mathbb{T}$.
Any such function $\sigma$ is a cocycle in the sense of Section~\ref{sec:cocycle_twisted}.
One also says that $(\pi, \Hpi)$ is a unitary \emph{projective representation}.
The adjoint operator $\pi(x)^*$ of $\pi(x)$ for $x \in G$ is given by
\begin{equation}
  [\pi(x)]^*
  = [\pi(x)]^{-1}
  = \overline{\sigma(x,x^{-1})} \, \pi(x^{-1})
  = \overline{\sigma(x^{-1}, x)} \, \pi(x^{-1})
  .
  \label{eq:ProjectiveRepresentationInverse}
\end{equation}
In particular, $\sigma(x, x^{-1}) = \sigma(x^{-1}, x)$.
The map $\pi$ is called \emph{irreducible} if $\{0\}$ and $\Hpi$
are the only closed $\pi$-invariant subspaces of $\Hpi$, i.e.,
if $V \subseteq \Hpi$ and $\pi(x) v \in V$
for all $v \in V$ and $x \in G$, then $V = \{0\}$ or $V = \Hpi$.

For $g \in \Hpi$, define its associated \emph{coefficient transform} $V_g : \Hpi \to L^{\infty} (G)$ by
\[
  V_g f (x)
  = \langle f, \pi(x) g \rangle,
  \quad x \in G.
\]
By \cite[Lemma~7.1, Theorem~7.5]{varadarajan1985geometry},
it follows that
\begin{equation}
  |V_g f| : G \to \mathbb{C} 
  \label{eq:VoiceTransformAbsoluteValueContinuous}
\end{equation}
is continuous for all $f, g \in \Hpi$. A vector $g \in \Hpi$ is called \emph{admissible}
if $V_g : \Hpi \to L^2 (G)$ is well-defined and an isometry.

The following simple lemma collects several identities that will be used repeatedly
throughout this article.
Here, the twisted translation operators $L_x^{\sigma}$ and $R_x^{\sigma}$
and twisted convolution operator $\ast_{\sigma}$
are as defined in Section~\ref{sec:cocycle_twisted}.

\begin{lemma}\label{lem:basic_cocycle}
  Let $f, g, h \in \Hpi$.
  \begin{enumerate}[label=(\roman*)]
   \item For $x \in G$, the following \emph{intertwining property} holds:
         \[
           V_g f (x^{-1})
           = \overline{\sigma(x, x^{-1})} \, \overline{V_f g (x)},
           \quad
           V_g [\pi(x) f] = \twisttranslationL{x} [V_g f]
           \quad \text{and} \quad
           V_{\pi (x) g} f = \twisttranslationR{x} [V_g f].
         \]

   \item If $h \in \Hpi$ is admissible, then the following \emph{reproducing formula} holds:
         \begin{align} \label{eq:repro}
           V_g f  = V_h f \ast_{\sigma} V_g h.
         \end{align}
  \end{enumerate}
\end{lemma}

\begin{proof}
(i)
If $x, y \in G$, then
\begin{align*}
   V_g [\pi(x) f] (y)
   &= \langle f, \pi(x)^{\ast} \pi(y) g \rangle
    = \sigma(x, x^{-1})
      \langle f, \pi(x^{-1}) \pi( y)g \rangle \\
   &= \sigma(x, x^{-1})
      \overline{\sigma(x^{-1}, y)}
      V_g f (x^{-1} y).
\end{align*}
Using that
\(
  \sigma(x, x^{-1} y)
  = \overline{\sigma(x^{-1}, y)} \,
    \sigma(x x^{-1}, y) \,
    \sigma(x, x^{-1})
  = \sigma(x, x^{-1}) \,
    \overline{\sigma(x^{-1}, y)}
  ,
\)
it follows that $V_g [\pi(x) f] = \twisttranslationL{x} [V_g f]$.
The other two identities of part (i) are immediate consequences of the definitions
and of \Cref{eq:SigmaProjectionDefinition,eq:ProjectiveRepresentationInverse}.

(ii)
If $x \in G$, then using that $V_h : \Hpi \to L^2 (G)$ is an isometry gives
\begin{align*}
  V_g f (x)
  &= \langle f, \pi(x) g \rangle
   = \langle V_h f, V_h [\pi(x) g] \rangle_{L^2}
   = \int_G
       %\langle f, \pi(y) h \rangle_{\Hpi}
       V_h f (y) \,
       \langle \pi(y) h, \pi(x) g \rangle_{\Hpi}
     \; d\mu_G (y) \\
  &=
     \int_G
       V_h f (y) \twisttranslationL{y} [V_g h] (x)
     \; d\mu_G (y)
   = (V_h f \ast_{\sigma} V_g h) (x),
\end{align*}
where the penultimate equality marked follows from part (i).
\end{proof}

If $g \in \Hpi$ is admissible, then the image space $\CalK_g := V_g (\Hpi)$
is a closed subspace of $L^2 (G)$, since it is the isometric image of the Hilbert space $\Hpi$.
For arbitrary $F \in \CalK_g$, say $F = V_g f$, it follows from \Cref{lem:basic_cocycle}
that for any $x \in G$,
\begin{align} \label{eq:RKHS}
  F(x)
  = \langle f, \pi(x) g \rangle
  = \langle V_g f, V_g [\pi(x) g] \rangle
  = \langle V_g f, \twisttranslationL{x} [V_g g] \rangle
  = \int_{\group}
      F(y) \overline{\twisttranslationL{x} [V_g g](y)}
    \; d\mu_G (y),
\end{align}
which shows that $\CalK_g$ is a \emph{reproducing kernel Hilbert space}, i.e.,
for any $x \in G$, the point evaluation $\CalK_g \ni F \mapsto F(x) \in \mathbb{C}$
is a well-defined, continuous linear functional.
The \emph{reproducing kernel} of $\CalK_g$ is the function
\begin{align}\label{eq:RK}
  K : \quad
  G \times G \to \mathbb{C}, \quad
  (x,y) \mapsto V_g [\pi(y) g](x)
                = \overline{\twisttranslationL{x} [V_g g] (y)}
\end{align}
satisfying $F(x) = \langle F, K(\cdot, x) \rangle_{L^2}$ for any $F \in \CalK_g$ and $x \in G$.

For an \emph{irreducible} $\sigma$-representation $(\pi, \Hpi)$,
a convenient criterion for the existence
of admissible vectors is provided by the \emph{orthogonality relations}
for square-integrable representations,
cf.\ \cite[Theorem~3]{duflo1976on} or \cite[Theorem~4.3]{carey1976square}
for genuine representations,
and \cite[Theorem~3]{aniello2006square} for projective representations.

\begin{theorem}[\cite{duflo1976on, carey1976square, aniello2006square}]\label{thm:ortho}
  Let $(\pi, \Hpi)$ be an irreducible $\sigma$-representation of $G$.
  Suppose there exists $g \in \Hpi \setminus \{0\}$ satisfying $V_g g \in L^2 (G)$.
  Then there exists a unique, self-adjoint, positive operator $C_{\pi} : \dom (C_{\pi}) \to \Hpi$
  such that
  \begin{align} \label{eq:ortho}
    \big \langle V_{g_1} f_1, V_{g_2} f_2 \big \rangle_{L^2}
    = \langle C_{\pi} g_2, C_{\pi} g_1 \rangle_{\Hpi} \langle f_1 , f_2 \rangle_{\Hpi}
  \end{align}
  for all $f_1, f_2 \in \Hpi$ and $g_1, g_2 \in \dom(C_{\pi})$,
  with $\dom(C_{\pi}) = \{ g \in \Hpi : V_g g \in L^2 (G) \}$.
\end{theorem}

\begin{remark*}
  The positivity of $C_\pi$ implies, in particular, that $C_\pi$ is injective.
  This can also be derived from \Cref{eq:ortho} by noting that
  \begin{equation}
    \| C_\pi h \|_{\Hpi}^2
    = \| V_h h \|_{L^2}^2 / \| h \|_{\Hpi}^2
    > 0
    \quad \text{for} \quad
    h \in \dom(C_\pi) \setminus \{ 0 \}
    .
    \label{eq:DufloMoorePositivity}
  \end{equation}
  Here, the norm $\| V_h h \|_{L^2}$ is positive since $|V_h h|$ is continuous
  (see \Cref{eq:VoiceTransformAbsoluteValueContinuous})
  and satisfies $|V_h h (e_G)| = \| h \|_{\Hpi}^2 > 0$.
\end{remark*}

Theorem~\ref{thm:ortho} yields that if $\pi$ is irreducible and $\dom (C_{\pi}) \neq \{0\}$,
then any $g \in \dom (C_{\pi})$ is (a scalar multiple of) an admissible vector.
In addition, the orthogonality relations \eqref{eq:ortho} yield that
\begin{align}\label{eq:repro_ortho}
  \langle C_{\pi} g_1, C_{\pi} g_2 \rangle_{\Hpi} V_{f_2} f_1
  =  V_{g_2} f_1 \ast_{\sigma} V_{f_2} g_1
\end{align}
for $f_1, f_2 \in \Hpi$ and $g_1, g_2 \in \dom(C_{\pi})$;
the argument is similar to the proof of Part~(ii) of Lemma~\ref{lem:basic_cocycle}.

\section{Integrable vectors}
\label{sec:integrable}

Let $w : G \to [1, \infty)$ be a measurable submultiplicative weight,
i.e., $w(x y) \leq w(x) w(y)$ for all $x,y \in \group$.
Henceforth, the $\sigma$-representation $\pi$ is assumed to be $w$-integrable
in the sense that the set
\[
  \Aw
  := \big\{ g \in \Hpi : V_g g \in L^1_w (G) \big\}
\]
is nontrivial; that is, $\Aw \neq \{0\}$.
In addition, it is assumed that there exists an \emph{admissible} vector $g \in \Aw$.
For any such admissible $g \in \Aw$, define the space
\[
  \Hw
  := \Hw(g)
  := \big\{ f \in \Hpi : V_g f \in L^1_w (G) \big\}
\]
and equip it with the norm $\| f \|_{\Hw} := \| f \|_{\Hw(g)} := \| V_g f \|_{L^1_w}$.

The following lemma collects several basic properties of the space $\Hw$.

\begin{lemma}\label{lem:basic_Hw}
  Fix an admissible vector $g \in \Aw$ and write $\Hw := \Hw(g)$.
  \begin{enumerate}[label=(\roman*)]
    \item If $h \in \Aw$ is an admissible vector satisfying $V_g h, V_h g \in L^1_w (G)$,
          then $\Hw (g) = \Hw(h)$ with norm equivalence
          $\| \cdot \|_{\Hw(g)} \asymp \| \cdot \|_{\Hw(h)}$.

    \item The pair $(\Hw, \| \cdot \|_{\Hw})$ is a $\pi$-invariant separable Banach space
          satisfying $\Hw \hookrightarrow \Hpi$ and $\pi(x) g \in \Hw$ for all $x \in \group$.
          Moreover, $\Hw$ is norm dense in $\Hpi$ and $\pi$ acts on $\Hw$, with operator norm
          bounded by
          \[
            \| \pi(x) \|_{\Hw \to \Hw} \leq w(x), \quad x \in G.
          \]

    \item For each $f \in \Hw$, the vector-valued maps
          \begin{alignat*}{5}
            & \Xi_f : \quad
            &&  \group \to \Hw, \quad
            &&  x \mapsto \pi(x) f \\
            \qquad \text{and} \qquad
            & F_{f,g} : \quad
            &&  G \to \Hw, \quad
            &&  x \mapsto V_g f (x)  \pi (x) g
          \end{alignat*}
          are Bochner measurable.
          Furthermore, $F_{f,g}$ is Bochner integrable with \[f = \int_G F_{f,g} d\mu_G.\]

    \item The orbit $\pi(G) g$ is complete in $\Hw$,
          i.e., $\overline{\Span \{ \pi(x) g : x \in G \}} = \Hw$.
  \end{enumerate}
\end{lemma}

\begin{proof}
(i)
If $f \in \Hw(g)$, then \eqref{eq:repro} yields that
$V_h f = V_g f \ast_{\sigma} V_h g$.
Since $V_g f, V_h g \in L^1_w$ by assumption and since $f \in \Hw(g)$,
and because of $L^1_w \ast L^1_w \hookrightarrow L^1_w$
(see, e.g., \cite[\S 3.7]{ReiterClassicalHarmonicAnalysis}),
this implies that $V_h f \in L^1_w(G)$ (and hence $f \in \Hw(h)$), with
\[
  \| f \|_{\Hw(h)}
  \leq \big\| |V_g f| \ast |V_h g| \big\|_{L^1_w}
  \lesssim \|V_g f \|_{L^1_w}
  = \| f \|_{\Hw(g)}.
\]
By symmetry, this yields that $\| \cdot \|_{\Hw(h)} \asymp \| \cdot \|_{\Hw (g)}$.
In the sequel, simply set $\Hw := \Hw(g)$.

(ii)
If $f \!\in \Hw$ and $x \!\in\! G$, then
\[
  \| \pi (x) f \|_{\Hw} \!
  = \| \twisttranslationL{x} [V_g f] \|_{L^1_w}
  = \| L_x [V_g f] \|_{L_w^1}
  \leq w(x)  \| V_g f \|_{L_w^1}
  \leq w(x) \| f \|_{\Hw}
\]
since $\| L_x \|_{L^1_w \to L_w^1} \leq w(x)$; see, e.g.,
\cite[Proposition~3.7.6]{ReiterClassicalHarmonicAnalysis}.
Thus, $\Hw$ is $\pi$-invariant and $\| \pi(x) \|_{\Hw \to \Hw} \leq w(x)$.
The orbit $\pi(G) g$ is complete in $\Hpi$ if and only if $V_g : \Hpi \to L^2 (G)$ is injective.
Therefore, since $V_g$ is an isometry and since $g \in \Aw$ and thus $g \in \Hw$,
it follows that $\Hw \supseteq \pi(\group) g$ is norm dense in $\Hpi$.
The reproducing formula \eqref{eq:repro}, combined with the convolution relation
$L^1 \ast L^2 \hookrightarrow L^2$ (see e.g.\ \cite[Proposition~2.39]{FollandAHA})
and with $w \geq 1$ yields that
\begin{align} \label{eq:Hw_embedding}
  \| f \|_{\Hpi}^2
  = \| V_g f \|^2_{L^2}
  \leq \big\| |V_g f| \ast |V_g g| \big\|^2_{L^2}
  \leq \| V_g g \|_{L^2}^2 \| V_g f \|_{L^1}^2
  \lesssim \| f \|^2_{\Hw}.
\end{align}
This shows that $\Hw \hookrightarrow \Hpi$ and that $\| \cdot \|_{\Hw}$ is positive definite,
hence defines a norm.
In addition, if $(f_n)_{n \in \mathbb{N}}$ is a Cauchy sequence in $\Hw$,
then \eqref{eq:Hw_embedding} yields that $(f_n)_{n \in \mathbb{N}}$ is Cauchy in $\Hpi$,
and hence converges to some $f' \in \Hpi$.
The sequence $F_n := V_g f_n$ being Cauchy in $L^1_w (G)$
yields convergence to some $F \in L^1_w(G)$.
Since $F_n (x) = V_g f_n (x) \to V_g f' (x)$ as $n \to \infty$ for all $x \in G$,
it follows that $V_g f' = F$, and thus $f' \in \Hw$
with $\| f' - f_n \|_{\Hw} = \| F_n - F \|_{L^1_w} \to 0$.
This shows that $(\Hw, \| \cdot \|_{\Hw})$ is a Banach space.

(iii)
Part (ii) shows that $\Xi_f$ and $F_{f,g}$ are well-defined for $f \in \Hw$.
We first show that $\Xi_f$ is Bochner measurable.
Since $V_g f$ is measurable and $g \in \Hw$,
this then easily implies the measurability of $F_{f,g} = V_g f \cdot \Xi_g$.

To show that $\Xi_f$ is measurable, first note that the group $G$ is second countable,
so that the space $L^1_w (G)$ is separable;
see, e.g., \cite[Proposition~3.4.5]{cohn2013measure}.
As $V_g : \Hw \to L^1_w (G)$ is an isometry, also $\Hw$ is separable,
hence Pettis' measurability theorem (cf.\ \cite[Theorem~E.9]{cohn2013measure})
implies that $\Xi_f : G \to \Hw$ is strongly measurable
whenever $\varphi \circ \Xi_{f}$ is Borel measurable
for each continuous linear functional $\varphi \in (\Hw)'$.
To show the latter, given $\varphi \in (\Hw)'$, define $\psi_0 : V_g (\Hw) \to \mathbb{C}$
by $\psi_0 (V_g h) = \varphi(h)$ for $h \in \Hw$.
Since $V_g : \Hw \to L^1_w(G)$ is an isometry,
the functional $\psi_0$ is well-defined, linear, and bounded with respect to $\| \cdot \|_{L_w^1}$.
Hence, by the Hahn-Banach theorem, $\psi_0$ extends to a linear functional $\psi \in (L^1_w(G))'$,
which is then given by integration against some $H \in L^{\infty}_{1/w} (G)$.
Thus,
\begin{align*}
  \varphi(\Xi_{f} (x))
  &= \psi_0 \bigl(V_g [\pi(x) f] \bigr)
   = \int_G H(y)  V_g [\pi(x) f] (y) \; d\mu_G (y) \\
  &= \int_G H(y)  \twisttranslationL{x} [V_g f] (y) \; d\mu_G (y),
\end{align*}
and hence Fubini's theorem
implies the Borel measurability of $x \mapsto \varphi(\Xi_{f} (x))$.
Thus, by Pettis' measurability theorem, the map $\Xi_{f}$ is strongly measurable,
and hence so is the map $F_{f,g}$.
In addition, the estimate $\| \pi (x) g \|_{\Hw} \leq w(x) \| g \|_{\Hw}$
of part (ii) implies directly that
\[
  \| F_{f,g} (\cdot) \|_{\Hw}
  \leq w(\cdot) \| g \|_{\Hw} |V_g f (\cdot)| \in L^1 (G),
\]
whence $F_{f,g}$ is Bochner integrable.
Thus $f' := \int_G F_{f,g} \; d\mu_G \in \Hw \hookrightarrow \Hpi$ is well-defined.
For showing $f = f'$, it suffices to show
that $V_g f' = V_g f$, since $V_g : \Hpi \to L^2(\group)$ is an isometry.
Let $\iota : L^1_w (G) \hookrightarrow L^1 (G)$ be the canonical embedding.
Since the Bochner integral commutes with bounded linear operators
(cf.\ \cite[VI, Theorem~4.1]{LangRealFunctional}),
a direct calculation using \Cref{eq:repro} entails
\begin{align*}
  \iota (V_g f')
  &= \int_G
       [(\iota \circ V_g) F_{f,g}] (x)
     \; d\mu_G (x)
   = \int_G V_g f (x) \twisttranslationL{x} [V_g  g] (\cdot) \; d\mu_G (x) \\
  &= V_g f \ast_{\sigma} V_g g
   = V_g f
   = \iota ( V_g  f),
\end{align*}
which yields that $V_g f' = V_g f$.

(iv)
If $f \in \Hw$, then clearly $V_g f (x)  \pi (x) g \in \overline{\Span \pi(G) g} \leq \Hw$
for each $x \in G$.
Hence, by part (iii), also
\(
  f
  = \int_G
      V_g f (x) \pi (x) g
    \; d\mu_G (x)
  \in \overline{\Span \pi(G) g}.
\)
Thus, $\Hw = \overline{\Span \pi(G) g}$.
\end{proof}

The proofs of parts (i) and (ii) of Lemma~\ref{lem:basic_Hw}
closely follow \cite[Lemma~4.2]{feichtinger1988unified},
but the direct proof of part (iv) using the Bochner integral (see part (iii)) appears to be new.

\begin{remark} \label{rem:reducible_vgh}
  In Part~(i) of Lemma~\ref{lem:basic_Hw}, the assumption
  that the admissible vectors $g, h \in \Aw$ satisfy $V_g h, V_h g \in L^1_w (G)$
  is essential for the independence claim---at least, when $\pi$ is reducible.
  Indeed, in \cite[\S 2.1]{fuehr2015coorbit}, an example is given of a reducible representation
  $\pi$ admitting admissible vectors $g, h \in \Aw$, but for which $V_g h \notin L^1_w (G)$,
  showing that $h \in \Hw (h)$ but $h \notin \Hw (g)$.
\end{remark}

The following result shows that the behavior discussed in the above remark
cannot occur for \emph{irreducible} representations;
in particular, it shows that the additional assumption $V_g h, V_h g \in L^1_w (G)$
in Part~(i) of Lemma~\ref{lem:basic_Hw} is automatic in the irreducible case.

\begin{lemma} \label{lem:cross_irreducible}
  Let $(\pi, \Hpi)$ be an \emph{irreducible} unitary $\sigma$-representation.
  Let $Y \hookrightarrow L^1 (G)$ be a solid quasi-Banach function space
  that is left- and right translation invariant and also satisfies $Y \ast Y \hookrightarrow Y$
  (or, $Y \ast Y^{\vee} \hookrightarrow Y$).
  Then the space
  \[
    \mathcal{C}_Y := \big\{ g \in \Hpi : V_g g \in Y \big\}
  \]
  is a $\pi$-invariant vector space with $V_h g \in Y$ for all $g, h \in \mathcal{C}_Y$
  and with $\mathcal{C}_Y \subseteq \dom (C_{\pi}^2)$, where $C_{\pi} : \dom (C_{\pi}) \to \Hpi$
  is the operator given by Theorem~\ref{thm:ortho}.
  If $\mathcal{C}_Y$ is nontrivial, then $\mathcal{C}_Y$ is norm dense in $\Hpi$.
\end{lemma}

\begin{proof}
If $g \in \mathcal{C}_Y \subseteq \Hpi$, then $V_g g \in L^\infty(G)$
and $V_g g \in Y \hookrightarrow L^1(G)$ and thus $V_g g \in L^2(G)$.
Therefore, $g \in \dom(C_{\pi})$, and $C_{\pi} g \in \Hpi$ is well-defined.
We next show that $C_{\pi} g \in \dom(C_{\pi}^*)$,
for which we can clearly assume that $g \neq 0$.
For $f \in \dom(C_{\pi})$, the orthogonality relations \eqref{eq:ortho} gives
\begin{align*}
  |\langle C_{\pi} f, C_{\pi} g \rangle_{\Hpi} |
  = |\langle V_g g, V_f g \rangle_{L^2} |
     \| g \|_{\Hpi}^{-2}
  \leq \| g \|_{\Hpi}^{-2} \|V_g g \|_{L^1} \| V_f g \|_{L^{\infty}}
  \lesssim \| g \|_{\Hpi}^{-1} \| f \|_{\Hpi} \|V_g g \|_Y.
\end{align*}
In particular, this shows that the linear functional
$f \mapsto \langle C_{\pi} f, C_{\pi} g \rangle_{\Hpi}$
is bounded from $\dom(C_{\pi})$ into $\mathbb{C}$.
Thus $C_{\pi} g \in \dom(C_{\pi}^*) = \dom (C_\pi)$ and therefore $g \in \dom(C_\pi^2)$.

Let $g, h \in \mathcal{C}_Y \setminus \{0\}$ be arbitrary.
For showing that $V_h g \in Y$, it will first be shown
that there exists $f \in \mathcal{C}_{Y}$ with
\(
  \langle C_{\pi} g, C_{\pi} f \rangle_{\Hpi}
  \neq 0
  \neq \langle C_{\pi} h , C_{\pi} f \rangle_{\Hpi}.
\)
Note that if $\langle C_{\pi} g, C_{\pi} h \rangle_{\Hpi} \neq 0$,
then choosing $f = g$ gives
$\langle C_\pi h, C_\pi f \rangle_{\Hpi} = \langle C_\pi h, C_{\pi} g \rangle_{\Hpi} \neq 0$
and $\langle C_{\pi} g, C_{\pi} f \rangle_{\Hpi} = \| C_\pi g \|_{\Hpi}^2 > 0$;
see \Cref{eq:DufloMoorePositivity}.
Therefore, it remains to consider the case $\langle C_{\pi} g , C_{\pi} h \rangle_{\Hpi} = 0$.
Since $h \neq 0$ and $C_{\pi}$ is injective, it follows that $C_{\pi} h\neq 0 \neq C_{\pi}^2 h$.
Since $\pi$ is irreducible, the orbit $\pi (G) g$ is complete in $\Hpi$,
and hence
\(
  0
  \neq \langle \pi(x) g , C_{\pi}^2 h \rangle_{\Hpi}
  = \langle C_{\pi} [\pi(x) g] , C_{\pi} h \rangle_{\Hpi}
\)
for a suitable $x \in G$.
For $\varepsilon > 0$, define $f_{\varepsilon} := g + \varepsilon  \pi (x) g$.
Then
\[
  \langle C_{\pi} f_{\varepsilon} , C_{\pi} h \rangle_{\Hpi}
  = \langle C_{\pi} g, C_{\pi} h \rangle_{\Hpi}
    + \varepsilon \langle C_{\pi} [\pi(x) g] , C_{\pi} h \rangle_{\Hpi}
  = \varepsilon  \langle C_{\pi} [\pi(x) g], C_{\pi} h \rangle_{\Hpi}
  \neq 0
\]
and
\[
  \langle C_{\pi} f_{\varepsilon} , C_{\pi} g \rangle_{\Hpi}
  = \|  C_{\pi} g \|^2_{\Hpi}
    + \varepsilon  \langle C_{\pi} [\pi(x) g] , C_{\pi} g \rangle_{\Hpi}
  \neq 0,
\]
whenever $\varepsilon > 0$ is chosen sufficiently small.
For such $\varepsilon$, it remains to show that $f := f_{\varepsilon} \in \mathcal{C}_Y$.
Using the intertwining properties
(Part~(i) of \Cref{lem:basic_cocycle})
and the left- and right invariance of $Y$, it follows that
\begin{align*}
  V_f f
  &= V_g g
     + \varepsilon V_{\pi (x) g} g
     + \varepsilon V_g [\pi(x) g]
     + \varepsilon^2 V_{\pi(x) g} [\pi(x) g] \\
  &= V_g g
     + \varepsilon \twisttranslationR{x} [V_g g]
     + \varepsilon \twisttranslationL{x} [V_g g]
     + \varepsilon^2 \twisttranslationL{x} \twisttranslationR{x} [V_g g]
   \in Y,
\end{align*}
so $f \in \mathcal{C}_Y$.
Applying the reproducing formula \eqref{eq:repro_ortho} yields
\begin{align*}
 V_h g
 = \frac{1}{\langle C_{\pi} h, C_{\pi} f \rangle_{\Hpi}} V_f g \ast_{\sigma} V_h h
 = \frac{1}{\langle C_{\pi} h, C_{\pi} f \rangle_{\Hpi} \langle C_{\pi} f, C_{\pi} g \rangle_{\Hpi}}
   V_g g \ast_{\sigma} V_f f \ast_{\sigma} V_h h.
\end{align*}
Therefore, if $Y \ast Y \hookrightarrow Y$, then $V_h g \in Y$.
On the other hand, if $Y \ast Y^{\vee} \hookrightarrow Y$ holds,
then using that $V_g g, V_f f, V_h h \in Y \cap Y^{\vee}$,
also easily yields $V_h g \in Y$.

Next, we show that $\mathcal{C}_Y$ is a $\pi$-invariant vector space.
The invariance follows directly from the identity
$V_{\pi(x) f} [\pi(x) f] = \twisttranslationL{x} \twisttranslationR{x} [V_f f]$,
where $x \in G$,
combined with the left- and right invariance of $Y$.
For $g, h \in \mathcal{C}_Y$, the identity
\[
  V_{g+h} (g+h)
  = V_g g + V_g h + V_h g + V_h h,
\]
and $V_g h, V_h g \in Y$, imply that $\mathcal{C}_Y$ is a vector space.

Finally, if $\mathcal{C}_Y \neq \{0\}$, then it is norm dense in $\Hpi$
by irreducibility of $\pi$ and since $\mathcal{C}_Y$ is $\pi$-invariant.
\end{proof}

Lemma~\ref{lem:basic_Hw} implies, in particular, that the \emph{anti}-dual space
\[
  \Reservoir_w
  := \Reservoir_w (g)
  := \bigl(\Hw (g)\bigr)^*
\]
is a well-defined Banach space.
The associated conjugate-linear pairing will be denoted by
\[
  \langle f, h \rangle
  := f(h), \quad f \in \Reservoir_w, \; h \in \Hw.
\]
The representation $\pi$ on $\Hpi$ can be extended to act on $\Reservoir_w$ via
\begin{equation}
  \pi(x) f : \,\,\,\,
  \Hw \to \mathbb{C}, \,\,\,\,
  h \mapsto f\bigl( [\pi(x)]^\ast h\bigr)
          = \sigma(x^{-1},x) f\bigl(\pi(x^{-1}) h\bigr),
  \label{eq:RepresentationOnReservoir}
\end{equation}
for $x \in G$ and $f \in \Reservoir_w$.
This is well-defined, since $\pi(x^{-1}) : \Hw \to \Hw$ is well-defined, linear, and bounded.
The associated (extended) matrix coefficients are defined by
\[
  V_h f (x)
  = \langle f, \pi(x) h \rangle,
  \quad f \in \Reservoir_w, h \in \Hw
\]
for $x \in G$.

The next lemma collects the most important properties of these objects
for the purposes of this article.

\begin{lemma}\label{lem:basic_Rw}
  Let $g \in \Aw$ be admissible and write $\Hw = \Hw(g)$.
  Then the following hold:
  \begin{enumerate}[label=(\roman*)]
    \item The pairing $\langle \cdot , \cdot \rangle : \Reservoir_w \times \Hw \to \mathbb{C}$
          is an extension of the inner product $\langle \cdot , \cdot \rangle_{\Hpi}$;
          that is, $\Hpi \hookrightarrow \Reservoir_w$
          and $\langle f,h \rangle = \langle f,h \rangle_{\Hpi}$ for $f \in \Hpi$
          and $h \in \Hw$.
          Moreover, if $f \in \Hpi \subseteq \Reservoir_w$, then the definition
          of $\pi(x) f$ in \Cref{eq:RepresentationOnReservoir}
          agrees with the original definition.
          Finally, for any $f \in \Reservoir_w$, the extended mapping
          \[
            V_h : \quad
            \Reservoir_w \to L^{\infty}_{1/w} (G)
          \]
          is well-defined, linear, and bounded.

    \item If $f \in \Reservoir_w$ and $h \in \Hw$, then
          \[
            V_h [\pi(x) f]
            = \twisttranslationL{x} [V_h f],
          \]
          for $x \in G$.

    \item The map $V_g : \Reservoir_w \to L^{\infty}_{1/w} (G)$ is injective.
          In addition,
          \begin{align} \label{eq:repro_extended}
            V_h f = V_g f \ast_{\sigma} V_h g
          \end{align}
          for any $f \in \Reservoir_w$ and $h \in \Hw$.
          Moreover,
          $\langle f, h \rangle =  \langle V_g f, V_g h \rangle_{L^{\infty}_{1/w}, L^1_{w}}$.

    \item There exists a bounded linear operator
          \[
            V_g^* : \quad
            L^{\infty}_{1/w} (G) \to \Reservoir_w
          \]
          satisfying $V_g (V_g^* F) = F$ for every $F \in L^{\infty}_{1/w} (G)$
          that satisfies $F = F \ast_{\sigma} V_g g$.

    \item The function $f \mapsto \| V_g f \|_{L^{\infty}_{1/w}}$ defines
          an equivalent norm on $\Reservoir_w$.
  \end{enumerate}
\end{lemma}

\begin{proof}
(i)
 Using the embedding $\Hw \hookrightarrow \Hpi$ of Lemma~\ref{lem:basic_Hw},
 it follows that  $\iota : \Hpi \to \Reservoir_w$
 given by $(\iota f) (h) = \langle f, h \rangle_{\Hpi}$ for $h \in \Hw$
 is well-defined, and
 \[
   |(\iota f) (h) |
   = |\langle f, h \rangle_{\Hpi} |
   \lesssim \| f \|_{\Hpi} \| h \|_{\Hw}
 \]
 for $f \in \Hpi$ and $h \in \Hw$.
 Thus $\iota f \in \Reservoir_w$ and $\iota : \Hpi \to \Reservoir_w$ is bounded.
 Since $\Hw$ is norm dense in $\Hpi$ by Lemma~\ref{lem:basic_Hw},
 the map $\iota$ is injective, and thus a continuous embedding.
 In particular, $\langle \iota f, h \rangle = (\iota f)(h) = \langle f, h \rangle_{\Hpi}$,
 so that $\langle \cdot , \cdot \rangle$ extends $\langle \cdot , \cdot \rangle_{\Hpi}$.

 Next, if $f \in \Hpi$ and $h \in \Hw \subseteq \Hpi$,
 then we have with $\pi(x) f \in \Hpi$ as defined before
 and $\pi(x) \iota f \in \Reservoir_w$ as defined in \Cref{eq:RepresentationOnReservoir} that
 \[
   \iota [\pi(x) f] (h)
   = \langle \pi(x) f, h \rangle_{\Hpi}
   = \langle f, [\pi(x)]^\ast h \rangle_{\Hpi}
   = \iota(f) \big( [\pi(x)]^\ast h \big)
   = \pi(x) [\iota(f)] (h)
 \]
 and hence $\iota[\pi(x) f] = \pi(x) [\iota (f)]$, showing that
 the definition of $\pi(x) f$ in \Cref{eq:RepresentationOnReservoir}
 agrees with the action of $\pi$ on $\Hpi \subseteq \Reservoir_w$.

 Next, let $f \in \Reservoir_w = (\Hw)^{\ast}$ and $h \in \Hw$ and recall
 from Part~(iii) of \Cref{lem:basic_Hw} that the map $\Xi_h : G \to \Hw, \; x \mapsto \pi(x) h$
 is measurable.
 Since $f : \Hw \to \CC$ is continuous, this implies that
 $V_h f : G \to \CC, x \mapsto V_h f(x) = f (\pi(x) h)$ is Borel measurable.
 In addition, the norm estimate of Part~\emph{(ii)} of Lemma~\ref{lem:basic_Hw} shows that
 \[
   |V_h f (x) |
   \leq  \| f \|_{\Reservoir_w} \|\pi(x) h \|_{\Hw}
   \leq w(x) \| f \|_{\Reservoir_w}  \| h \|_{\Hw},
 \]
 whence $V_h f \in L^{\infty}_{1/w} (G)$.

(ii)
 If $x, y \in G$, then
 \begin{align*}
   (V_h \pi(x) f)(y)
   = \sigma(x, x^{-1}) \langle f, \pi(x^{-1}) \pi(y) h \rangle
   = \sigma(x, x^{-1}) \overline{\sigma(x^{-1}, y)} V_h f(x^{-1} y)
   = \twisttranslationL{x} [V_h f] (y),
 \end{align*}
 since
  \(
    \sigma(x, x^{-1} y)
    = \overline{\sigma(x^{-1}, y)} \sigma(x x^{-1}, y) \sigma(x, x^{-1})
    = \sigma(x, x^{-1}) \overline{\sigma(x^{-1}, y)}
    .
  \)

(iii)
 By \Cref{lem:basic_Hw}, $\pi(G) g$ is complete in $\Hw = \Hw(g)$;
 thus, the map $V_g : \Reservoir_w \to L^{\infty}_{1/w}$ is injective.
 Let $f \in \Reservoir_w$, $h \in \Hw$ and $x \in G$.
 Let $(\Hw)'$ denote the dual space of $\Hw$ and
 define $f' \in (\Hw)'$ by $f' (k) = \overline{f(\pi(x) k)}$ for $k \in \Hw$.
 Using the Bochner integrable map $F_{h, g} : G \to \Hw$ defined in Lemma~\ref{lem:basic_Hw},
 it follows that $h = \int_G F_{h,g} \; d\mu_G$ with convergence in $\Hw$.
 Since Bochner integration commutes with bounded linear functionals
 (cf.\ \cite[Proposition E.11]{cohn2013measure}), a direct calculation using \Cref{lem:basic_cocycle} entails
 \begin{align*}
   V_h f (x)
   &= \overline{f'(h)} = \overline{\int_G f' (F_{h,g} (y)) \; d\mu_G (y)}
    = \overline{\int_G V_g h (y) f'( \pi(y) g) \; d\mu_G (y)} \\
   &= \int_G \overline{V_g h (y)} f (\pi(x) \pi(y) g ) \; d\mu_G(y)
    =
      \int_G  \sigma(y, y^{-1}) V_h g (y^{-1}) \overline{\sigma(x,y)} f(\pi(x y) g) \; d\mu_G (y) \\
   &=
      \int_G \sigma(xy, y^{-1}) V_h g(y^{-1}) V_g f (xy) \; d\mu_G (y)
    = \int_G \sigma(z, z^{-1} x) V_h g(z^{-1} x) V_g f (z) \; d\mu_G (z) \\
   &= V_g f \ast_\sigma V_h g (x),
 \end{align*}
where the sixth step used that
\(
  \sigma(x y, y^{-1})
  = \overline{\sigma(x,y)} \sigma(x, y y^{-1}) \sigma(y, y^{-1})
  = \sigma(y, y^{-1}) \overline{\sigma(x,y)}
\)
for $x,y \in G$.
Setting $x = e_G$, the above calculations also show that
$\langle f, h \rangle  = \langle V_g f, V_g h \rangle_{L^{\infty}_{1/w}, L^1_w}$.

(iv)
For $F \in L^{\infty}_{1/w} (G)$, define
\[
  V_g^* F :
  \Reservoir_w \to \mathbb{C}, \quad
  h \mapsto \int_G F(y) \langle \pi(y) g, h \rangle_{\Hpi} \; d\mu_G (y),
\]
which is well-defined since
\begin{align*}
  \int_G |F(y)| |\langle \pi (y) g, h \rangle_{\Hpi} | \; d\mu_G (y)
  &\leq \int_G \frac{1}{w(y)} |F(y)| w(y) |V_g h (y)| \; d\mu_G (y) \\
  &\leq \| F \|_{L^{\infty}_{1/w}} \| V_g h \|_{L^1_w}
      = \| F \|_{L^{\infty}_{1/w}} \| h \|_{\Hw}
   <    \infty.
\end{align*}
In particular, this shows that $V^*_g : L^{\infty}_{1/w} (G) \to \Reservoir_w$
is a bounded linear map.

If $F \in L^{\infty}_{1/w} (G)$ satisfies $F = F \ast_{\sigma} V_g g$,
then \Cref{lem:basic_cocycle} shows
\[
  \langle \pi(y) g, \pi(x) g \rangle_{\Hpi}
  = V_g [\pi(y) g](x)
  = \twisttranslationL{y} [V_g g] (x)
\]
and hence
\begin{align*}
  V_g (V^*_g F) (x)
  &= \int_G F(y) \langle \pi(y) g, \pi(x) g \rangle_{\Hpi} \; d\mu_G (y) \\
  &= \int_G F(y) \twisttranslationL{y} [V_g g] (x) \; d\mu_G (y) \\
  &= F \ast_{\sigma} V_g g (x)
   = F(x)
\end{align*}
for any $x \in G$.

(v)
Let $f \in \Reservoir_w$.
The estimate $\| V_g f \|_{L^{\infty}_{1/w}} \lesssim \| f \|_{\Reservoir_w}$
follows from Part~(i).
For the converse, set $F := V_g f \in L^{\infty}_{1/w} (G)$,
so that $F = F \ast_{\sigma} V_g g$ by Part~(ii).
Then, by Part~(iv), $h := V_g^* F \in \Reservoir_w$
with $\| h \|_{\Reservoir_w} \lesssim \| F \|_{L^{\infty}_{1/w}}$, and $V_g h = F = V_g f$.
Since $V_g : \Reservoir_w \to L^{\infty}_{1/w} (G)$ is injective by Part~(iii),
it follows that $f = h$, and thus
$\| f \|_{\Reservoir_w} \lesssim \| F \|_{L^{\infty}_{1/w}} = \| V_g f \|_{L^{\infty}_{1/w}}$,
where the implied constant is independent of $f \in \Reservoir_w$.
\end{proof}

\section{Coorbit spaces}
\label{sec:def_coorbit}

Throughout this section, $Y \subseteq L^0 (G)$ will be a solid quasi-Banach function space on $G$
and $w$ will be a $p$-weight for some $p \in (0,1]$.
We also assume that $Y$ is translation-invariant
and that $Y$ is $L^p_w$-compatible in the sense of Definition~\ref{def:compatible}.

The $\sigma$-representation $(\pi, \Hpi)$ will be assumed to be $p$-integrable
in the sense that for
\[
  \Bwp
  = \big\{ g \in \Hpi : V_g g \in \wienerSt{L^p_w} \big\},
\]
there exists an \emph{admissible} $g \in \Bwp$.
Note that $\Bwp \subseteq \Aw$,
since $\wienerSt{L^p_w} \hookrightarrow \wienerL{L_w^p} \hookrightarrow L^1_w$
by Part~(ii) of Lemma~\ref{lem:AmalgamWeightedLInftyEmbedding}.

The central object of this paper are the coorbit spaces defined next.

\begin{definition}\label{def:coorbit}
Let $Y$ be $L^p_w$-compatible and let $g \in \Bwp$ be an admissible vector.
The associated coorbit space $\Co(Y) = \Co_g (Y)$ is defined as the collection
\begin{equation}
  \Co_{g} (Y)
  := \big\{ f \in \Reservoir_w(g) : V_g f \in \wienerL{Y} \big\}
  \label{eq:CoorbitDefinition}
\end{equation}
and equipped with the quasi-norm
$\| f \|_{\Co_g(Y)} := \| V_g f \|_{\wienerL{Y}}$.
\end{definition}

The following result provides basic properties of the coorbit spaces just defined.
In particular, it provides a sufficient condition under which the coorbit space is independent of the choice of defining vector.

\begin{proposition}\label{prop:coorbit_basic}
  Let $Y$ be $L^p_w$-compatible and let $g \in \Bwp$ be an admissible vector.
  \begin{enumerate}[label=(\roman*)]
    \item If $h \in \Bwp$ is an admissible vector such that $V_h g \in \wienerSt{L^p_w}$,
          then $\Hw(g) = \Hw(h)$ and $\Reservoir_w(g) = \Reservoir_w(h)$, as well as
          $\Co_g (Y) = \Co_h (Y)$
          with $\| \cdot \|_{\Co_g(Y)} \asymp \| \cdot \|_{\Co_h (Y)}$.

    \item The space $\Co_g(Y)$ is a $\pi$-invariant quasi-Banach space
          with $\Co_g(Y) \hookrightarrow \Reservoir_w$,
          and $\| \cdot \|_{\Co_g(Y)}$ is a $p$-norm.
  \end{enumerate}
\end{proposition}

\begin{proof}
(i) By assumption, $V_h g \in \wienerSt{L^p_w}$.
Since $|V_g h| = |V_h g|^{\vee}$ and $w$ is a $p$-weight, it follows that also $V_g h \in \wienerSt{L^p_w}$.
By Part~(ii) of Lemma~\ref{lem:AmalgamWeightedLInftyEmbedding}, this implies that $V_g h, V_h g \in L^1_w$.
Therefore, an application of \Cref{lem:basic_Hw} yields that $\Hw(g) = \Hw(h)$,
which implies that $\Reservoir_w(g) = \Reservoir_w(h)$,
and thus it remains to show that $\| \cdot \|_{\Co_g(Y)} \asymp \| \cdot \|_{\Co_h (Y)}$.

For $f \in \Co_g (Y)$, the reproducing formula \eqref{eq:repro_extended} gives
$V_h f = V_g f \ast_{\sigma} V_h g$.
The space $Y$ is assumed to be $L^p_w$-compatible, and therefore
$\wienerL{Y} \ast \wienerSt{L^p_w} \hookrightarrow \wienerL{Y}$, which yields that
\[
  \| f \|_{\Co_h (Y)}
  \leq  \big\| |V_g f | \ast |V_h g| \big\|_{\wienerL{Y}}
  \lesssim \| V_g f \|_{\wienerL{Y}} \| V_h g \|_{\wienerSt{L^p_w}}
  \lesssim \| f \|_{\Co_g (Y)}.
\]
Since also $V_g h \in \wienerSt{L^p_w}$, interchanging the role of $g$ and $h$  yields
$\| \cdot \|_{\Co_g (Y)} \lesssim \| \cdot \|_{\Co_h (Y)}$.

(ii)
The $\pi$-invariance follows directly from Part~(ii) of Lemma~\ref{lem:basic_Rw}
and the translation-invariance of $Y$, which implies the 
translation-invariance of $\wienerL{Y}$.
To show that the embedding $\Co_g (Y) \hookrightarrow \Reservoir_w$, let $f \in  \Co_g (Y)$.
Since $Y$ is assumed to be $L^p_w$-compatible, we have,
in particular, $\wienerL{Y} \hookrightarrow L^{\infty}_{1/w}$, and thus
$\| V_g f \|_{L^{\infty}_{1/w}} \lesssim \| V_g f \|_{\wienerL{Y}} = \| f \|_{\Co(Y)}$.
Since the mapping $f \mapsto \| V_g f \|_{L^{\infty}_{1/w}}$ defines an equivalent norm
on $\Reservoir_w$ by Lemma~\ref{lem:basic_Rw}, it follows immediately
that $\| f \|_{\Reservoir_w} \lesssim \| f \|_{\Co_g (Y)}$.

To prove the completeness of $\Co_g(Y)$, let $(f_n)_{n \in \mathbb{N}}$ be a Cauchy sequence
in $\Co_g(Y)$.
Then the embedding $\Co_g(Y) \hookrightarrow \Reservoir_w$ yields
that $(f_n)_{n \in \mathbb{N}}$ is Cauchy in $\Reservoir_w$,
whence convergent to some $f \in \Reservoir_w$.
In particular, this implies $V_g f_n (x) \to V_g f (x)$ for $x \in G$ as $n \to \infty$.
On the other hand, the sequence $(V_g f_n)_{n \in \mathbb{N}}$ is Cauchy in $\wienerL{Y}$,
hence converging to some $F \in \wienerL{Y}$.
Since $\wienerL{Y} \hookrightarrow L^{\infty}_{1/w}(G)$ as $Y$ is $L^p_w$-compatible,
it also follows that $V_g f_n \to F$ in $L^{\infty}_{1/w}(G)$.
Thus $F = V_g f$.
Since $F \in \wienerL{Y}$, this means that $f \in \Co_g(Y)$, and
\[
  \| f - f_n \|_{\Co_g(Y)}
  = \| V_g f_n - F \|_{\wienerL{Y}}
  \to 0,
  \quad \text{as} \quad n \to \infty.
\]
The $p$-norm properties of $\| \cdot \|_{\Co_g(Y)}$ follow easily
from those of $\| \cdot \|_{\wienerL{Y}}$.
\end{proof}

\begin{remark} Let $g, h \in \Bwp$ be admissible vectors.
\begin{enumerate}[label=(\alph*)]
  \item In part (i) of \Cref{prop:coorbit_basic}, the assumption $V_h g \in \wienerSt{L^p_w}$
        is essential for the validity of the claim  $\Co_g (Y) = \Co_h (Y)$
        in the case of reducible representations.
        More precisely, as mentioned in \Cref{rem:reducible_vgh},
        the paper \cite{fuehr2015coorbit} provides an example of a reducible representation
        $\pi$ admitting admissible vectors $g, h \in \Bw$ satisfying $V_h g \notin L^1_w (G)$,
        cf.\ \cite[\S 2.1]{fuehr2015coorbit}.
        Since $\Co_g (L^1_w(G)) = \Hw(g)$ and $\Co_h (L^1_w(G)) = \Hw(h)$ by \Cref{lem:coincidence},
        this shows that $g \in \Co_g (L^1_w(G))$, but $g \notin \Co_h (L^1_w(G))$.
        As such, one could have $\Co_g (L^1_w(G)) \neq \Co_h (L^1_w(G))$ for reducible representations.
        For \emph{irreducible} representations, on the other hand, the assumption $V_h g \in  \wienerSt{L^p_w}$
        is not needed to guarantee the coincidence $\Co_g (Y) = \Co_h (Y)$;
        see \Cref{prop:independence_irreducible}.

  \item In what follows, we will often simply write $\Co(Y)$ for $\Co_g(Y)$
        whenever the precise choice of defining vector $g \in \Bwp$ does not play an important role,
        or if the choice of $g$ is clear from the context.
        The reader should remember, however, that the space $\Co(Y)$ in general
        does depend on the choice of $g$.
\end{enumerate}
\end{remark}

The last result of this subsection is the following simple consequence of Lemma~\ref{lem:basic_Rw}, which is often helpful
in concrete settings.

\begin{corollary}\label{cor:independence_reservoir}
  Let $Y$ be $L^p_w$-compatible,
  let $\mathcal{S} \hookrightarrow \Hpi$ be a $\pi$-invariant topological vector space
  and let $g \in \mathcal{S} \cap \Bwp$ be admissible.
  Suppose that $\mathcal{S} \hookrightarrow \Hw(g)$ and that the reproducing formula
  \begin{align} \label{eq:repro_S}
   \langle f, h \rangle
   = \int_{G}
       \langle f, \pi(x) g \rangle_{\mathcal{S}^\ast, \mathcal{S}}
        \langle \pi(x) g, h \rangle_{\Hpi}
     \; d\mu_G (x)
  \end{align}
  holds for all $f \in \mathcal{S}^{\ast}$ (the \emph{anti}-dual space of $\mathcal{S}$)
  and $h \in \mathcal{S}$, where
  \(
    \langle \cdot, \cdot \rangle_{\mathcal{S}^\ast, \mathcal{S}}
    = \langle \cdot , \cdot \rangle :
    \mathcal{S}^{\ast} \times \mathcal{S} \to \mathbb{C}
  \)
  denotes the anti-dual pairing.

  Then
  \[
    \Co_g (Y)
    = \big\{ f \in \Reservoir_w : V_g f \in \wienerL{Y} \big\}
    = \big\{ f \in \mathcal{S}^{\ast} : V_g f \in \wienerL{Y} \big\}.
  \]
  in the sense that the restriction map
  \[
    \Co_g (Y) \to \bigl\{ f \in \mathcal{S}^\ast \colon V_g f \in \wienerL{Y} \bigr\}, \quad
    f \mapsto f|_{\mathcal{S}}
  \]
  is a bijection.
  Here, the coefficient transform $V_g f$ for $f \in \mathcal{S}^\ast$ is defined by
  $V_g f (x) = \langle f, \pi(x) g \rangle$.
\end{corollary}

\begin{proof}
Since $\mathcal{S} \hookrightarrow \Hw$, it follows
that the restriction map $f \mapsto f|_{\mathcal{S}}$ is well-defined
from $\Reservoir_w$ into $\mathcal{S}^{\ast}$ and hence also from $\Co_g(Y)$
into $\bigl\{ f \in \mathcal{S}^\ast \colon V_g f \in \wienerL{Y} \bigr\}$.
Furthermore, since we have $\pi(G) g \subseteq \mathcal{S}$ and since $\pi(G) g \subseteq \Hw$
is norm-dense by \Cref{lem:basic_Hw}, this restriction map is injective.

To show surjectivity, let $f \in \mathcal{S}^{\ast}$
be such that $F := V_g f \in \wienerL{Y}$.
Then also $F \in L^{\infty}_{1/w} (G)$ since $Y$ is $L^p_w$-compatible
and hence $\wienerL{Y} \hookrightarrow L_{1/w}^\infty (G)$.
The assumption \eqref{eq:repro_S} yields, in particular,
that $F = F \ast_{\sigma} V_g g$.
Hence, by Part~(iv) of Lemma~\ref{lem:basic_Rw},
there exists $\widetilde{f} := V_g^* F \in \Reservoir_w$
such that $V_g \widetilde{f} = F = V_g f \in \wienerL{Y}$.
Therefore, $\widetilde{f} \in \Co_g (Y)$,
and $\widetilde{f}|_{\mathcal{S}} = f$ by \eqref{eq:repro_S}
and Part~(iii) of \Cref{lem:basic_Rw}.
\end{proof}

\section{Analyzing and better vectors}

This section provides (somewhat more) explicit descriptions of the coorbit spaces
$\Co(Y)$ for the specific choices $Y$ being $L_w^1(G), L^2(G)$ or $L_{1/w}^\infty(G)$.
In particular, this section completely resolves the question of the
relation between the sets of so-called ``analyzing vectors''  and ``better vectors'';
see \Cref{prop:AnalyzingVectorsAreBetterOnlyIfIN}.
To the best of our knowledge, this question was an open problem in the literature even for Banach spaces;
see for instance \cite[Section 6.3]{groechenig1991describing}.

As defined in \Cref{eq:CoorbitDefinition}, the coorbit space $\Co_g(Y)$
consists of those $f \in \Reservoir_w$ for which $V_g f \in \wienerL{Y}$.
However, in the literature considering coorbit spaces associated with genuine
Banach function spaces $Y$, one usually defines
$\Co_g (Y) = \{ f \in \Reservoir_w \colon V_g f \in Y \}$.
The following proposition identifies a sufficient condition regarding $Y$
under which both definitions coincide; in particular, it is applicable to all Banach function spaces.

\begin{proposition}\label{prop:selfimproving}
  Let $Y$ be a solid, translation-invariant quasi-Banach function space
  with $p$-norm $\| \cdot \|_Y$ (for some $p \in (0,1]$) and
  such that $Y \ast \wienerSt{L^p_w} \hookrightarrow Y$
  and $\wienerL{Y} \hookrightarrow L_{1/w}^\infty (G)$
  for a $p$-weight $w$.
  Then $Y$ is $L_w^p$-compatible, and if $g \in \Bw^p$ is an admissible vector, then
  \begin{equation}
    \Co_g (Y)
    := \big\{
         f \in \Reservoir_w : V_g f \in \wienerL{Y}
       \big\}
     = \big\{ f \in \Reservoir_w : V_g f \in Y \big\},
    \label{eq:CoorbitWithoutWiener}
  \end{equation}
  with $\| f \|_{\Co_g (Y)} \asymp \|V_g f \|_{Y}$.

  In addition, if $\wienerL{Y} \hookrightarrow L^2(G)$, then (up to canonical identifications),
  it holds that
  \begin{equation}
    \Co_g (Y)
    = \big\{
        f \in \Hpi
        \colon
        V_g f \in \wienerL{Y}
      \big\}
    = \big\{
        f \in \Hpi
        \colon
        V_g f \in Y
      \big\}
    .
    \label{eq:CoorbitWithHilbertAsReservoir}
  \end{equation}
\end{proposition}

\begin{proof}
  Using the estimate \eqref{eq:ConvolutionMaximalFunctionEstimate},
  it holds for $F \in \wienerL{Y}$ and $H \in \wienerSt{L_w^p}$ that
  \begin{align*}
    \| F \ast H \|_{\wienerL{Y}}
    & = \big\| \maxL [F \ast H] \big\|_{\wienerL{Y}}
      \leq \big\| |F| \ast \maxL H \big\|_{Y} \\
    & \lesssim \big\| |F| \big\|_Y  \| \maxL H \|_{\wienerSt{L_w^p}}
      \lesssim \big\| F \big\|_{\wienerL{Y}}  \| H \|_{\wienerSt{L_w^p}}
    .
  \end{align*}
  Hence, $Y$ is $L_w^p$-compatible.

  Assume $g \in \Hpi$ is admissible and $V_g g \in \wienerSt{L^p_w} \subseteq L^1_w$. An application of
  Lemma~\ref{lem:basic_Rw} yields that $V_g f = V_g f \ast_{\sigma} V_g g$
  for $f \in \Reservoir_w$.
  Therefore, if $V_g f \in Y$, then the estimate \eqref{eq:ConvolutionMaximalFunctionEstimate}
  shows that
  \begin{align*}
    \| V_g f \|_{\wienerL{Y}}
    \leq \| \maxL (|V_g f| \ast |V_g g|) \|_{Y}
    \lesssim \| V_g f \|_Y \| \maxL [V_g g] \|_{\wienerSt{L_w^p}}
    \lesssim \|V_g f \|_Y \|V_g g\|_{\wienerSt{L_w^p}},
  \end{align*}
  wand hence $\| f \|_{\Co_g (Y)} \lesssim \| V_g f \|_{Y}$.
  The estimate $\|V_g f \|_{Y} \leq \| f \|_{\Co_g (Y)}$ is immediate.

  \medskip{}

  To prove \Cref{eq:CoorbitWithHilbertAsReservoir},
  it suffices because of $\Hpi \hookrightarrow \Reservoir_w$
  to show that each $f \in \Co_g (Y)$ satisfies $f \in \iota(\Hpi)$, with the canonical embedding $\iota : \Hpi \to \Reservoir_w$.
  For clarity, we denote by $V_g^e$ (resp. $V_g$) the coefficient transform on $\Reservoir_w$ (resp. $\Hpi$).
  Let $f \in \Co_g (Y)$ be arbitrary and note
  that $F := V_g^e f \in \wienerL{Y} \hookrightarrow L^2(G)$
  satisfies $F = F \ast_\sigma V_g g$; see \Cref{lem:basic_Rw}.
  With the usual Hilbert-space adjoint $V_g^\ast : L^2(G) \to \Hpi$ of $V_g : \Hpi \to L^2(G)$,
  let $\widetilde{f} := V_g^\ast F \in \Hpi$.
  Then, for any $x \in G$,
  \begin{align*}
    V_g^e [\iota \widetilde{f}] (x)
    & = \langle \iota \widetilde{f}, \pi(x) g \rangle_{\Reservoir_w, \Hw}
      = \langle \widetilde{f}, \pi(x) g \rangle_{\Hpi}
      = \langle V_g^\ast F, \pi(x) g \rangle_{\Hpi}
      = \langle F, V_g [\pi(x) g] \rangle_{L^2} \\
    & = \int_{G}
          F(y)
          \overline{
            \langle
              \pi(x) g,
              \pi(y) g
            \rangle
          }
        \, d \mu_G(y)
      = \int_G
          F(y)
          V_g [\pi(y) g] (x)
        \, d \mu_G(y) \\
    & =
        \int_G
          F(y)
          \twisttranslationL{y} [V_g g] (x)
        \, d \mu_G(y)
      = (F \ast_\sigma V_g g) (x)
      = F(x)
      = V_g^e f (x)
      .
  \end{align*}
Hence, $V_g^e (\iota \widetilde{f} \,) = V_g^e f$.
  Since $V_g^e : \Reservoir_w \to L_{1/w}^\infty (G)$ is injective by \Cref{lem:basic_Rw},
  this implies $f = \iota \widetilde{f} \in \iota(\Hpi)$, as required.
\end{proof}

\begin{remark}
  As stated above, \Cref{prop:selfimproving} could also apply for \emph{quasi}-Banach spaces,
  but the issue is that when $Y$ is not a Banach space but rather a quasi-Banach space,
  then the convolution relation $Y \ast \wienerSt{L_w^p} \hookrightarrow Y$ is rarely valid.
  Nevertheless, the identity \eqref{eq:CoorbitWithoutWiener} is known to hold
  for a variety of quasi-Banach spaces;
  see, e.g., \cite{rauhut2007coorbit, galperin2004time, koppensteiner2023anisotropic, koppensteiner2023anisotropic2,liang2012new}.
  It is an open problem whether \eqref{eq:CoorbitWithoutWiener} holds
  in general for quasi-Banach spaces; see also \Cref{thm:local_wavelet}.
\end{remark}

\begin{lemma}\label{lem:coincidence}
Let $w : G \to [1,\infty)$ be a $p$-weight for some $p \in (0,1]$
and let $g \in \Bw^p$ be admissible.
Then each of the spaces $L_w^1(G), L^2(G), L^\infty_{1/w}(G)$ and $\wienerR{L_w^1}$
 is $L_w^p$-compatible.
Furthermore, the following identifications hold:
\begin{enumerate}[label=(\roman*)]
  \item $\Co_g (L^1_w(G)) = \Hw(g)$.

  \item $\Co_g (L^2 (G)) = \Hpi$.

  \item $\Co_g (\wienerR{L^1_w}) = \{ f \in \Hpi \colon V_g f \in \wienerSt{L_w^1} \}$.

  \item $\Co_g (L^{\infty}_{1/w}(G)) = \Reservoir_w$.
\end{enumerate}
\end{lemma}

\begin{proof}
As a preparation, suppose that $v, v_0 : G \to (0,\infty)$ are measurable and
such that $v(x y) \leq v(x) v_0 (y)$ and $v(x y) \leq v_0(x) v(y)$ for all $x,y \in G$.
Using the well-known identity $\int_G f(t s) \, d \mu_G (t) = \Delta(s^{-1}) \int_G f(t) \, d \mu_G (t)$
(see, e.g., \cite[Proposition~2.24]{FollandAHA}), it then follows that
\begin{align*}
  \| \translationR{x} f \|_{L_v^p}^p
  & = \int_G
        |f(y x)|^p  [v(y x x^{-1})]^p
      \, d \mu_G(y) \\
  & = \Delta(x^{-1})
      \int_G
        |f(z)|^p  [v(z x^{-1})]^p
      \, d \mu_G(z)
    \leq [v_0(x^{-1})]^p  \Delta(x^{-1})  \| f \|_{L_v^p}^p
    .
\end{align*}
A similar computation for the left-translation shows
\begin{equation}
  \| \translationL{x} f \|_{L_v^p}
  \leq v_0(x)  \| f \|_{L_v^p}
  \qquad \text{and} \qquad
  \| \translationR{x} f \|_{L_v^p}
  \leq v_0(x^{-1})  [\Delta(x^{-1})]^{1/p}  \| f \|_{L_v^p}
  \label{eq:LebesgueModeratelyWeightedTranslationNorm}
\end{equation}
for all $x \in G$ and $f \in L_v^p(G)$, for arbitrary $p \in (0,\infty)$;
these estimates also remain valid for $p = \infty$.
Next, note $w(x) = w(x y y^{-1}) \leq w(x y) w(y^{-1})$ and hence
$\frac{1}{w (x y)} \leq \frac{1}{w(x)} w^{\vee}(y)$;
similar arguments show that also $\frac{1}{w(x y)} \leq w^{\vee}(x) \frac{1}{w(y)}$.
Therefore, applying \Cref{eq:LebesgueModeratelyWeightedTranslationNorm} with $(v,v_0) = (w,w)$
or $(v,v_0) = (1,1)$ or $(v,v_0) = (\frac{1}{w}, w^{\vee})$, respectively,
shows that each space $Y \in \{ L_w^1, L^2, L_{1/w}^\infty \}$ is invariant
under left- and right-translations.
More precisely, we see because of $p \in (0,1]$ and by conditions \ref{enu:WeightBoundedBelow}
and \ref{enu:WeightPSymmetric} that
\[
  \| \translationR{x} \|_{L^2 \to L^2}
  \leq [\Delta(x^{-1})]^{1/2}
  \leq \begin{cases}
         [\Delta(x^{-1})]^{1/p} \leq w(x^{-1}) \Delta^{1/p}(x^{-1}) = w(x),
         & \text{if } \Delta(x^{-1}) \geq 1, \\
         1 \leq w(x),
         & \text{if } \Delta(x^{-1}) <    1.
       \end{cases}
\]
and $\| \translationR{x} \|_{L_{1/w}^\infty \to L_{1/w}^\infty} \leq w^{\vee}(x^{-1}) = w(x)$.

We now verify that \Cref{prop:selfimproving} is applicable for each of the spaces in question.

First, note that any $Y \in \{ L_w^1(G), L^2(G), L^{\infty}_{1/w}(G), \wienerR{L_w^1} \}$
is a solid Banach space and thus, in particular, a $p$-normed solid quasi-Banach function space.
We saw above that each of the spaces $Y \in \{ L_w^1(G), L^2(G), L^{\infty}_{1/w}(G)\}$
is translation-invariant.
By the properties of Wiener amalgam spaces collected in \Cref{sec:AmalgamPrelims},
the same then also holds for $Y = \wienerR{L_w^1}$.

Second, it will be shown that  $Y \ast \wienerSt{L_w^p} \hookrightarrow Y$.
To handle the case $Y = L_w^1$, we recall from \cite[§3.7]{ReiterClassicalHarmonicAnalysis}
that the convolution relation $L_w^1 \ast L_w^1 \hookrightarrow L_w^1$ holds,
since $w$ is submultiplicative.
Next, note that
\(
  \wienerSt{L_w^p}
  = \wienerR{\wienerL{L_w^p}}
  \hookrightarrow \wienerL{L_w^p}
  \hookrightarrow L_w^1
  ,
\)
where the last step follows from \Cref{lem:AmalgamWeightedLInftyEmbedding}.
Overall, we thus see that
$L_w^1 \ast \wienerSt{L_w^p} \hookrightarrow L_w^1 \ast L_w^1 \hookrightarrow L_w^1$,
as required.
For $F_1 \in \wienerR{L_w^1}$ and $F_2 \in \wienerSt{L_w^p}$,
\Cref{eq:ConvolutionMaximalFunctionEstimate} shows that
\begin{align*}
  \| F_1 \ast F_2 \|_{\wienerR{L_w^1}}
  & = \big\| \maxR[F_1 \ast F_2] \big\|_{L_w^1}
    \leq \big\| \maxR F_1 \ast |F_2| \big\|_{L_w^1} \\
  & \hspace*{-0.95cm} \lesssim
    \| \maxR F_1 \|_{L_w^1}  \| F_2 \|_{\wienerSt{L_w^p}}
    = \| F_1 \|_{\wienerR{L_w^1}}  \| F_2 \|_{\wienerSt{L_w^p}}
  ,
\end{align*}
which proves that $\wienerR{L_w^1} \ast \wienerSt{L_w^p} \hookrightarrow \wienerR{L_w^1}$,
thereby settling the case $Y = \wienerR{L_w^1}$.
If $Y \in \{ L^2, L_{1/w}^\infty \}$ recall from above
that $\| R_x \|_{Y \to Y} \leq w(x)$; hence, \Cref{lem:SolidBanachSpaceConvolutionRelation}
shows that $Y \ast [L_w^1]^{\vee} \hookrightarrow Y$.
Since $\wienerSt{L_w^p} \hookrightarrow \wienerR{L_w^p} \hookrightarrow [L_w^1]^{\vee}$
by \Cref{lem:AmalgamWeightedLInftyEmbedding}, this implies that
$Y \ast \wienerSt{L_w^p} \hookrightarrow Y \ast [L_w^1]^{\vee} \hookrightarrow Y$
for $Y \in \{ L^2, L_{1/w}^\infty \}$.

Lastly, it will be shown that $\wienerL{Y} \hookrightarrow L_{1/w}^\infty(G)$
Since $w \geq 1$, we have $L_w^\infty \hookrightarrow L_{1/w}^\infty$.
Thus, \Cref{lem:AmalgamWeightedLInftyEmbedding} implies that
\(
  \wienerL{\wienerR{L_w^1}}
  \hookrightarrow \wienerL{L_w^1}
  \hookrightarrow L_w^\infty
  \hookrightarrow L_{1/w}^\infty
  ,
\)
which proves the conclusion for $Y \in \{ L_w^1, \wienerR{L_w^1} \}$.
\Cref{lem:AmalgamWeightedLInftyEmbedding} also shows
$\wienerL{L^2} \hookrightarrow L^\infty \hookrightarrow L_{1/w}^\infty$,
where the last step again used that $w \geq 1$.
We trivially have $\wienerL{L_{1/w}^\infty} \hookrightarrow L_{1/w}^\infty$.

Overall, \Cref{prop:selfimproving} shows that each
$Y \in \{ L_w^1, L^2, L_{1/w}^\infty, \wienerR{L_w^1} \}$ is $L_w^p$-compatible.
We now prove the remaining parts of the lemma.

\medskip{}

(i)
\Cref{lem:AmalgamWeightedLInftyEmbedding} shows that
$\wienerL{L_w^1} \hookrightarrow L_w^2 \hookrightarrow L^2$, since $w \geq 1$.
Therefore, \Cref{prop:selfimproving} shows that
\(
  \Co_g (L_w^1)
  = \{ f \in \Hpi \colon V_g f \in L_w^1 \}
  = \Hw
  .
\)

(ii)
We trivially have $\wienerL{L^2} \hookrightarrow L^2$, and by admissibility of $g$
it holds that $V_g f \in L^2(G)$ for every $f \in \Hpi$.
Therefore, \Cref{prop:selfimproving} shows
\(
  \Co_g (L^2)
  = \{ f \in \Hpi \colon V_g f \in L^2(G) \}
  = \Hpi
  .
\)

(iii)
We have $\wienerR{L_w^1} \hookrightarrow L_w^1 \hookrightarrow L^2$ (see the proof of \emph{(i)}).
Therefore, \Cref{prop:selfimproving} shows that
\(
  \Co_g (\wienerR{L_w^1})
  = \{ f \in \Hpi \colon V_g f \in \wienerL{\wienerR{L_w^1}} \}
  = \{ f \in \Hpi \colon V_g f \in \wienerSt{L_w^1} \}
  ,
\)
as claimed.

(iv)
By \Cref{lem:basic_Rw}, we have $V_g f \in L_{1/w}^\infty (G)$ for all $f \in \Reservoir_w$.
Therefore, \Cref{prop:selfimproving} shows
\(
  \Co_g (L_{1/w}^\infty)
  = \{ f \in \Reservoir_w \colon V_g f \in L_{1/w}^\infty \}
  = \Reservoir_w
  ,
\)
as claimed.
\end{proof}

Lastly, it will be shown that the important auxiliary spaces
of so-called ``analyzing vectors'' and ``better vectors''
(cf.\ \cite{feichtinger1989banach1, feichtinger1989banach2})
coincide precisely if the group $G$ is an IN group.
That $G$ being IN is sufficient is simple and well-known
(see, e.g., \cite[Lemma~7.2]{feichtinger1989banach2}), but its necessity remained open;
see also \cite[Section~6.3]{groechenig1991describing}.

\begin{proposition}\label{prop:AnalyzingVectorsAreBetterOnlyIfIN}
Let $w : G \to [1,\infty)$ be a $p$-weight and let $g \in \Bw^p$ be admissible.
The following assertions are equivalent:
\begin{enumerate}[label=(\roman*)]
    \item The spaces $\Co_g (L^1_w)$ and $\Co_g (\wienerR{L^1_w})$
          coincide, i.e., $\Co_g (L^1_w) = \Co_g (\wienerR{L^1_w})$.

    \item The group $G$ is an \emph{IN group}, i.e., there exists a relatively compact
          unit neighborhood $U \subseteq G$ such that $x^{-1} U x = U$ for all $x \in G$.
\end{enumerate}
\end{proposition}

The following result is an essential ingredient
in the proof of Proposition~\ref{prop:AnalyzingVectorsAreBetterOnlyIfIN}.
Its proof was provided to us by T.~Tao \cite{MathOverflowTaoAnswer};
hence, no originality is claimed.
The complete argument is included here, with a few added details.

\begin{proposition}\label{prop:INCharacterizedByTwosidedMeasure}
  If there exists an open, relatively compact unit neighborhood $Q \subseteq G$  satisfying
  \[
    \haarMeasure(Q x Q) \lesssim 1
    \qquad \text{for all } x \in \group,
  \]
  then $\group$ is an IN group.
\end{proposition}

\begin{proof} Throughout the proof, let $C > 0$ be a constant
  satisfying $\mu_G (QxQ) \leq C$ for all $x \in G$.

  \medskip{}

  \textbf{Step~1.} This step shows that $G$ is unimodular.
  If not, then $\Delta(x_0) \neq 1$ for some $x_0 \in \group$.
  Replacing $x_0$ by $x_0^{-1}$ if necessary, it may be assumed that $\Delta(x_0) > 1$.
  Then, by definition of the modular function,
  \[
    C
    \geq \haarMeasure(Q x_0^n Q)
    \geq \haarMeasure(Q x_0^n)
    =    \Delta(x_0^n) \, \haarMeasure(Q)
    =    (\Delta(x_0))^n \, \haarMeasure(Q)
    \rightarrow \infty \quad \text{as} \quad n \to \infty,
  \]
  which is a contradiction.

  \medskip{}

  \textbf{Step~2.}
  For the remainder of the proof, the space $L^2(G) = L^2(G;\R)$
  will be considered as a vector space over $\mathbb{R}$.
  For $U := Q \cap Q^{-1}$, this step will show that there exists $\delta > 0$ such that
  \begin{align} \label{eq:claim_IN}
      \langle \indicator_{x U^2 x^{-1}}, \indicator_{U^2} \rangle_{L^2} \geq \delta > 0,
      \quad x \in G.
  \end{align}

  For this, note by unimodularity of $\group$ that
  $\haarMeasure(U \cdot x U x^{-1}) = \haarMeasure(U x U) \leq \haarMeasure(Q x Q) \leq C$
  for all $x \in \group$.
  Furthermore, $ (\indicator_U \ast \indicator_{x U x^{-1}})^{-1} (\CC \setminus \{0\}) \subseteq U \cdot x U x^{-1}$.
  Hence, using Tonelli's theorem, it follows
  \begin{align*}
    \big( \haarMeasure(U) \big)^2
    & = \haarMeasure(U)  \haarMeasure(x U x^{-1})
      = \int_{\group}
          \indicator_U (z)
          \int_{\group}
            \indicator_{x U x^{-1}} (z^{-1} y)
          \dd{y}
        \dd{z} \\
    & = \int_{\group}
          (\indicator_U \ast \indicator_{x U x^{-1}}) (y)
        \dd{y}
      \leq \haarMeasure(U \cdot x U x^{-1})
            \| \indicator_U \ast \indicator_{x U x^{-1}} \|_{L^\infty} \\
    & \leq C  \| \indicator_U \ast \indicator_{x U x^{-1}} \|_{L^\infty}
    .
  \end{align*}
  Thus, there exists $y = y(x) \in \group$ satisfying
  \begin{align*}
    0
    < \delta
    := \frac{(\haarMeasure(U))^2}{C}
    & \leq (\indicator_U \ast \indicator_{x U x^{-1}})(y^{-1})
      = \int_{\group}
          \indicator_U (z)
          \indicator_{x U x^{-1}} (z^{-1} y^{-1})
        \dd{z} \\
    & = \int_{\group}
          \indicator_U (y^{-1} w)
          \indicator_{x U x^{-1}} (w^{-1})
        \dd{w}
      = \haarMeasure(y U \cap x U x^{-1}),
  \end{align*}
  where the change-of-variables $w = y z$
  and the identity $(x U x^{-1})^{-1} = x U x^{-1}$ were used.

  If $W \subseteq \group$ is open and satisfies $\haarMeasure(W) \geq \delta$,
  then there exists some $w \in W$, which implies that
  \[ \haarMeasure(W^{-1} W) \geq \haarMeasure(w^{-1} W) = \haarMeasure(W) \geq \delta.\]
  Applying this observation to $W = W_y := y U \cap x U x^{-1}$ and using that
  $W_y^{-1} W_y \subseteq U^2 \cap x U^2 x^{-1}$,  it follows that
  \[
    \langle \indicator_{x U^2 x^{-1}}, \indicator_{U^2} \rangle_{L^2}
    = \haarMeasure\bigl(U^2 \cap x U^2 x^{-1}\bigr)
    \geq \haarMeasure(W_y^{-1} W_y)
    \geq \delta,
    \qquad  x \in \group ,
  \]
  which establishes the asserted claim \eqref{eq:claim_IN}.

  \medskip{}

  \textbf{Step~3.}
  Set
  \(
    \Omega
    := \bigl\{ \indicator_{x U^2 x^{-1}} \colon x \in \group \bigr\}
    \subseteq L^2(G) .
  \)
  Since $\indicator_{x U^2 x^{-1}} = L_x R_x \indicator_{U^2}$,
  it is easy to see $L_x R_x \Omega \subseteq \Omega$ for all $x \in \group$.
  Furthermore, since $U = U^{-1}$, each $F \in \Omega$ satisfies $F^{\vee} = F$,
  $0 \leq F \leq 1$ almost everywhere, and $\langle F, \indicator_{U^2} \rangle_{L^2} \geq \delta > 0$
  (cf.\ Step~2).
  Note that all of these properties are preserved under convex combinations
  and under limits in $L^2$;
  this uses that $\group$ is unimodular so that $F \mapsto F^{\vee}$ is a bounded
  linear map on $L^2$ and that if $F_n \to F$ in $L^2$,
  then $F_{n_\ell} \to F$ almost everywhere for a suitable subsequence.
  Hence, letting $\Sigma := \overline{\operatorname{conv} \Omega} \subseteq L^2$
  denote the closed convex hull of $\Omega$, it follows that $L_x R_x \Sigma \subseteq \Sigma$
  for all $x \in \group$ and that each $F \in \Sigma$ satisfies $F = F^{\vee}$ and $0 \leq F \leq 1$
  almost everywhere, and finally $\langle F, \indicator_{U^2} \rangle_{L^2} \geq \delta > 0$.

  By the Hilbert projection theorem (see, e.g., \cite[Theorem~12.3]{RudinFA}),
  there exists a unique $F_0 \in \Sigma$ satisfying
  $\| F_0 \|_{L^2} \leq \| H \|_{L^2}$ for all $H \in \Sigma$.
  Since $\group$ is unimodular, the operator $L_x R_x : L^2 \to L^2$ is unitary;
  hence, $L_x R_x F_0 \in \Sigma$ with $\| L_x R_x F_0 \|_{L^2} = \| F_0 \|_{L^2} \leq \| H \|_{L^2}$
  for all $H \in \Sigma$.
  By the uniqueness of $F_0$, this implies $L_x R_x F_0 = F_0$ for all $x \in \group$.
  Note that $F_0$ is nontrivial since $\langle F_0, \indicator_{U^2} \rangle_{L^2} \geq \delta > 0$.

  \medskip{}

  \textbf{Step~4.}
  Since $\group$ is unimodular, the identity $F \ast H (x) \!=\! \langle F, L_x H^{\vee} \rangle$
  shows that the bilinear map
  \[
    L^2(G) \times L^2(G) \to C_b(\group), \quad
    (F,H) \mapsto F \ast H
  \]
  is well-defined and continuous.
  Since $F \ast H \in C_c(\group)$ for $F,H \in C_c (\group)$, this implies by density
  that $F \ast H \in C_0(\group)$ for all $F,H \in L^2(G)$.
  In particular, $H_0 := F_0 \ast F_0 \in C_0(\group)$, where $F_0$ is as in Step~3.
  Furthermore, since $F_0$ satisfies $F_0 = F_0^{\vee}$
  and $F_0 = \translationL{x} \translationR{x} F_0$,
  it follows that for all $x,y \in \group$, it holds that
  $F_0 (z^{-1} x^{-1} y x) = F_0(x^{-1} y^{-1} x z) = F_0(y^{-1} x z x^{-1})$
  for almost all $z \in \group$.
  Since $\group$ is unimodular, the change-of-variables $w = x z x^{-1}$ therefore shows
  \begin{equation}
    \begin{split}
      H_0(x^{-1} y x)
      & =\! \int_{\group}
              F_0(z)
              F_0(z^{-1} x^{-1} y x) \;
            \dd{z}
        = \int_{\group}
            F_0(z)
            F_0(y^{-1} x z x^{-1}) \;
          \dd{z} \\
      & =\! \int_{\group}
              F_0(x^{-1} w x)
              F_0(y^{-1} w) \;
            \dd{w}
        =
          \int_{\group}
            F_0(w)
            F_0(w^{-1} y) \;
          \dd{w} \\
      & = (F_0 \ast F_0)(y)
        = H_0(y) .
    \end{split}
    \label{eq:INCharacterizationAlmostDone}
  \end{equation}
  where the fourth step used again
  that $F_0 = \translationL{x} \translationR{x} F_0$ and $F_0 = F_0^{\vee}$.

  Lastly, note because of $F_0 \geq 0$ and $F_0 = F_0^{\vee}$ that
  \[H_0(e_{\group}) = \int_{\group} F_0(x) F_0(x^{-1}) \; \dd{x} \!=\! \| F_0 \|_{L^2}^2 > 0.\]
  Since $H_0 \in C_0(\group)$, this implies that
  $V := \bigl\{ x \in \group \colon H_0(x) > \| F_0 \|_{L^2}^2 / 2 \bigr\}$
  is an open, relatively compact unit neighborhood.
  In view of \Cref{eq:INCharacterizationAlmostDone}, it follows that $x V x^{-1} = V$
  for all $x \in \group$.
  Hence, $\group$ is an IN group.
\end{proof}

\begin{proof}[Proof of \Cref{prop:AnalyzingVectorsAreBetterOnlyIfIN}]
Throughout the proof, the identifications
\[
  \Co_g (L^1_w)
  = \{ f \in \Hpi : V_g f \in \wienerL{L_w^1} \}
  \quad \text{and} \quad
  \Co_g (\wienerR{L^1_w})
  = \{ f \in \Hpi : V_g f \in \wienerSt{L^1_w} \}
\]
provided by Lemma~\ref{lem:coincidence} will be used.

\medskip{}

First, suppose that $G$ is an IN group.
Since the spaces $\wienerL{L_w^1}$ and $\wienerSt{L_w^1}$ are independent of the choice of
the neighborhood $Q$ (as $L_w^1$ is left- and right invariant),
and since $\group$ is an IN group, it may be assumed that $x Q x^{-1} = Q$ for all $x \in \group$,
that is, $x Q = Q x$.
This easily implies $\maxL F = \maxR F$ for any measurable $F : \group \to \CC$
and hence $\wienerL{Y} = \wienerR{Y}$ for any solid quasi-Banach function space $Y$ on $\group$.
Therefore,
\[
  \wienerSt{L_w^1}
  = \wienerR{\wienerL{L_w^1}}
  = \wienerL{\wienerL{L_w^1}}
  = \wienerL{L_w^1} ,
\]
which easily implies $\Co_g (L^1_w) = \Co_g (\wienerR{L^1_w})$.

\medskip{}

Second, suppose $\Co_g (\wienerR{L^1_w}) = \Co_g (L^1_w)$.
The inclusion $\iota : \Co_g (\wienerR{L^1_w}) \hookrightarrow \Co_g (L^1_w)$
is clearly bounded and linear.
Since $\Co_g (L^1_w) = \Co_g (\wienerR{L^1_w})$, it follows that $\iota$ is bijective.
By the bounded inverse theorem, this implies that $\iota$ is boundedly invertible,
i.e., $\| \cdot \|_{\Co_g (\wienerR{L^1_w})} \asymp \| \cdot \|_{\Co_g (L^1_w)}$.
In particular, this implies existence of a constant $C_1 > 0$ satisfying
\begin{align*}
    \| \pi(x) g \|_{\Co_g (\wienerR{L^1_w})}
    & \leq C_1  \| \pi(x) g \|_{\Co_g (L^1_w)}
      =    C_1  \big\| V_g [\pi(x) g] \big\|_{\wienerL{L_w^1}}
     =  C_1  \| \twisttranslationL{x} [V_g g] \big\|_{\wienerL{L_w^1}} \\
    & =  C_1  \big\| \translationL{x} [V_g g] \big\|_{\wienerL{L_w^1}}
      = C_1  \big\| \maxL [\translationL{x} (V_g g)] \big\|_{L_w^1}
      =    C_1  \big\| \translationL{x} [ \maxL (V_g g)] \big\|_{L_w^1} \\
    & \leq C_1  w(x)  \big\| \maxL [V_g g] \big\|_{L_w^1}
      =:   C_2  w(x)
  \numberthis \label{eq:AnalyzingBetterBound1}
\end{align*}
for all $x \in \group$.
Since $|V_g g|$ is continuous with $|V_g g (e_\group)| = \| g \|_{\Hpi}^2 > 0$,
there exist $\delta > 0$ and an open, symmetric, relatively compact
unit neighborhood $U \subseteq Q$ satisfying $|V_g g(x)| \geq \delta$ for all $x \in U$.
In particular, this implies that
\begin{align*}
  \maxSt (V_g [\pi(x) g]) (y)
  & = \esssup_{q_1, q_2 \in Q}
        \bigl|V_g [\pi(x) g] (q_1 y q_2)\bigr|
    %= \sup_{q_1, q_2 \in Q}
    %    \bigl|V_g [\pi(x) g] (q_1 y q_2)\bigr| \\
    = \esssup_{q_1, q_2 \in Q}
        \bigl| \twisttranslationL{x} [V_g g] (q_1 y q_2) \bigr| \\
    & =    \sup_{q_1, q_2 \in Q}
             |\translationL{x} [V_g g] (q_1 y q_2)|
    \geq \delta
          \sup_{q_1, q_2 \in Q}
                 \indicator_U (x^{-1} q_1 y q_2) \\
  & = \delta  \indicator_{Q x U Q} (y)
    \geq \delta  \indicator_{Q x Q} (y)
\end{align*}
for all $x,y \in \group$.
Since $w$ is locally bounded (cf.\ \Cref{rem:SubmultiplicativeWeightsLocallyBounded}),
there exists a constant $C_3 = C_3(w,Q) \!>\! 0$ satisfying $w(q) \leq C_3$
for all $q \in Q = Q^{-1}$.
Hence, if $\indicator_{Q x Q}(y) \neq 0$, then $x = q_1^{-1} y q_2^{-1}$
for certain $q_1, q_2 \in Q$ and hence
\(
  w(x)
  \leq w(q_1^{-1}) w(y) w(q_2^{-1})
  \leq C_3^2 \, w(y)
  .
\)
Combining these observations with \eqref{eq:AnalyzingBetterBound1} gives
\begin{align*}
  C_2  w(x)
  \geq \| \pi(x) g \|_{\Co_g (\wienerR{L_w^1})}
  & = \big\| V_g [\pi(x) g] \big\|_{\wienerSt{L_w^1}}
    = \big\| \maxSt (V_g [\pi(x) g]) \big\|_{L_w^1}
    \geq \delta  \| \indicator_{Q x Q} \|_{L_w^1} \\
  & \geq C_3^{-2} \delta  w(x)  \| \indicator_{Q x Q} \|_{L^1}
    =    C_3^{-2} \delta  w(x)  \haarMeasure(Q x Q) .
\end{align*}
It follows therefore that $\haarMeasure(Q x Q) \leq C_2 C_3^2 / \delta =: C$
for all $x \in \group$.
An application of \Cref{prop:INCharacterizedByTwosidedMeasure} shows that $\group$ is an IN group.
\end{proof}

\chapter{Convolution-dominated operators and local spectral invariance}
\label{sec:CD}

This chapter considers classes of convolution-dominated operators.
The first section is devoted to \emph{integral operators},
whereas the second concerns convolution-dominated \emph{matrices}.

Throughout the section, the weight $w : G \to [1,\infty)$ will always be assumed
to be a $p$-weight for some $p \in (0,1]$.

\section{Integral operators}

Throughout, let $g \in \Bwp$ be a fixed admissible vector (see \Cref{sec:def_coorbit}).
Then the image space
\[
  \CalK_g
  := V_g (\Hpi) \subseteq L^2 (G)
\]
is a closed subspace forming a reproducing kernel Hilbert space, with reproducing kernel
\[
  K : \quad
  G \times G \to \mathbb{C}, \quad
  (x,y) \mapsto V_g [\pi(y) g] (x) = \twisttranslationL{y} [V_g g] (x),
\]
see \Cref{sec:admissible}.
Since $g \in \Bwp$ by assumption, it follows that $|V_g g| \in \wienerStC{L^p_w}$, and
\[
  |K(x,y)|
  = |K(y,x)|
  \leq |V_g g| (y^{-1} x),
  \quad x,y \in G.
\]
The following general class of localized kernels will be the central object of study.

\begin{definition}\label{def:CDOperator}
  Let $p \in (0,1]$, let $w : \group \to [1,\infty)$ be a $p$-weight, and let $g \in \Bwp$ be admissible.
  A measurable function $H : \group \times \group \to \CC$ is called
  \emph{$L_w^p$-localized in $\CalK_g$} if it satisfies the following properties:
  \begin{enumerate}[label=(i\arabic*)]
    \item $H(\cdot, y) \in \CalK_g$ for all $y \in \group$,

    \item $\overline{H(x, \cdot)} \in \CalK_g$ for all $x \in \group$,

    \item \label{enu:EnvelopeCondition}
          There exists a nonnegative \emph{envelope} $\Phi \in \wienerStC{L_w^p}$ satisfying
          \begin{equation}
            \max \big\{ |H(x,y)|, |H(y,x)| \big\}
            \leq \Phi(y^{-1} x),
            \qquad  \, x, y \in \group.
            \label{eq:CDDominiationCondition}
          \end{equation}
  \end{enumerate}
  For a measurable function $H$ satisfying condition \ref{enu:EnvelopeCondition},
  the notation $H \dominated \Phi$ will be used.
\end{definition}

\begin{remark}
  Note that if \Cref{eq:CDDominiationCondition} holds, then it also holds for
  $\Phi_0 = \min \{ \Phi, \Phi^{\vee} \}$ instead of $\Phi$.
  Hence, one can always assume $\Phi$ to be symmetric.
\end{remark}

The following lemma summarizes the basic elementary properties of $L_w^p$-localized
kernels and of the associated integral operators.

\begin{lemma}\label{lem:CDElementary}
  If $H$ is $L^p_w$-localized in $\CalK_g$, then the associated integral operator
  \[
    T_H : \quad
    L^r(G) \to L^r(G), \quad
    T_H F(x) = \int_{\group} H(x,y) F(y) \; \dd{y}
  \]
  is well-defined and bounded for arbitrary $r \in [1,\infty]$,
  with absolute convergence of the defining integral for all $x \in \group$.
  Moreover, the following properties hold:
  \begin{enumerate}[label=(\roman*)]
    \item The map $T_H: L^2(G) \to \CalK_g$ is well-defined.

    \item For all $x,y \in \group$,
          \begin{equation}
            \quad
            H(x,y)
            = \big\langle T_H [V_g (\pi(y) g)], V_g(\pi(x) g) \big\rangle_{L^2}
               \label{eq:CDPointwiseBoundedByOperatorNorm}
          \end{equation}
          and $|H(x,y)| \leq \| T_H \|_{\CalK_g \to L^2}  \| g \|_{\Hpi}^2$.

    \item The \emph{adjoint kernel} $\widetilde{H} : \group \times \group \to \CC$
          defined by $\widetilde{H}(x,y) = \overline{H(y,x)}$ is also $L_w^p$-localized in $\CalK_g$.
          In fact, if $H \dominated \Phi$, then $\widetilde{H} \dominated \Phi$.

    \item If $L : \group \times \group \to \CC$ is also $L_w^p$-localized in $\CalK_g$,
          then so is the product $H \odot L$ defined by
          \[
            H \odot L (x,y)
            := \int_{\group}
                 H(x,z)  L(z,y)
               \dd{z}
            =  T_H [L(\cdot,y)] (x).
          \]
          In addition, if $H \dominated \Phi$ and $L \dominated \Theta$ with $\Phi,\Theta$ symmetric,
          then 
          \[ H \odot L \dominated \max \{ \Phi \ast \Theta, \Theta \ast \Phi \}.\]
          Moreover, the identity $(T_H \circ T_L) F = T_{H \odot L} F$ holds for all $F \in L^2(G)$.
  \end{enumerate}
\end{lemma}

\begin{proof}
  Throughout the proof, let $\Phi \in \wienerStC{L_w^p}$ be symmetric with $H \dominated \Phi$,
  so that
  \begin{align} \label{eq:localizationH}
    |H(x,y)|
    \leq \Phi(y^{-1} x)
    =    \Phi(x^{-1} y)
  \end{align}
  for all $x,y \in \group$.
  Before proving the individual statements of the lemma, we collect a few auxiliary
  observations and prove the boundedness of $T_H : L^r \to L^r$.

  By \Cref{lem:AmalgamWeightedLInftyEmbedding} and because of $w \geq 1$, it follows
  that $\wienerSt{L_w^p} \hookrightarrow \wienerL{L_w^p} \hookrightarrow L^s$
  for all $s \in [1,\infty]$.
  The estimate $|H(x,y)| \leq \Phi(x^{-1} y)$ shows that $H(x,\cdot) \in L^s$ for all $x \in \group$
  and $s \in [1,\infty]$, which implies that the integral defining $T_H F(x)$ exists for
  any $F \in L^r$ and $x \in \group$.
  Moreover, the estimate \eqref{eq:localizationH} yields
  \[
    \int_{\group}
      |H(x,y)| \;
    \dd{y}
    \leq \| \Phi \|_{L^1}
    \qquad \text{and} \qquad
    \int_{\group}
      |H(x,y)| \;
    \dd{x}
    \leq  \| \Phi \|_{L^1}
  \]
  for all $x,y \in \group$.
  An application of Schur's test (see, e.g., \cite[Theorem~6.18]{FollandRA}) therefore yields that
  $T_H : L^r (G) \to L^r(G)$ is bounded for arbitrary $r \in [1,\infty]$.

  Let $F_1,F_2 \in L^2(G)$.
  Then, using the Cauchy-Schwarz inequality,
  the pointwise estimate $|H(x,y)| \leq \sqrt{\strut \Phi(x^{-1} y) \Phi(y^{-1} x)}$
  and Tonelli's theorem, it follows that
  \begin{align*}
    & \int_{\group}
        \int_{\group}
          |F_2(x)|  |H(x,y)|  |F_1(y)|
        \; \dd{y}
      \dd{x} \\
    & \leq \int_{\group}
             |F_2(x)|
             \bigg(
               \int_{\group}
                 \Phi(x^{-1} y)
              \; \dd{y}
             \bigg)^{1/2}
             \bigg(
               \int_{\group}
                 \Phi(y^{-1} x)
                 |F_1(y)|^2 \;
             \;  \dd{y}
             \bigg)^{1/2}
           \dd{x} \\
    & \leq \| \Phi \|_{L^1}^{1/2}
           \bigg(
             \int_{\group}
               |F_2(x)|^2 \;
             \dd{x}
           \bigg)^{1/2}
           \bigg(
             \int_{\group}
               \int_{\group}
                 \Phi(y^{-1} x)
                 |F_1(y)|^2
              \; \dd{y}
             \dd{x}
           \bigg)^{1/2} \\
    & \leq \| \Phi \|_{L^1}^{1/2}
            \| F_2 \|_{L^2}
            \| \Phi \|_{L^1}^{1/2}
            \| F_1 \|_{L^2}.
  \end{align*}
  Therefore, Fubini's theorem is applicable and justifies the calculation
  \begin{equation}
    \begin{split}
      \langle T_H F_1, F_2 \rangle
      & = \int_{\group}
            \int_{\group}
              \overline{F_2 (x)}
               H(x,y)
               F_1(y)
            \; \dd{y}
          \dd{x} \\
      & = \int_{\group}
            F_1(y)
            \int_{\group}
              H(x,y)
               \overline{F_2 (x)}
           \; \dd{x}
          \dd{y} \\
      & = \int_{\group}
            F_1(y)
             \langle H(\cdot,y), F_2 \rangle
         \; \dd{y}
    \end{split}
    %\numberthis
    \label{eq:IntegralOperatorFubini}
  \end{equation}
  for arbitrary $F_1, F_2 \in L^2 (G)$.

  (i)
  Let $F_1 \in L^2(G)$.
  Since $\CalK_g \subseteq L^2(G)$ is closed, to show $T_H F_1 \in \CalK_g$, it suffices to show
  $\langle T_H F_1, F_2 \rangle = 0$ for all $F_2 \in \CalK_g^{\perp}$.
  For this, simply note that $H(\cdot,y) \in \CalK_g$ for all $y \in \group$, so that
  \Cref{eq:IntegralOperatorFubini} shows
  \(
    \langle T_H F_1, F_2 \rangle
    = \int_{\group}
        F_1(y)
         \langle H(\cdot,y), F_2 \rangle \;
      \dd{y}
    = 0 .
  \)

  (ii)
  Note that the reproducing formula \eqref{eq:RKHS} gives
  $F(x) = \langle F, V_g (\pi(x) g) \rangle_{L^2}$ for all $F \in \CalK_g$ and $x \in G$.
  This, combined with the identity \eqref{eq:IntegralOperatorFubini},
  and $H(\cdot, z) \in \CalK_g$ and $\overline{H(x,\cdot)} \in \CalK_g$, gives
  \begin{align*}
    \big\langle T_H [V_g (\pi(y) g)], V_g (\pi(x) g) \big\rangle_{L^2}
    & = \int_{\group}
          V_g (\pi(y) g) (z)
           \langle H(\cdot,z), V_g (\pi(x) g) \rangle
      \;  \dd{z} \\
    & = \int_{\group}
          V_g (\pi(y) g) (z)
           H(x,z)
       \; \dd{z} \\
    & = \overline{
          \langle \overline{H(x,\cdot)}, V_g(\pi(y) g) \rangle
        }
      = H(x,y) ,
  \end{align*}
  which shows \eqref{eq:CDPointwiseBoundedByOperatorNorm}.
  In particular, using the isometry of $V_g$
  and because of $\| \pi(x) g \|_{\Hpi} = \| g \|_{\Hpi}$ and $V_g(\pi(y) g) \in \CalK_g$,
  this implies that
  \begin{align*}
    |H(x,y)|
     \leq \| T_H \|_{\CalK_g \to L^2}
            \| V_g (\pi(y) g) \|_{L^2}
            \| V_g (\pi(x) g) \|_{L^2}
      =    \| T_H \|_{\CalK_g \to L^2}
            \| g \|_{\Hpi}^2
    ,
  \end{align*}
  which proves part (ii).

  (iii)
  Simply note that $\widetilde{H}(\cdot,y) = \overline{H(y,\cdot)} \in \CalK_g$
  and $\overline{\widetilde{H}(x,\cdot)} = H(\cdot,x) \in \CalK_g$ for all $x,y \in \group$ and that
  \(
    \max
    \{
      |\widetilde{H}(x,y)|,
      |\widetilde{H}(y,x)|
    \}
    = \max
    \{
      |H(x,y)|,
      |H(y,x)|
    \}
    \leq \Phi(y^{-1} x)
    ,
  \)
  even without assuming that $\Phi$ is symmetric.

  (iv)
  By part (i) and because of $L(\cdot,y) \in \CalK_g$, it follows that
  \(
    (H \odot L)(\cdot,y)
    = T_H [L(\cdot, y)]
    \in \CalK_g
    .
  \)
  Another direct calculation shows that
  \(
    \overline{
      (H \odot L)(x, \cdot)
    }
    = T_{\widetilde{L}} \big[ \widetilde{H}(\cdot, x) \big]
    \in \CalK_g ,
  \)
  by Part~(i) applied to $\widetilde{L}$
  (which is $L_w^p$-localized in $\CalK_g$ by Part~(iii)
  and since $\widetilde{H}(\cdot, x) \in \CalK_g$.

  If $H \dominated \Phi$ and $L \dominated \Theta$
  with symmetric $\Phi,\Theta \in \wienerStC{L_w^p}$, then
  \begin{align*}
    \big|
      H \odot L (x,y)
    \big|
    & \leq \int_{\group}
             |H(x,z)|  |L(z,y)|
         \;  \dd{z}
      \leq \int_{\group}
             \Phi(x^{-1} z) \Theta(z^{-1} y)
          \; \dd{z} \\
    & =    \int_{\group}
             \Phi(w) \Theta(w^{-1} x^{-1} y)
         \;  \dd{w}
      =    (\Phi \ast \Theta) (x^{-1} y)
      =    (\Theta \ast \Phi) (y^{-1} x) ,
  \end{align*}
  where the last step used the elementary identity $(F \ast H)^{\vee} = H^\vee \ast F^\vee$.
  The calculation from above also shows $|H \odot L(y,x)| \leq (\Phi \ast \Theta) (y^{-1} x)$.
  Hence, $H \odot L \dominated \max \{ \Phi \ast \Theta, \Theta \ast \Phi \}$.
  Since $\Phi \ast \Theta, \Theta \ast \Phi \in \wienerStC{L_w^p}$
  by \Cref{cor:UserFriendlyConvolutionBounds}, it follows that
  $H \odot L$ is $L_w^p$-localized in $\CalK_g$.

  Lastly, note that \Cref{lem:AmalgamWeightedLInftyEmbedding} shows
  $\Phi \ast \Theta \in \wienerSt{L_w^p} \hookrightarrow \wienerL{L_w^p}  \hookrightarrow L^2$.
  Hence, it follows for arbitrary $F \in L^2(G)$ that
  \begin{align*}
    \int_{\group}
      \int_{\group}
        |H(x,y)  L(y,z)|
        |F(z)|
        \;
      \dd{y}
    \dd{z}
    & \leq \int_{\group}
             \int_{\group}
               \Phi(x^{-1} y) \Theta(y^{-1} z) |F(z)| \;
             \dd{y}
           \dd{z} \\
    & =    \int_{\group}
             |F(z)|  (L_x (\Phi \ast \Theta))(z) \;
           \dd{z} \\
    & \leq \| F \|_{L^2}  \| L_x (\Phi \ast \Theta) \|_{L^2}
      <    \infty
  \end{align*}
  for any $x \in \group$.
  Therefore, Fubini's theorem is applicable in the calculation
  \begin{align*}
    \bigl[T_H (T_L F)\bigr] (x)
    & = \int_{\group}
          H(x,y)  T_L F(y) \;
        \dd{y}
      = \int_{\group}
          \int_{\group}
            H(x,y) \, L(y,z) \, F(z) \;
          \dd{z}
        \dd{y} \\
    & = \int_{\group}
          (H \odot L)(x,z)  F(z) \;
        \dd{z}
      = (T_{H \odot L} F) (x),
  \end{align*}
  which completes the proof.
\end{proof}

The next result establishes a form of \emph{local spectral invariance}
of $L_w^p$-localized integral operators.
Concerning the \emph{holomorphic spectral calculus} appearing in the statement
of the theorem, see, e.g., \cite[Sections~10.21--10.29]{RudinFA}.

\begin{theorem}\label{thm:CDIntegralOperatorsLocalSpectralInvariance}
  Let $p \in (0,1]$ and let $w : \group \to [1,\infty)$ be a $p$-weight.
  For an admissible $g \in \Bwp$, set $\Phi := |V_g g| \in \wienerStC{L^p_w}$ 
  and let $\Theta \in \wienerStC{L_w^p}$ be nonnegative.

  For arbitrary $\delta > 0$, there exists $\eps = \eps(\Theta,g,Q,w,p,\delta) \in (0,\delta)$
  with the following property:
  If $\phi : B_\delta (1) \subseteq \CC \to \CC$ is holomorphic
  and if $H : \group \times \group \to \CC$ is $L_w^p$-localized in $\CalK_g$ satisfying
  \begin{enumerate}[label=(\arabic*)]
      \item $H \dominated \Theta$,
      \item $\| T_H - \identity_{\CalK_g} \|_{\CalK_g \to L^2} \leq \eps$,
  \end{enumerate}
  then there exists an $L_w^p$-localized $H_\phi : \group \times \group \to \CC$
 such that the operator
  $\phi(T_H) : \CalK_g \to \CalK_g$ defined through the holomorphic functional calculus
  satisfies $\phi(T_H) = T_{H_\phi}|_{\CalK_g}$.
\end{theorem}

\begin{proof}
  The  proof is divided into several steps and closely follows
  the proof of \cite[Theorem~4.3]{MoleculePaper}.

  \medskip{}

  \textbf{Step~1.} \emph{(Choice of $\eps$).}
  Let $\delta > 0$ be given and set $\beta := \| g \|_{\Hpi}^2$
  and $\Theta_0 := \min \{ \Theta, \Theta^{\vee} \}$, so that $H \dominated \Theta_0$.
  Using \Cref{cor:UserFriendlyConvolutionBounds}, choose a constant $C_1 = C_1(w,p,Q) \geq 1$
  satisfying
  \(
    \| F_1 \ast F_2 \|_{\wienerSt{L_w^p}}
    \leq \| |F_1| \ast |F_2| \|_{\wienerSt{L_w^p}}
    \leq C_1  \| F_1 \|_{\wienerSt{L_w^p}}  \| F_2 \|_{\wienerSt{L_w^p}}
  \)
  for all $F_1,F_2 \in \wienerSt{L_w^p}$.

  For $\eps > 0$, set $\Psi_\eps := \min \{ \eps \beta, \Phi + \Theta_0 \}$.
  Since $\maxSt \Psi_\eps \leq \min \{ \eps \beta, \maxSt \Phi + \maxSt \Theta \} \in L_w^p$,
  it follows from the dominated convergence theorem
  (applied along an arbitrary null-sequence $\eps_n \to 0$)
  that $\| \Psi_\eps \|_{\wienerSt{L_w^p}} \to 0$ as $\eps \downarrow 0$.
  Hence, there exists $\eps = \eps(\Theta,g,\delta,w,p,Q) \in (0,\frac{\delta}{2})$
  such that $\| \Psi_\eps \|_{\wienerSt{L_w^p}} \leq \frac{\delta}{4 C_1}$.

  \medskip{}

  \textbf{Step~2.} \emph{(Series representation of $\phi(T_H)$).}
  Let $\phi : B_\delta (1) \to \CC$ be holomorphic.
  By assumption, $\| \identity_{\CalK_g} - T_H \|_{\CalK_g \to \CalK_g} \leq \eps < \frac{\delta}{2}$,
  and hence $\sigma(T_H) \subseteq B_{\delta/2} (1)$.
  This implies that $\phi(T_H) : \CalK_g \to \CalK_g$ is a well-defined bounded linear operator.
  By expanding $\phi$ into a power series,
  we can write $\phi(z) = \sum_{n=0}^\infty a_n \, (z - 1)^n$ for all ${z \in B_\delta (1)}$,
  for a suitable sequence $(a_n)_{n \in \N_0} \subseteq \CC$.
  The series representing $\phi$ convergences locally uniformly on $B_\delta (1)$.
  Therefore, elementary properties of the holomorphic functional calculus
  (see, e.g., \mbox{\cite[Theorem~10.27]{RudinFA}}) show that
  \begin{equation}
    \phi(T_H) = \sum_{n=0}^\infty a_n \, (T_H - \identity_{\CalK_g})^n,
    \label{eq:SpectralCalculusSeriesRepresentation}
  \end{equation}
  with convergence in the operator norm.
  An application of the Cauchy-Hadamard formula gives
  $\delta \leq \big[\, \limsup_{n \to \infty} |a_n|^{1/n} \,\big]^{-1}$.
  Thus, there exists some $N = N(\phi,\delta) \in \N$ such that $|a_n|^{1/n} \leq \frac{2}{\delta}$
  for all $n \geq N$.
  Consequently, there exists $C_\phi = C_\phi (\delta) > 0$ such that
  \begin{equation}
    |a_n| \leq C_\phi  (2/\delta)^n
    \label{eq:SpectralCalculusCoefficientBound}
  \end{equation}
  for all $n \in \N_0$.

 \medskip{}

  \textbf{Step~3.} \emph{(Integral representation of $T_H - \identity_{\CalK_g}$)}.
  For $F \in \CalK_g$ and $x \in G$, the reproducing formula \eqref{eq:RKHS} yields that
 $F(x)
      = T_K F (x)
      $, where $K$ is the reproducing kernel given by \eqref{eq:RK}.
  In other words,
  \begin{equation}
    T_K |_{\CalK_g} = \identity_{\CalK_g}
    \qquad \text{and} \qquad
    T_H - \identity_{\CalK_g}
    = T_L |_{\CalK_g}
    \quad \text{for} \quad
    L := H - K.
    \label{eq:IdentityIntegralRepresentation}
  \end{equation}
  Since $\overline{K(x,\cdot)} = V_g [\pi(x) g] \in \CalK_g$
  and $K(\cdot,y) = V_g [\pi(y) g] \in \CalK_g$ for arbitrary $x,y \in \group$,
  it follows that $L(\cdot,y) \in \CalK_g$ and $\overline{L(x,\cdot)} \in \CalK_g$
  for all $x,y \in \group$.
  Since also
  \[
    |L(y,x)|
    \leq \Phi(x^{-1} y) + \Theta_0(x^{-1} y)
    = \Phi(y^{-1} x) + \Theta_0(y^{-1} x)
  \]
  and $|L(x,y)| \leq \Phi(y^{-1} x) + \Theta_0(y^{-1} x)$
  and hence $L \dominated \Phi + \Theta_0$ with $\Phi + \Theta_0 \in \wienerStC{L_w^p}$,
  it follows that $L$ is $L_w^p$-localized in $\CalK_g$.

  \medskip{}

  \textbf{Step~4.} \emph{(Refined $L_w^p$-localization of $L$).}
  By Step~3 and the assumptions of the theorem, it holds that
  \[
    \| T_L \|_{\CalK_g \to L^2}
    = \| T_H - \identity_{\CalK_g} \|_{\CalK_g \to L^2}
    \leq \eps.
  \]
  Therefore, the pointwise estimate following from the identity
  \eqref{eq:CDPointwiseBoundedByOperatorNorm} shows that
  \[
    |L(x,y)|
    \leq \| T_L \|_{\CalK_g \to L^2}  \| g \|_{\Hpi}^2
    \leq \eps \beta,
    \qquad  x,y \in \group.
  \]
  Combined with the estimate from the end of Step~3,
  this shows $L \dominated \Psi_\eps$.

  \medskip{}

  \textbf{Step~5.} \emph{(Powers of $\Psi_\eps$ and $L$):}
  For $n \in \N$, define $\Psi_\eps^{\ast(1)} := \Psi_\eps$
  and $\Psi_{\eps}^{\ast (n+1)} := \Psi_\eps \ast \Psi_\eps^{\ast(n)}$ inductively.
  Similarly, let $L^{\circ (1)} := L$ and $ L^{\circ (n+1)} := L^{\circ (n)} \odot L$
  for $n \in \N$, where $\odot$ is the product defined in \Cref{lem:CDElementary}.
  By a straightforward induction, that lemma shows that
  $(T_H - \identity_{\CalK_g})^n = T_L^n |_{\CalK_g} = T_{L^{\circ(n)}} |_{\CalK_g}$
  and that $L^{\circ(n)}$ is $L_w^p$-localized in $\CalK_g$ for each $n \in \N$.
  In particular, this implies
  \begin{equation}
    L^{\circ(n)}(\cdot,y) \in \CalK_g
    \quad \text{and} \quad
    \overline{L^{\circ(n)} (x,\cdot)} \in \CalK_g,
    \qquad  x ,y \in \group .
    \label{eq:IteratedKernelsBelongToRKHS}
  \end{equation}
  Furthermore, a straightforward induction shows
  $\Psi_\eps^{\ast(n+1)} = \Psi_\eps^{\ast(n)} \ast \Psi_\eps$
  and $(\Psi_\eps^{\ast(n)})^{\vee} = \Psi_\eps^{\ast(n)}$ for all $n \in \N$.

  Another induction argument shows that $L^{\circ(n)} \dominated \Psi_\eps^{\ast(n)}$:
  For $n = 1$, this was shown in Step~4.
  For the induction step, \Cref{lem:CDElementary} shows by symmetry of $\Psi_\eps$
  and $\Psi_\eps^{\ast(n)}$ and because of $L^{\circ(n)} \dominated \Psi_\eps^{\ast(n)}$
  and $L \dominated \Psi_\eps$ that
  \[
    L^{\circ(n+1)}
    = L^{\circ(n)} \odot L
    \dominated \max
               \big\{
                 \Psi_\eps \ast \Psi_\eps^{\ast(n)} ,
                 \Psi_\eps^{\ast(n)} \ast \Psi_\eps
               \big\}
    = \Psi_\eps^{\ast(n+1)} ,
  \]
  as required.
  Lastly, it holds that
  \begin{equation}
    \| \Psi_\eps^{\ast(n)} \|_{\wienerSt{L_w^p}}
    \leq \bigl(\delta / 4\bigr)^n
    ,
    \qquad  n \in \N.
    \label{eq:PsiConvolutionPowersNormEstimate}
  \end{equation}
  Indeed, for $n = 1$ this follows since $C_1 \geq 1$ and hence
  $\| \Psi_\eps \|_{\wienerSt{L_w^p}} \leq \frac{\delta}{4 C_1} \leq \frac{\delta}{4}$
  by the choice of $\eps$ in Step~1.
  Next, for the induction step note by choice of $C_1$ that
  \[
    \| \Psi_\eps^{\ast(n+1)} \|_{\wienerSt{L_w^p}}
    = \| \Psi_\eps \ast \Psi_\eps^{\ast(n)} \|_{\wienerSt{L_w^p}}
    \leq C_1  \| \Psi_\eps \|_{\wienerSt{L_w^p}}
              \| \Psi_\eps^{\ast(n)} \|_{\wienerSt{L_w^p}}
    \leq \delta/4  (\delta/4)^n
    =    (\delta/4)^{n+1} ,
  \]
  which establishes the claim \eqref{eq:PsiConvolutionPowersNormEstimate}.

  \medskip{}

  \textbf{Step~6.} \emph{(Construction of $H_\phi$).}
  Combining \Cref{eq:SpectralCalculusCoefficientBound,eq:PsiConvolutionPowersNormEstimate} gives
  \[
    \sum_{n=1}^\infty
      \big(
        |a_n|  \| \Psi_\eps^{\ast(n)} \|_{\wienerSt{L_w^p}}
      \big)^p
    \leq \sum_{n=1}^\infty
           \big[
             C_\phi  (2/\delta)^n  (\delta/4)^n
           \big]^p
    =    C_\phi^p
          \sum_{n=1}^\infty
                 (1/2^p)^n
    <    \infty .
  \]
  Since the norm on $\wienerStC{L_w^p}$ is a $p$-norm, the preceding estimate implies
  by \Cref{lem:QuasiBanachAbsoluteConvergence} that the series
  $\sum_{n=1}^\infty |a_n| \, \Psi_\eps^{\ast(n)}$
  is unconditionally convergent in $\wienerStC{L_w^p}$.
  Since $\wienerStC{L_w^p} \hookrightarrow C_b(\group)$
  as a consequence of \Cref{lem:AmalgamWeightedLInftyEmbedding},
  the series in particular converges uniformly.

  Define $\Psi := |a_0| \, \Phi + \sum_{n=1}^\infty |a_n| \, \Psi_\eps^{\ast(n)} \in \wienerStC{L_w^p}$.
  Then the kernel $H_\phi : \group \times \group \to \CC $ defined by
  \[
    H_{\phi} (x,y)
    = a_0  K(x,y) + \sum_{n=1}^\infty a_n \, L^{\circ (n)}(x,y)
  \]
  is well-defined with the series converging absolutely, and
  \begin{equation}
    \begin{split}
      |H_\phi (x,y)|
      & \leq |a_0|  |K(x,y)| + \sum_{n=1}^\infty |a_n|  |L^{\circ(n)}(x,y)| \\
      & \leq |a_0|  \Phi(y^{-1} x)
             + \sum_{n=1}^\infty
                 |a_n|  \Psi_\eps^{\ast(n)}(y^{-1} x)
        =    \Psi(y^{-1} x)
        <    \infty ,
    \end{split}
    \label{eq:LocalSpectralInvarianceKernelDomination}
  \end{equation}
  since $L^{\circ (n)} \dominated \Psi_\eps^{\ast(n)}$ (cf.\ Step~5.)
  Similar arguments show that also $|H_\phi (y,x)| \leq \Psi(y^{-1} x)$
  and thus $H_\phi \dominated \Psi$.

  To prove that $H_\phi$ is $L_w^p$-localized in $\CalK_g$, it remains to show
  $H_\phi(\cdot,y) \in \CalK_g$ and $\overline{H_\phi(x,\cdot)} \in \CalK_g$
  for all $x,y \in \group$.
  To see this, note that \Cref{lem:AmalgamWeightedLInftyEmbedding} shows
  $\Psi \in \wienerSt{L_w^p} \hookrightarrow \wienerL{L_w^p} \hookrightarrow L^2$.
  In combination with \Cref{eq:LocalSpectralInvarianceKernelDomination}
  and the dominated convergence theorem, this implies that the series defining
  $H_\phi (\cdot,y)$ converges in $L^2(G)$.
  Since $\CalK_g \subseteq L^2(G)$ is closed,
  since $L^{\circ(n)}(\cdot,y) \in \CalK_g$ by \Cref{eq:IteratedKernelsBelongToRKHS},
  and since $K(\cdot,y) \in \CalK_g$,
  this implies $H_\phi(\cdot,y) \in \CalK_g$, as required.
  The proof of $\overline{H_\phi(x,\cdot)} \in \CalK_g$ is similar,
  using that $\big( y \mapsto \Psi(y^{-1} x) = \Psi(x^{-1} y) \big) \in L^2(G)$.
\\~\\
  \textbf{Step~7.} \emph{($\phi(T_H) = T_{H_\phi} |_{\CalK_g}$).}
  Let $F \in \CalK_g$ and $x \in \group$.
  Note that Step 6 shows
  \[ \sum_{n=1}^\infty |a_n| \, |L^{\circ(n)}(x,\cdot)| \leq \Psi(x^{-1} \cdot), \]
  where $\Psi \in L^2(G)$.
  Hence, the dominated convergence theorem justifies the following calculation:
  \begin{align*}
    T_{H_\phi} F (x)
    & = a_0  T_K F (x) + \sum_{n=1}^\infty [a_n  T_{L^{\circ(n)}} F (x)]
     = a_0  F (x)
       + \sum_{n=1}^\infty
         \big(
           a_n  [(T_H - \identity_{\CalK_g})^n F] (x)
         \big) \\
    & = \sum_{n=0}^\infty \big( a_n  [(T_H - \identity_{\CalK_g})^n F] (x) \big)
     = [\phi(T_H) F] (x) ,
  \end{align*}
  where the second (resp.\ fourth) equality used \eqref{eq:IdentityIntegralRepresentation}
  (resp.\ \eqref{eq:SpectralCalculusSeriesRepresentation}).
  Thus, $\phi(T_H) = T_{H_\phi} |_{\CalK_g}$.
\end{proof}

\Cref{thm:CDIntegralOperatorsLocalSpectralInvariance} provides an extension
of \cite[Theorem~4.3]{MoleculePaper} from $L^1_w$-localized kernels
to general $L^p_w$-localized kernels for $p \in (0,1]$.
This extension will be crucial for developing the theory of molecules for quasi-Banach coorbit spaces
in \Cref{sec:molecules}.

\section{Matrices}

This section concerns matrices that are indexed by discrete subsets of $G$
(that do not need to be subgroups) and that possess a certain off-diagonal decay.
The precise notion is as follows.

\begin{definition}
Let $p \in (0,1]$ and let $w : G \to [1,\infty)$ be a $p$-weight.
Let $\Lambda = (\lambda_i)_{i \in I}$ and $\Gamma = (\gamma_j)_{j \in J}$
be relatively separated families in $G$.

A matrix $A = (A_{i,j})_{(i,j) \in I \times J} \in \mathbb{C}^{I \times J}$
is called \emph{$L^p_w$-localized}
if there exists an \emph{envelope}
$\Phi \in \wienerStC{L^p_w}$ such that
\begin{align} \label{eq:CD_matrix_envelope}
  |A_{i,j}|
  \leq \min
       \{
         \Phi(\gamma_j^{-1} \lambda_i), \;
         \Phi(\lambda_i^{-1} \gamma_j)
       \},
  \quad i \in I, \; j \in J.
\end{align}
If condition \eqref{eq:CD_matrix_envelope} holds, the notation $A \dominated \Phi$
 will be used.
The space of all $L^p_w$-localized matrices in $\CC^{\Lambda \times \Gamma}$ is the collection
\[
  \goodMatrices(\Lambda, \Gamma)
  := \bigg\{
       A \in \mathbb{C}^{I \times J}
       \; : \;
       A \dominated \Phi \;\; \text{for some} \;\; \Phi \in \wienerStC{L^p_w}
     \bigg\},
\]
and is equipped with the mapping $\| \cdot \|_{\goodMatrices} : \goodMatrices \to [0,\infty)$
defined by
\[
  \| A \|_{\goodMatrices}
  := \inf
     \big\{
       \| \Phi \|_{\wienerSt{L^p_w}}
       \; : \;
       A \dominated \Phi \in \wienerStC{L^p_w}
     \big\}.
\]
In case of $\Lambda = \Gamma$, the notation
$\goodMatrices(\Lambda) := \goodMatrices(\Lambda, \Lambda)$ will be used.
\end{definition}

The basic properties of $L^p_w$-localized matrices are collected in the following lemma.

\begin{lemma} \label{lem:CD_matrix_basic}
  Let  $\Lambda = (\lambda_i)_{i \in I}$, $\Gamma = (\gamma_j)_{j \in J}$,
  and $\Upsilon = (\upsilon_k)_{k \in K}$ be relatively separated families in $G$.
  Then the following hold:
  \begin{enumerate}[label=(\roman*)]
    \item The space $\goodMatrices(\Lambda, \Gamma)$ forms a quasi-Banach space
          with $p$-norm $\| \cdot \|_{\goodMatrices}$.

    \item Any $A = (A_{i, j})_{(i,j) \in I \times J} \in \goodMatrices (\Lambda, \Gamma)$
          satisfies
          \begin{align} \label{eq:CD_schurtype}
              \sum_{i \in I}
                |A_{i, j}|
              \lesssim \rel(\Lambda)  \| A \|_{\goodMatrices}
              \quad \text{and} \quad
              \sum_{j \in J}
                |A_{i, j}|
              \lesssim \rel(\Gamma)  \| A \|_{\goodMatrices},
          \end{align}
          with an implicit constant depending only on $p,w$ and $Q$.
          In particular, the embedding
          \(
            \goodMatrices (\Lambda, \Gamma)
            \hookrightarrow \mathcal{B} (\ell^r (J), \ell^r (I))
          \)
          holds for all $r \in [1,\infty]$, with
          \begin{align} \label{eq:CD_schurtype2}
              \| A \|_{\ell^r (J) \to \ell^r (I)}
              \lesssim  \max \{ \rel(\Lambda), \rel(\Gamma) \}
                         \| A \|_{\goodMatrices}.
          \end{align}

    \item If $A \in \goodMatrices (\Upsilon, \Lambda)$
          and $B \in \goodMatrices (\Lambda, \Gamma)$,
          then the product $A B$
          (which is defined as usual by $(A B)_{k,j} = \sum_{i \in I} A_{k,i} B_{i,j}$)
          satisfies $AB \in \goodMatrices ( \Upsilon, \Gamma)$, with
          \begin{align} \label{eq:CD_matrix_algebra}
            \| A B \|_{\goodMatrices}
            \lesssim \rel(\Lambda)  \| A \|_{\goodMatrices} \| B \|_{\goodMatrices}.
          \end{align}
          for an implicit constant depending only on $p,w$ and $Q$.
  \end{enumerate}
\end{lemma}

\begin{proof}
All implied constants in this proof only depend on $p,w,Q$.

(i)
The absolute homogeneity and the $p$-norm property
$\| A + B \|_{\goodMatrices}^p \leq \| A \|_{\goodMatrices}^p + \| B \|_{\goodMatrices}^p$
of $\| \cdot \|_{\goodMatrices}$
follow directly from the corresponding properties of $\| \cdot \|_{\wienerSt{L^p_w}}$.
For showing that $\| \cdot \|_{\goodMatrices}$ is positive definite, let $A \in \goodMatrices (\Lambda, \Gamma)$. For any
 $\Phi \in \wienerStC{L^p_w}$ satisfying $A \dominated \Phi$, note that, by \Cref{lem:AmalgamWeightedLInftyEmbedding},
\begin{equation}
  |A_{i,j}|
  \leq \Phi(\gamma_j^{-1} \lambda_i)
  \leq \| \Phi \|_{C_b}
  =    \| \Phi \|_{L^{\infty}}
  \lesssim \| \Phi \|_{\wienerSt{L^p_w}}.
  \label{eq:SpecialMatrixNormDominatesMatrixEntries}
\end{equation}
Hence, taking the infimum over all $\Phi$ satisfying  $A \dominated \Phi$ yields directly that
$|A_{i,j}| \lesssim \| A \|_{\goodMatrices}$. Thus, if $\| A \|_{\goodMatrices} = 0$, then $A = 0$, which shows that
 $\| \cdot \|_{\goodMatrices}$ is positive definite.

It remains to show that $\goodMatrices(\Lambda, \Gamma)$ is complete.
For this, by \Cref{lem:QuasiBanachAbsoluteConvergence},
it suffices to show that if $(A^{(n)})_{n \in \mathbb{N}}$ is a sequence
in $\goodMatrices(\Lambda, \Gamma)$ satisfying
$\sum_{n \in \mathbb{N}} \| A^{(n)} \|_{\goodMatrices}^p < \infty$,
then $\sum_{n \in \mathbb{N}}  A^{(n)} $ converges in $\goodMatrices(\Lambda, \Gamma)$.
Given such a sequence $(A^{(n)})_{n \in \mathbb{N}}$,
define $A \in \mathbb{C}^{I \times J}$ by $A_{i,j} := \sum_{n \in \mathbb{N}} A_{i,j}^{(n)}$.
That $A$ is well-defined follows from
the observation after \Cref{eq:SpecialMatrixNormDominatesMatrixEntries} by noting
\[
  \sum_{n \in \N}
    |A_{i,j}^{(n)}|
  \leq \sum_{n \in \N}
         |A_{i,j}^{(n)}|^p
  \lesssim \sum_{n \in \N}
             \| A^{(n)} \|_{\goodMatrices}^p
  < \infty
  .
\]
%Note that $\sum_{n \in \mathbb{N}} \| A_{i,j}^{(n)} \|^p_{\goodMatrices} < \infty$.
For $n \in \mathbb{N}$, let $\Phi_n \in \wienerStC{L^p_w}$ be such that $A^{(n)} \dominated \Phi_n$
and $\|\Phi_n \|_{\wienerSt{L^p_w}} \leq 2 \| A^{(n)} \|_{\goodMatrices}$,
and define $\Theta_N := \sum_{n = N+1}^{\infty} \Phi_n$ for $N \in \mathbb{N}_0$.
The choice of $\Phi_n$ implies that
\[
  \sum_{n \in \mathbb{N}}
    \| \Phi_n \|_{\wienerSt{L_w^p}}^p
  \leq 2^p \sum_{n \in \mathbb{N}}
           \| A^{(n)} \|_{\goodMatrices}^p
  < \infty
  ,
\]
so that $\Theta_N \in \wienerStC{L^p_w}$ by \Cref{lem:QuasiBanachAbsoluteConvergence}.
In addition, each $\Theta_N$ satisfies the envelope property
\[
  \bigg| \bigg( A - \sum_{n = 1}^N A^{(n)} \bigg)_{i,j} \bigg|
  \leq \sum_{n=N+1}^{\infty}
         \min \{ \Phi_n (\gamma_j^{-1} \lambda_i), \; \Phi_n (\lambda_i^{-1} \gamma_j) \}
  \leq \min \{ \Theta_N (\gamma_j^{-1} \lambda_i), \; \Theta_N (\lambda_i^{-1} \gamma_j) \}
\]
for all $i \in I$ and $j \in J$.
In particular, this implies that $A - \sum_{n = 1}^N A^{(n)} \in \goodMatrices(\Lambda, \Gamma)$
with
\[
  \bigg\| A - \sum_{n = 1}^N A^{(n)} \bigg\|_{\goodMatrices}^p
  \leq \| \Theta_N \|^p_{\wienerSt{L^p_w}}
  \leq \sum_{n = N+1}^{\infty} \| \Phi_n \|_{\wienerSt{L^p_w}}^p
  \to  0
  \quad \text{as} \quad N \to \infty.
\]
Thus, $\sum_{n \in \mathbb{N}}  A^{(n)} $ converges in $\goodMatrices(\Lambda, \Gamma)$.
Overall, this shows that $\goodMatrices(\Lambda, \Gamma)$ is complete.

\medskip{}

(ii)
Let $A \in \goodMatrices(\Lambda, \Gamma)$ and let $\Phi \in \wienerStC{L^p_w}$
be such that $A \dominated \Phi$.
Because of $w \geq 1$, an application of \Cref{lem:StandardShiftedSeriesEstimates} yields,
for all $j \in J$, that
\[
  \sum_{i \in I} |A_{i,j} |
  \leq \sum_{i \in I} \Phi(\gamma_j^{-1} \lambda_i)
  \leq \frac{\rel(\Lambda)}{\mu_G(Q)} \| \Phi \|_{\wienerL{L^1_w}}
  \lesssim \frac{\rel(\Lambda)}{\mu_G(Q)} \| \Phi \|_{\wienerStC{L^p_w}},
\]
where the last inequality follows from \Cref{lem:AmalgamWeightedLInftyEmbedding}.
Similarly,
$\sum_{j \in J} |A_{i,j} | \lesssim \frac{\rel(\Gamma)}{\mu_G(Q)} \| \Phi \|_{\wienerStC{L^p_w}}$.
This implies the estimates \eqref{eq:CD_schurtype}.
The embedding \eqref{eq:CD_schurtype2} is a direct consequence of Schur's test,
see, e.g., \cite[Theorem~6.18]{FollandRA}.

\medskip{}

(iii)
Let $\Phi, \Theta \in \wienerStC{L^p_w}$ be such that $A \dominated \Phi$ and $B \dominated \Theta$.
By \Cref{lem:StandardShiftedSeriesEstimates},
it follows that, for arbitrary $\upsilon_k \in \Upsilon$ and $\gamma_j \in \Gamma$,
\[
  |(AB)_{k, j}|
  \leq \sum_{i \in I} |A_{k, i}| |B_{i, j}|
  \leq \sum_{i \in I} \Phi(\lambda_i^{-1} \upsilon_k) \Theta(\gamma_j^{-1} \lambda_i)
  \leq \frac{\rel(\Lambda)}{\mu_G (Q)}
        \big(\maxL \Theta \ast \maxR \Phi\big) (\gamma_j^{-1} \upsilon_k).
\]
A similar calculation also gives
\(
  |(AB)_{k,j}|
  \leq \frac{\rel(\Lambda)}{\mu_G (Q)}
        \big( \maxL \Phi \ast \maxR \Theta \big) (\upsilon_k^{-1} \gamma_j)
  .
\)
Define the function $H : G \to [0,\infty)$ by
\[
  H
  := \frac{\rel(\Lambda)}{\mu_G (Q)}
      \bigg(
             \big(\maxL \Theta \ast \maxR \Phi\big) + \big(\maxL \Phi \ast \maxR \Theta \big)
           \bigg),
\]
so that $|(AB)_{k,j}| \leq \min \{ H(\gamma_j^{-1} \upsilon_k), \; H(\upsilon_k^{-1} \gamma_j) \}$
for all $j \in J$ and $k \in K$.
By \Cref{cor:UserFriendlyConvolutionBounds}, it follows that
$\maxL \Theta \ast \maxR \Phi \in \wienerStC{L^p_w}$ and
$\maxL \Phi \ast \maxR \Theta \in \wienerStC{L^p_w}$, and thus $H \in \wienerStC{L^p_w}$, with
\[
  \| H \|_{\wienerStC{L^p_w}}
  \lesssim \rel(\Lambda)  \| \Theta \|_{\wienerStC{L^p_w}} \| \Phi \|_{\wienerStC{L^p_w}}.
\]
Overall, $AB \dominated H$ with $\|A B \|_{\goodMatrices} \leq \| H \|_{\wienerStC{L^p_w}}$,
which easily yields the desired claim.
\end{proof}

The proof of the following theorem resembles the proof
of \Cref{thm:CDIntegralOperatorsLocalSpectralInvariance},
but for matrices instead of integral operators.
For $L^1_w$-localized matrices, the result can already be found in
\cite{MoleculePaper}.

\begin{theorem}\label{thm:CDMatricesLocalSpectralInvariance}
  Let $w : G \to [1,\infty)$ be a $p$-weight for some $p \in (0,1]$
  and let $\Theta \in \wienerStC{L^p_w}$.
  Let $R > 0$, and let $\Lambda = (\lambda_i)_{i \in I}$
  be a relatively separated family in $G$ with $\rel(\Lambda) \leq R$.

  For arbitrary $\delta > 0$, there exists $\eps = \eps (\Theta, R, \delta, w, p, Q) \in (0, \delta)$
  with the following property:
  If $\phi : B_{\delta} (1) \subseteq \mathbb{C} \to \mathbb{C}$ is holomorphic
  and if $A \in \goodMatrices (\Lambda)$ satisfies
  \begin{enumerate}[label=(\arabic*)]
    \item $A \dominated \Theta$,
    \item \(
            \| A - \identity_{\ell^2 (I)} \|_{\ell^2(I) \to \ell^2 (I)}
            \leq \eps
            ,
          \)
  \end{enumerate}
  then the operator $\phi (A) : \ell^2 (I) \to \ell^2 (I)$
  defined through the holomorphic functional calculus is well-defined
  and its associated matrix is an element of $\goodMatrices(\Lambda)$.
\end{theorem}

\begin{proof}
Throughout, let $\Phi \in \wienerStC{L^p_w}$ be such that $A \dominated \Phi$.
The proof is split into four steps.

\medskip{}

\textbf{Step 1.} \emph{(Choice of $\eps$).}
For a symmetric function $\varphi \in C_c (G) \subseteq \wienerStC{L^p_w}$
such that $\varphi \geq 0$ and $\varphi(e_G) = 1$, define $\Theta := \varphi + \Phi$.
For $k \in \mathbb{N}$, let $\Theta_k := \min \{k^{-1}, \Theta \}$ and note the pointwise
estimate and convergence
\[
  \maxSt \Theta_k (x)
  \leq \min \{ k^{-1}, \maxSt \Theta (x) \}
  \to 0
  \quad \text{as} \quad k \to \infty.
\]
Since $\maxSt \Theta \in L^p_w (G)$, an application of Lebesgue's dominated convergence theorem
implies $\| \Theta_k \|_{\wienerSt{L^p_w}} = \| \maxSt \Theta_k \|_{L^p_w} \to 0$ as $k \to \infty$.

Let $C = C(p,w,Q) > 0$ be the implicit constant appearing in \Cref{eq:CD_matrix_algebra}
and define $C_1 := \max \{1, C  R \}$ and $C_2 := (4/\delta)  C_1 > 0$.
By the previous paragraph, there exists $k_0 \in \mathbb{N}$ satisfying $k_0 \geq 2/\delta$
and $\| \Theta_{k_0} \|_{\wienerStC{L^p_w}} \leq 1/C_2$.
For such a fixed $k_0$, set $\eps := k_0^{-1}$ throughout.

\medskip{}

\textbf{Step 2.} \emph{(Series representation of $\phi (A)$).}
Let $(a_n)_{n \in \mathbb{N}_0}$ with $\phi(z) = \sum_{n \in \mathbb{N}_0} a_n  (z-1)^n$
for all $z \in B_{\delta}(1)$ with uniform convergence on compact subsets of $B_{\delta} (1)$.
By assumption, $\| A - \identity_{\ell^2(I)} \|_{\ell^2 \to \ell^2} \leq \eps$,
and $\eps \leq \delta / 2$ by Step 1,
and thus $\sigma (A) \subseteq \overline{B_{\delta /2}(1)} \subseteq B_{\delta} (1)$.
Therefore, basic properties of the holomorphic functional calculus
(see, e.g., \mbox{\cite[Theorem~10.27]{RudinFA}}) yield that
\begin{align} \label{eq:series_expansion_CD_matrix}
  \phi(A)
  = \sum_{n \in \mathbb{N}_0}
      a_n  (A - \identity_{\ell^2 (I)})^n
\end{align}
with convergence in the operator norm topology.
Since $\delta \leq \big[ \limsup_{n \to \infty} |a_n|^{1/n} \big]^{-1}$
by the Cauchy-Hadamard formula, there exists $N = N(\phi, \delta) \in \mathbb{N}$
such that $|a_n|^{1/n} \leq 2/\delta$ for all $n \geq N$.
In particular, this implies that there exists $C_{\phi} = C_{\phi} (\delta) > 0$ satisfying
\begin{align} \label{eq:coefficient_estimate}
  |a_n|
  \leq C_{\phi}  \big( 2/\delta \big)^n
\end{align}
for all $n \in \mathbb{N}_0$.

\medskip{}

\textbf{Step 3.} \emph{($L^p_w$-localization of $A -\identity_{\ell^2(I)}$).}
Identify the operator $A - \identity_{\ell^2(I)}$ with the matrix $B \in \mathbb{C}^{I \times I}$
given by $B_{i,i'} = A_{i,i'} - \delta_{i,i'}$ for $i, i' \in I$.
Since $\| A - \identity_{\ell^2 (I)} \|_{\ell^2 \to \ell^2} \leq \eps = k_0^{-1}$,
it follows that $|B_{i, i'} | \leq k_0^{-1}$.
In addition, a direct calculation gives
\[
  \big|B_{i, i'}\big| = \big| A_{i,i'} - \delta_{i, i'} \big|
  \leq \Phi (\lambda_{i'}^{-1} \lambda_i) + \varphi (\lambda_{i'}^{-1} \lambda_i)
  = \Theta (\lambda_{i'}^{-1} \lambda_i ),
  \quad i, i' \in I.
\]
Similarly, it follows that $|B_{i, i'} | \leq \Theta(\lambda_i^{-1} \lambda_{i'} )$.
Thus, $B \dominated \Theta_{k_0}$.

\medskip{}

\textbf{Step 4.} \emph{(Norm convergence of $\phi(A)$).}
By the choice of $C_1 \geq 1$ in Step 1, it follows by an induction argument
and \Cref{eq:CD_matrix_algebra} that $B^n \in \goodMatrices (\Lambda)$, with
\[
  \| B^n \|_{\goodMatrices}
  \leq C_1^n  \| B \|_{\goodMatrices}^n
  \leq \bigg( \frac{C_1}{C_2} \bigg)^n
  = \bigg( \frac{\delta}{4} \bigg)^n,
  \quad n \in \mathbb{N},
\]
where the second equality used that
$\| B \|_{\goodMatrices} \leq \|\Theta_{k_0} \|_{\wienerSt{L^p_w}} \leq C_2^{-1}$; see Step 3.
Combining this, together with \Cref{eq:coefficient_estimate}, yields
\[
  \sum_{n \in \mathbb{N}}
    \big(|a_n|  \big\| (A - \identity_{\ell^2(I)})^n \|_{\goodMatrices} \big)^p
  \leq \sum_{n \in \mathbb{N}}
         \Big( C_{\phi}  (2/\delta)^n  (\delta/4)^n \Big)^p
  \leq C_{\phi}^p \sum_{n \in \mathbb{N}} ( 2^{-p} )^n < \infty.
\]
Since \Cref{lem:CD_matrix_basic} shows that $\goodMatrices(\Lambda)$ is a quasi-Banach space
with $p$-norm $\| \cdot \|_{\goodMatrices}$,
it follows by \Cref{lem:QuasiBanachAbsoluteConvergence} that the series
$\sum_{n \in \mathbb{N}} a_n (A - \identity_{\ell^2(I)} )^n$
converges in $\goodMatrices(\Lambda)$.
Since $\identity_{\ell^2 (I)} \dominated \varphi$, it follows that
$(A - \identity_{\ell^2 (I)})^0 = \identity_{\ell^2 (I)} \in \goodMatrices (\Lambda)$.
Therefore,
\[ \phi(A) =a_0 (A - \identity_{\ell^2 (I)})^0 + \sum_{n \in \mathbb{N}} a_n (A - \identity_{\ell^2(I)} )^n  \in \goodMatrices,
\]
which completes the proof.
\end{proof}

\chapter{Molecular frames and Riesz sequences for coorbit spaces} \label{sec:molecules}
This chapter is devoted to the notion of molecules in coorbit spaces.
The main results obtained show that coorbit spaces and associated sequence spaces
can be decomposed in terms of dual molecules of frames and Riesz sequences.
These results will be used in Section~\ref{sec:applications}
to provide criteria for boundedness of operators on coorbit spaces.

\section{Frames and Riesz sequences}

Let $\mathcal{H}$ be a separable Hilbert space.
A countable family $(g_i)_{i \in I}$ in $\mathcal{H}$ is called
a \emph{Bessel sequence} in $\mathcal{H}$ if there exists $B>0$ such that
\[
  \sum_{i \in I} |\langle f, g_i \rangle |^2
  \leq B \| f \|_{\mathcal{H}}^2
  \quad \text{for all} \quad f \in \mathcal{H}.
\]
Equivalently, $(g_i)_{i \in I}$ is a Bessel sequence
if the \emph{coefficient operator} associated to $(g_i)_{i \in I}$,
\[
  \analysis : \quad
  \mathcal{H} \to \ell^2 (I), \quad
  f \mapsto (\langle f, g_i \rangle)_{i \in I}
\]
is well-defined and bounded.
The \emph{reconstruction operator} $\synthesis := \analysis^* : \ell^2 (I) \to \mathcal{H}$
associated to $(g_i)_{i \in I}$ is given by $\synthesis (c_i)_{i \in I} = \sum_{i \in I} c_i g_i$.
The \emph{frame} and \emph{Gramian operator} associated to $(g_i)_{i \in I}$ are defined by
$\frameop := \synthesis \analysis : \mathcal{H} \to \mathcal{H}$
and $\gramian := \analysis \synthesis : \ell^2(I) \to \ell^2 (I)$, respectively.

A Bessel sequence $(g_i)_{i \in I}$ is called a \emph{frame} for $\mathcal{H}$,
if there exist $A, B > 0$ (called \emph{frame bounds}) satisfying
\begin{align}\label{eq:frame_ineq}
  A  \| f \|_{\CalH}^2
  \leq \sum_{i \in I} |\langle f, g_i \rangle |^2
  \leq B  \| f \|_{\mathcal{H}}^2
  \quad \text{for all} \quad f \in \mathcal{H};
\end{align}
this holds if and only if the frame operator $\frameop : \mathcal{H} \to \mathcal{H}$
is bounded and invertible.
Two Bessel sequences $(g_i)_{i \in I}$ and $(h_i)_{i \in I}$ are said to be
\emph{dual frames} for $\mathcal{H}$ if
\[
  f
  = \sum_{i \in I} \langle f, g_i \rangle h_i
  = \sum_{i \in I} \langle f, h_i \rangle g_i
  \quad \text{for all} \quad f \in \mathcal{H}.
\]
If $(g_i)_{i \in I}$ is a frame for $\mathcal{H}$,
then $(\frameop^{-1} g_i)_{i \in I}$ is a dual frame of $(g_i)_{i \in I}$,
called the \emph{canonical dual frame}.
A frame is called \emph{Parseval} if \eqref{eq:frame_ineq} holds with equality,
i.e., if $\frameop = \identity_{\mathcal{H}}$.  If $(g_i)_{i \in I}$ is a frame for $\mathcal{H}$,
then the system $(\frameop^{-1/2} g_i)_{i \in I}$ is a Parseval frame for $\mathcal{H}$.

A Bessel sequence $(g_i)_{i \in I}$ is called a \emph{Riesz sequence} in $\mathcal{H}$,
if there exist $A,B > 0$ (called \emph{Riesz bounds}) satisfying
\begin{align}\label{eq:riesz_ineq}
  A  \| c \|_{\ell^2}^2
  \leq \bigg\| \sum_{i \in I} c_i g_i \bigg\|_{\mathcal{H}}^2
  \leq B  \| c \|_{\ell^2}^2
  \quad \text{for all} \quad c = (c_i)_{i \in I} \in \ell^2 (I).
\end{align}
Equivalently, $(g_i)_{i \in I}$ is a Riesz sequence if and only if the Gramian operator
$\gramian : \ell^2 (I) \to \ell^2(I)$ is bounded and invertible.
If $(g_i)_{i \in I}$ is a Riesz sequence in $\mathcal{H}$,
then it admits a unique biorthogonal system
$(h_i)_{i \in I}$ in $\overline{\Span} \{ g_i  :  i \in I \}$.
A Riesz sequence $(g_i)_{i \in I}$ is a frame for $\overline{\Span} \{ g_i  :  i \in I \}$,
and $(\frameop^{-1/2} g_i)_{i \in I}$ is an orthonormal sequence in $\mathcal{H}$,
where $\frameop $ is the frame operator considered
as an operator on $\overline{\Span} \{ g_i  :  i \in I \}$.

For background, proofs, and further properties, see, e.g.,
the books \cite{christensen2016introduction, young2001introduction}.

\section{Molecular systems and their basic properties}
Throughout this section, $g \in \Bwp$ denotes an admissible vector and $\CalK_g := V_g (\Hpi)$
the associated reproducing kernel Hilbert space.

The following definition introduces the central notion of this section.

\begin{definition}\label{def:CoorbitMolecules}
  Let $w : G \to [1,\infty)$ be a $p$-weight for some $p \in (0,1]$.
  Let $\Lambda = (\lambda_i)_{i \in I}$ be relatively separated in $G$.
  \begin{enumerate}[label=(\alph*)]
    \item A family $(h_{i} )_{i \in I}$ in $\Hpi$ is a
          system of \emph{$(L^p_w,g)$-molecules} in $\Hpi$
          if there exists a symmetric \emph{envelope} $\Phi \in \wienerStC{L^p_w}$ such that
          \begin{align}
            |V_g h_{i} (x) |
            \leq \Phi (\lambda_i^{-1} x)
            \label{eq:molecule_envelope}
          \end{align}
          for all $i \in I$ and $x \in G$.

  \item A family  $(H_i)_{i \in I}$ in $\CalK_g$ is a system of
        \emph{$L^p_w$-molecules} in $\CalK_g$
        if there exists a symmetric \emph{envelope} $\Phi \in \wienerStC{L^p_w}$ such that
        \begin{align}
          |H_{i} (x) | \leq \Phi (\lambda_i^{-1} x)
          \label{eq:molecule_envelope2}
        \end{align}
        for all $i \in I$ and $x \in G$.
  \end{enumerate}
  If condition \eqref{eq:molecule_envelope} (resp.\ \eqref{eq:molecule_envelope2}) holds,
  this will be indicated using the notation $(h_i)_{i \in I} \dominated \Phi$
  (resp.\ $(H_i)_{i \in I} \dominated \Phi$).
\end{definition}

\begin{remark}
  A family $(H_i)_{i \in I}$ is a system of $L^p_w$-molecules in $\CalK_g$
  if and only if $(V_g^{-1} H_i)_{i \in I}$ is a system of $(L^p_w,g)$-molecules in $\Hpi$.
\end{remark}

The molecule condition \eqref{eq:molecule_envelope} is
independent of the choice of the admissible vectors $g, h \in \Bwp$, as long as the
matrix coefficient $V_g h$ is well-localized.

\begin{lemma}\label{lem:MoleculeConditionIndependence}
If $g,h \in \Bwp$ are admissible with $V_g h \in \wienerSt{L_w^p}$,
and if $(h_i)_{i \in I} \subseteq \Hpi$ is a system of $(L_w^p, g)$ molecules,
then it is also a system of $(L_w^p, h)$ molecules.
\end{lemma}

\begin{proof}
Since $(h_i)_{i \in I}$ is a system of $(L_w^p, g)$ molecules,
there exists a symmetric envelope $\Phi \in \wienerStC{L_w^p}$ satisfying
$|V_g h_i| \leq \translationL{\lambda_i} \Phi$ for all $i \in I$.
Since $w$ is a $p$-weight, $[L_w^p]^{\vee} = L_w^p$, whence the assumption
 $V_g h \in \wienerSt{L_w^p}$ implies by \Cref{lem:basic_cocycle}
that $\Psi_0 := |V_g h| + |V_h g| + \Phi \in \wienerSt{L_w^p}$ is symmetric.
Therefore, \Cref{cor:UserFriendlyConvolutionBounds} and the elementary identity
$(F_1 \ast F_2)^{\vee} = F_2^{\vee} \ast F_1^{\vee}$ imply that
$\Psi :=  \Psi_0 \ast \Psi_0 \in \wienerStC{L_w^p}$ is symmetric.
Using the reproducing formula $V_h h_i = V_g h_i \ast_{\sigma} V_h g$ (cf. \Cref{lem:basic_cocycle}), it follows that
\begin{align*}
  |V_h h_i (x)|
  & \leq  \big( |V_g h_i| \ast |V_h g| \big) (x)
    \leq  \big( \translationL{\lambda_i} \Phi \ast |V_h g| \big) (x)
    \leq  \translationL{\lambda_i} [\Phi \ast |V_h g|](x) \\
  & \leq  \translationL{\lambda_i} [\Psi_0 \ast \Psi_0](x)
    =    \Psi(\lambda_i^{-1} x),
\end{align*}
where it is used that $(\translationL{\lambda} F_1) \ast F_2 = \translationL{\lambda} [F_1 \ast F_2]$. This shows that $(h_i)_{i \in I}$ is a family of $(L_w^p, h)$ molecules.
\end{proof}

In light of \Cref{lem:MoleculeConditionIndependence}, a system $(h_i)_{i \in I}$ satisfying
the molecule condition \eqref{eq:molecule_envelope} will simply be referred to as a system of $L^p_w$-molecules in $\Hpi$.

The following lemma shows that the molecule property is preserved
under the action of convolution-dominated matrices and integral operators.

\begin{lemma}\label{lem:CDMapsMoleculesToMolecules}
  Let $(H_i)_{i \in I}$ be a system of $L^p_w$-molecules in $\CalK_g$
  indexed by the relatively separated family $\Lambda = (\lambda_i)_{i \in I}$.
  Then the following hold:
  \begin{enumerate}[label=(\roman*)]
    \item If $H : G \times G \to \mathbb{C}$ is $L_w^p$-localized in $\CalK_g$
          (see \Cref{def:CDOperator}), then the family
          $(\widetilde{H}_i)_{i \in I} \subseteq \CalK_g$ defined by
          \[
            \widetilde{H}_i (x)
            := (T_H H_i) (x) = \int_G H(x,y) H_i (x) \; d\mu_G (y)
          \]
          is also a system of $L_w^p$-molecules in $\CalK_g$.

    \item If $\Lambda = (\lambda_i)_{i \in I}$ and $\Gamma = (\gamma_j)_{j \in J}$
          are relatively separated families in $G$ and if the matrix
          $A = (A_{j,i})_{(j,i) \in J \times I}\in \goodMatrices (\Gamma, \Lambda)$,
          then the family $(H'_j)_{j \in J} \subseteq \CalK_g$ defined by
          \[
            H'_j
            := A (H_i)_{i \in I}
            := \sum_{i \in I} A_{j, i} H_i
          \]
          is also a system of $L^p_w$-localized molecules in $\mathcal{K}_g$.
  \end{enumerate}
\end{lemma}

\begin{proof}
  Let $\Phi \in \wienerStC{L_w^p}$ be a symmetric envelope
  such that $(H_i)_{i \in I} \dominated \Phi$.

  (i)
  Let $\Theta \in \wienerStC{L_w^p}$ be a symmetric envelope for $H$.
  An application of \Cref{lem:CDElementary} shows that
  $\widetilde{H}_i = T_H H_i \in \CalK_g$ for all $i \in I$.
  Moreover, it holds that
  \[
    |\widetilde{H}_i (x)|
    \leq \int_{\group} \Theta(y^{-1} x) \Phi(\lambda_i^{-1} y) \dd{y}
    =    (\Phi \ast \Theta) (\lambda_i^{-1} x)
  \]
  and similarly
  \(
    |\widetilde{H}_i(x)|
    \leq    (\Theta \ast \Phi) (x^{-1} \lambda_i) .
  \)
  Therefore, using the convolution relation of \Cref{cor:UserFriendlyConvolutionBounds},
  it follows that $\Phi \ast \Theta + \Theta \ast \Phi \in \wienerStC{L_w^p}$ is a (symmetric)
  envelope for $(\widetilde{H}_i)_{i \in I}$.

  (ii)
  Let $\Theta \in \wienerStC{L^p_w}$ be such that $A \dominated \Theta$.
  For $x \in G$ and $j \in J$,
  an application of \Cref{lem:StandardShiftedSeriesEstimates} gives
  \[
    |H'_j (x)|
    \leq \sum_{i \in I}
           |\Theta (\gamma_j^{-1} \lambda_i) | |\Phi (\lambda_i^{-1} x) |
    \leq \frac{\rel(\Lambda)}{\mu_G (Q)}
         \big( \maxL \Theta \ast \maxR \Phi \big) (\gamma_j^{-1} x).
  \]
  Since $\Phi, \Theta \in \wienerStC{L^p_w}$,
  it follows that also $\maxL \Theta, \maxR \Phi \in \wienerStC{L^p_w}$
  (cf.\ Section \ref{sec:AmalgamPrelims}), and hence
  $\Psi := \maxL \Theta \ast \maxR \Phi \in \wienerStC{L^p_w} + (\maxL \Theta + \maxR \Phi)^{\vee}$
  by \Cref{cor:UserFriendlyConvolutionBounds} and because of
  $[\wienerStC{L_w^p}]^{\vee} = \wienerStC{L_w^p}$,
  which holds since $w$ is a $p$-weight.
  Thus, $\Psi$ is a symmetric envelope for $(H_j')_{j \in J}$.
\end{proof}

The frame operator and Gramian associated to a system of molecules
are convolution-dominated operators, as shown next.

\begin{lemma}\label{lem:MoleculesFrameOperatorExpression}
  Let $(H_i)_{i \in I}$ be a system of $L^p_w$-molecules in $\CalK_g$,
  indexed by the relatively separated family $\Lambda = (\lambda_i)_{i \in I}$.
  Then the following hold:
  \begin{enumerate}[label=(\roman*)]
    \item The kernel
          \begin{equation}
            H : \quad
            \group \times \group \to \CC, \quad
            (x,y) \mapsto \sum_{i \in I} H_i(x) \overline{H_i(y)}
            \label{eq:FrameOperatorKernel}
          \end{equation}
          is well-defined (with absolute convergence of the series)
          and $L_w^p$-localized in $\CalK_g$.

    \item The family $(H_i)_{i \in I} \subseteq \CalK_g$ is Bessel and the frame operator
          $\frameop : \CalK_g \to \CalK_g$ is given by $\frameop = T_H |_{\CalK_g}$.

    \item If $\Phi \in \wienerStC{L^p_w}$ is symmetric and such that
          $(H_i)_{i \in I} \dominated \Phi$,
          then the Gramian matrix $\gramian = (\langle H_{i'}, H_i \rangle)_{i, i' \in I}$
          is an element of $\goodMatrices(\Lambda)$
          with envelope $\Phi \ast \Phi \in \wienerStC{L^p_w}$.
  \end{enumerate}
\end{lemma}

\begin{proof}
  Let $\Phi \in \wienerStC{L_w^p}$ be a symmetric envelope for $(H_i)_{i \in I}$.

  (i)
  For $x,y \in G$, an application of \Cref{eq:ShiftedSeriesEstimateTwoTerms} gives
  \[
    |H(x,y)|
    \leq \sum_{i \in I}
           |H_i (x)| \, |H_i(y)|
    \leq \sum_{i \in I}
           \Phi(\lambda_i^{-1} x) \Phi(y^{-1} \lambda_i)
    \leq \frac{\rel(\Lambda)}{\haarMeasure(Q)}
         \big( \maxL \Phi \ast \maxR \Phi \big) (y^{-1} x)
    .
  \]
  Since $\Phi \in \wienerStC{L_w^p}$, also $\maxL \Phi, \maxR \Phi \in \wienerStC{L_w^p}$;
  see \Cref{sec:AmalgamPrelims}.
  Thus, \Cref{cor:UserFriendlyConvolutionBounds} shows that
  $\Theta:= \maxL \Phi \ast \maxR \Phi \in \wienerStC{L_w^p}$ as well
  and in particular that $\big( \maxL \Phi \ast \maxR \Phi \big) (y^{-1} x) < \infty$.
  Thus, the series defining $H(x,y)$ converges absolutely.
  Since $|H(x,y)| = |H(y,x)|$, the above estimate also shows that
  $H \dominated \maxL \Phi \ast \maxR \Phi$.
  In addition, since $\overline{H(x,\cdot)} = H(\cdot,x)$,
  to show that $H$ is $L^p_w$-localized in $\CalK_g$,
  it remains to show that $H(\cdot, y) \in \CalK_g$ for all $y \in \group$.
  For this, first note as a consequence of \Cref{eq:ShiftedSeriesEstimateOneTerm}
  and \Cref{lem:AmalgamWeightedLInftyEmbedding} that
  \[
    \sum_{i \in I}
      |H_i(y)|
    \leq \sum_{i \in I}
           \Phi(y^{-1} \lambda_i)
    \leq \frac{\rel(\Lambda)}{\haarMeasure(Q)} \| \Phi \|_{\wienerL{L^1}}
    \lesssim \frac{\rel(\Lambda)}{\haarMeasure(Q)} \| \Phi \|_{\wienerSt{L_w^p}}
    < \infty .
  \]
  Furthermore, \Cref{lem:AmalgamWeightedLInftyEmbedding} also shows
  \(
    \| H_i \|_{L^2}
    \leq \| \Phi(\lambda_i^{-1} \cdot) \|_{L^2}
    \leq \| \Phi \|_{L^2}
    \lesssim \| \Phi \|_{\wienerSt{L_w^p}} ,
  \)
  so that it follows that
  \(
    \sum_{i \in I}
      |H_i (y)| \, \| H_i \|_{L^2}
    < \infty .
  \)
  This shows that the series defining $H(\cdot,y)$ converges in $L^2(G)$.
  Since $H_i \in \CalK_g$ for all $i \in I$ and since $\CalK_g   \subseteq L^2(G)$
  is closed, it follows that $H(\cdot,y) \in \CalK_g$.

  \medskip{}

  (ii) For showing that $(H_i)_{i \in I}$ is a Bessel sequence in $L^2 (G)$,
  let $c = (c_i)_{i \in I} \in \ell^2(I)$ be arbitrary.
  Recall from \Cref{eq:ShiftedSeriesEstimateOneTerm}
  that $\sum_{i \in I} \Phi(x^{-1} \lambda_i) \lesssim 1$.
  Using the Cauchy-Schwarz inequality and the estimate
  $|H_i (x)| \leq \Phi(\lambda_i^{-1} x) = \Phi(x^{-1} \lambda_i)$,
  it follows therefore that
  \begin{align*}
    \int_{\group}
      \bigg(
        \sum_{i \in I}
          |c_i|  |H_i(x)|
      \bigg)^2
    \dd{x}
    & \leq \int_{\group}
             \bigg(
               \sum_{i \in I} |c_i|^2 \, \Phi(\lambda_i^{-1} x)
             \bigg)
             \bigg(
               \sum_{i \in I}
                \Phi(x^{-1} \lambda_i)
             \bigg)
           \dd{x} \\
    & \lesssim \sum_{i \in I} |c_i|^2 \| \Phi \|_{L^1}
      \lesssim \| c \|_{\ell^2}^2
      .
  \end{align*}
  This shows that $(H_i)_{i \in I} \subseteq L^2(G)$ is a Bessel sequence.
  Lastly, by the first part of the proof and \Cref{lem:CDElementary},
  it follows that $T_H |_{\CalK_g} : \CalK_g \to \CalK_g$ is well-defined and bounded.
  Let $F \in \CalK_g$ and $x \in \group$.
  As shown in the proof of (i), the series
  $\overline{H(x,\cdot)} = \sum_{i \in I} \overline{H_i(x)} \, H_i$
  converges in $L^2(G)$, and thus
  \[
    T_H F (x)
    = \langle F, \overline{H(x,\cdot)} \rangle_{L^2}
    = \sum_{i \in I}
        \langle F, \overline{H_i(x)} \, H_i \rangle_{L^2}
    = \sum_{i \in I}
        \langle F, H_i \rangle_{L^2} \, H_i (x)
    = \frameop F (x),
  \]
  so that $\frameop = T_H|_{\CalK_g}$, as claimed.

  \medskip{}

  (iii)
  For $i, i' \in I$, a direct calculation gives
  \begin{align*}
    |\langle H_{i'}, H_i \rangle |
    &=    |\langle H_i, H_{i'} \rangle |
     \leq \int_G \Phi (\lambda_i^{-1} x) \Phi(\lambda_{i'}^{-1} x) \; d\mu_G (x)
     =    \int_G \Phi (x^{-1} \lambda_i) \Phi(\lambda_{i'}^{-1} x) \; d\mu_G (x) \\
    &= (\Phi \ast \Phi)(\lambda_{i'}^{-1} \lambda_i),
  \end{align*}
  where the penultimate equality used the symmetry of $\Phi$.
  By \Cref{cor:UserFriendlyConvolutionBounds}, $\Phi \ast \Phi \in \wienerStC{L^p_w}$,
  and $(\Phi \ast \Phi)^{\vee} = \Phi^{\vee} \ast \Phi^{\vee} = \Phi \ast \Phi$.
  Thus, $(\langle H_{i'}, H_i \rangle)_{i,i' \in I} \in \goodMatrices(\Lambda)$.
\end{proof}

\section{Dual frames of molecules}
\label{sec:dual_frame_molecule}

Henceforth, let $w : G \to [1,\infty)$ be a $p$-weight for some $p \in (0,1]$ and let
$g \in \Bwp$ be an admissible vector.
Recall that $\CalK_g := V_g (\Hpi)$ is a reproducing kernel Hilbert space (cf.\ \Cref{eq:RKHS})
with reproducing kernel given by
\[
  K(x,y)
  = V_g [\pi(y) g](x)
  = \overline{\twisttranslationL{x} [V_g g] (y)},
\]
see \Cref{eq:RK}.
For fixed $x \in G$, the notation
\[
  K_x
  := \overline{K(x, \cdot)}
   = \twisttranslationL{x} [V_g g]
   = V_g [\pi(x) g]
\]
will be used.

The following result proves the existence of \emph{almost tight} frames
of (weighted) reproducing kernels.
For the notion of a \emph{disjoint cover} that appears in the statement,
cf.\ Section~\ref{sec:discrete}.

\begin{proposition}\label{prop:AlmostTightFrames}
  For every $\eps \in (0,1)$, there exists a compact, symmetric unit neighborhood $U \subseteq Q$
  with the following property:

  If $\Lambda = (\lambda_i)_{i \in I}$ is a relatively separated, $U$-dense family in $\group$
  and $(U_i)_{i \in I}$ is a disjoint cover of $\group$ associated to $U$ and $\Lambda$,
  then the family $\big( \sqrt{\haarMeasure(U_i)} \cdot K_{\lambda_i} \big)_{i \in I}$
  is a frame for $\CalK_g$ with lower frame bound $1 - \eps$ and upper frame bound $1 + \eps$.
\end{proposition}

\begin{proof}
  The proof is divided into three steps.

  \medskip{}

  \textbf{Step~1.}
  In this step, it will be shown that there exists a measurable map
  $\tau : G \times G \to \T$ such that
  \begin{align} \label{eq:step1_frame_rkhs}
    \| K_x - \tau (x,y) K_y \|_{L^2} \to 0
    \quad \text{as} \quad y^{-1} x \to e_G.
  \end{align}
  For showing \eqref{eq:step1_frame_rkhs},
  consider the map $( (x,y), \tau) \mapsto  - \| K_x - \tau \, K_y \|_{L^2}$
  from $(\group \times \group) \times \T$ into $\R$,
  which is clearly measurable with respect to $(x,y)$ and continuous with respect to $\tau$.
  A straightforward application of the measurable maximum theorem
  (see, e.g., \cite[Theorem~18.19]{AliprantisBorderHitchhiker}) yields a measurable map
  $\tau : \group \times \group \to \T$ satisfying
  \[
    \| K_x - \tau(x,y) K_y \|_{L^2}
    = \min_{z \in \T} \| K_x - z \, K_y \|_{L^2}
    \qquad \text{for all} \quad x,y \in \group.
  \]
  Since $\group$ is first countable,  it suffices to show
  $\| K_{x_n} - \tau(x_n,y_n) K_{y_n} \|_{L^2} \to 0$ as $n \to \infty$,
  for arbitrary sequences $(x_n)_{n \in \N},(y_n)_{n \in \N} \subseteq \group$
  satisfying $y_n^{-1} x_n \to e_{\group}$.

  Let $\mathcal{U} := \mathcal{U} (\Hpi)$ denote the group of unitary operators on $\Hpi$,
  equipped with the strong topology.
  Furthermore, let $\mathcal{P} := \mathcal{U} / (\T \, \identity_{\Hpi})$ and let
  $\varrho : \mathcal{U} \to \mathcal{P}$ denote the canonical projection.
  Since $\pi$ is a $\sigma$-representation,
  an application of \cite[Theorem~7.5]{varadarajan1985geometry} implies that the quotient map
  $[\pi] : \group \to \mathcal{P}, x \mapsto \varrho(\pi(x))$ is continuous.
  Thus, $\varrho(\pi(x_n^{-1} y_n)) \to \varrho(\pi(e_{\group})) = \varrho(\identity_{\Hpi})$,
  and an application of \cite[Lemma~7.1]{varadarajan1985geometry} yields a sequence
  $(z_n)_{n \in \N}$ of numbers $z_n \in \T$
  satisfying $z_n  \pi(x_n^{-1} y_n) \to \identity_{\Hpi}$ in the strong topology.
  Since $V_g : \Hpi \to L^2(G)$ is an isometry and $K_x = V_g [\pi(x) g]$,
  it follows that
  \begin{align*}
    \big\| K_{x_n} - \tau(x_n, y_n) K_{y_n} \big\|_{L^2}
    & = \inf_{z \in \T}
          \big\| K_{x_n} - z  K_{y_n} \big\|_{L^2}
      = \inf_{z \in \T}
          \big\| \pi(x_n) g - z  \pi(y_n) g \big\|_{\Hpi} \\
    & = \inf_{z \in \T}
          \big\|
            g - z \, \overline{\sigma(x_n^{-1}, x_n)} \sigma(x_n^{-1}, y_n) \pi(x_n^{-1} y_n) g
          \big\|_{\Hpi} \\
    & \leq \big\| g - z_n  \pi(x_n^{-1} y_n) g \big\|_{\Hpi}
      \to 0,
  \end{align*}
  where the penultimate step used that $\pi(x_n)$ is unitary and that
  \[
    \pi(x_n)^{-1} \pi(y_n)
    = \overline{\sigma(x_n^{-1}, x_n)} \sigma(x_n^{-1}, y_n) \pi(x_n^{-1} y_n)
    .
  \]
  Therefore,
  \(
    \big\| K_{x_n} - \tau(x_n, y_n) K_{y_n} \big\|_{L^2}
    \leq \big\| g - z_n  \pi(x_n^{-1} y_n) g \big\|_{\Hpi}
    \to 0
  \)
  as $n \to \infty$.
  \medskip{}

  \textbf{Step~2.}
  This step shows that also $\| K_x - \tau(x,y) K_y \|_{L^1} \to 0$ as $y^{-1} x \to e_{\group}$.
  For this, set $\Theta := |V_g g|$ and note by \Cref{lem:AmalgamWeightedLInftyEmbedding}
  and because of $V_g g \in \wienerSt{L_w^p}$ and $w \geq 1$ that
  \(
    \Theta
    \in \wienerStC{L_w^p}
    \hookrightarrow \wienerL{L_w^1}
    \hookrightarrow L^1
    ;
  \)
  see also \Cref{eq:VoiceTransformAbsoluteValueContinuous}.
  Let $\delta > 0$ be arbitrary and choose a compact set $\Omega \subseteq \group$ satisfying
  $\int_{\group \setminus \Omega} \Theta d\haarMeasure \leq \frac{\delta}{4}$.
  By Step~1, there exists a compact, symmetric unit neighborhood $V \subseteq Q$ such that
  for all $x, y \in G$ satisfying $y^{-1} x \in V$,
  \[
    \| K_x - \tau(x,y) K_y \|_{L^2}
    \leq \frac{\delta/2}{1 + \sqrt{\haarMeasure(Q \Omega)}}
    \leq \frac{\delta / 2}{1 + \sqrt{\haarMeasure{(V \Omega)}}}.
  \]
  Using the estimate $|K_x(y)| \leq \Theta(x^{-1} y)$ and the Cauchy-Schwarz inequality,
  it follows therefore that, for $x,y \in \group$ with $y^{-1} x \in V$,
  \begin{align*}
    & \| K_x - \tau(x,y) K_y \|_{L^1} \\
    & \leq \int_{\group \setminus x V \Omega}
             \Theta(x^{-1} z) + \Theta(y^{-1} z)
           \dd{z}
           + \int_{\group}
               \indicator_{x V \Omega} (z)
                |K_x(z) - \tau(x,y) K_y(z)|
             \dd{z} \\
    & \leq
           2 \int_{\group \setminus \Omega} \Theta(w) \dd{w}
          + \| \indicator_{x V \Omega} \|_{L^2}
             \| K_x - \tau(x,y) K_y \|_{L^2} \\
    & \leq \frac{\delta}{2}
           + \sqrt{\vphantom{h}\smash{\haarMeasure(V \Omega)}}
              \frac{\delta / 2}{1 + \sqrt{\haarMeasure(V \Omega)}}
      \leq \delta ,
  \end{align*}
  where the second inequality used the change of variables $w = x^{-1} z$,
  respectively $w = y^{-1} z$, and the inclusions
  $\group \setminus V \Omega \subseteq \group \setminus \Omega$
  and $\group \setminus y^{-1} x V \Omega \subseteq \group \setminus \Omega$,
  which holds since $x^{-1} y \in V$.

  \medskip{}

  \textbf{Step~3.}
  Let $\eps > 0$, and choose $\delta > 0$ so small that $1 + \sqrt{\delta} \leq \sqrt{1 + \eps}$
  and $1 - \sqrt{\delta} \geq \sqrt{1 - \eps}$.
  By Step~2, there exists a compact symmetric unit neighborhood $U \subseteq Q$ satisfying
  \[
    \| K_x - \tau(x,y) K_y \|_{L^1}
    \leq \delta \big/ \bigl(1 + 2 \| \maxL \Theta \|_{L^1}\bigr)
  \]
  for all $x,y \in \group$ satisfying $y^{-1} x \in U$.
  Given this choice of $U$, let $\Lambda = (\lambda_i)_{i \in I}$ and $(U_i)_{i \in I}$
  as in the statement of the proposition.
  Let $F \in \CalK_g$ be arbitrary.
  Since the family $(U_i)_{i \in I}$ is pairwise disjoint and satisfies
  $\group = \bigcup_{i \in I} U_i$, a direct calculation entails
  \begin{align*}
    \bigg|
      \| F \|_{L^2}
      - \bigg(
          \sum_{i \in I}
            \big|
              \big\langle
                F,
                \sqrt{\vphantom{h} \smash{\haarMeasure{(U_i)}}} \, K_{\lambda_i}
              \big\rangle
            \big|^2
        \bigg)^{\frac{1}{2}}
    \bigg|
    & = \bigg|
          \big\| \, |F| \, \big\|_{L^2}
          - \bigg\|\!
              \bigg(
                \sum_{i \in I}
                  \big|
                    \big\langle F, \tau(x, \lambda_i) \, K_{\lambda_i} \big\rangle
                  \big|^2
                   \indicator_{U_i} (x) \!
              \bigg)^{\frac{1}{2}}
            \bigg\|_{L^2_x}
        \bigg| \\
    & \leq \bigg\|
             |F(x)|
             - \Big|
                 \sum_{i \in I}
                   \big\langle F, \tau(x, \lambda_i) \, K_{\lambda_i} \big\rangle
                    \indicator_{U_i} (x)
               \Big|
           \bigg\|_{L^2_x} \\
    & \leq \bigg\|
             \sum_{i \in I}
               \indicator_{U_i}(x)
                \big(
                       F(x) - \overline{\tau(x,\lambda_i)} \, F(\lambda_i)
                     \big)
           \bigg\|_{L^2_x}
    \,\, ,
  \end{align*}
  where the notation $\| \cdot \|_{L_x^2}$ is used to indicate that the $L^2$-norm is taken
  with respect to $x \in G$.
  Thus, setting
  \[
    H(x)
    := \bigg|
         \sum_{i \in I}
           \indicator_{U_i} (x)
           \int_{\group}
             F(y)  \overline{\bigl(K_x(y) - \tau(x,\lambda_i) K_{\lambda_i}(y) \bigr)}
          \; \dd{y}
       \bigg| ,
  \]
  it holds that
  \(
    \big|
      \| F \|_{L^2}
      - \big(
          \sum_{i \in I}
            |
              \langle F, \sqrt{\haarMeasure{(U_i)}} \, K_{\lambda_i} \rangle
            |^2
        \big)^{1/2}
    \big|
    \leq \| H \|_{L^2} .
  \)

  Note that if $\indicator_{U_i}(x) \neq 0$,
  then $x \in U_i \subseteq \lambda_i U \subseteq \lambda_i Q$.
  On the one hand, this implies $\lambda_i^{-1} x \in U$ and hence
  \(
    \| K_x - \tau(x,\lambda_i) \, K_{\lambda_i} \|_{L^1}
    \leq \delta / (1 + 2 \| \maxL \Theta \|_{L^1})
    .
  \)
  On the other hand, the above considerations show for $\indicator_{U_i}(x) \neq 0$
  that $y^{-1} \lambda_i = y^{-1} x x^{-1} \lambda_i \in y^{-1} x Q^{-1} = y^{-1} x Q$
  and hence $\Theta(y^{-1} \lambda_i) \leq \maxL \Theta (y^{-1} x)$,
  since $\Theta$ is continuous and $Q$ is open.
  Combining this with the estimate $|K_x(y)| \leq \Theta(y^{-1} x)$
  and the Cauchy-Schwarz inequality, it follows that
  \begin{align*}
    H(x)
    & \leq \sum_{i \in I}
             \indicator_{U_i}(x)
             \int_{\group}
               |F(y)|
                |K_x(y) - \tau(x, \lambda_i) K_{\lambda_i}(y)|^{1/2}
                \big( \Theta(y^{-1} x) + \Theta(y^{-1} \lambda_i) \big)^{1/2}
             \; \dd{y} \\
    & \leq \bigg(\!
             \sum_{i \in I}
               \int_{\group} \!
                 \indicator_{U_i} (x)
                  2 \, \maxL \Theta (y^{-1} x)
                  |F(y)|^2
              \; \dd{y} \!
           \bigg)^{\!\! \frac{1}{2}}
            \bigg(\!
                   \sum_{i \in I}
                     \indicator_{U_i} (x)
                     \, \| K_x - \tau(x,\lambda_i) \, K_{\lambda_i} \|_{L^1} \!
                 \bigg)^{\!\! \frac{1}{2}} \\
    & \leq \bigg( \frac{\delta}{1 + 2 \, \| \maxL \Theta \|_{L^1}} \bigg)^{1/2}
            \bigg(
                   2
                   \int_{\group}
                     |F(y)|^2
                      \maxL \Theta (y^{-1} x)
                   \; \dd{y}
                 \bigg)^{1/2}
           .
  \end{align*}
  Therefore,
  \begin{align*}
    \| H \|_{L^2}^2
    & \leq \frac{2 \delta}{1 + 2 \, \| \maxL \Theta \|_{L^1}}
           \int_{\group}
             \int_{\group}
               |F(y)|^2
                \maxL \Theta (y^{-1} x)
            \; \dd{y}
           \dd{x} \\
    & \leq \frac{2 \delta}{1 + 2 \, \| \maxL \Theta \|_{L^1}}
           \| F \|_{L^2}^2
           \| \maxL \Theta \|_{L^1}
      \leq \delta \, \| F \|_{L^2}^2 ,
  \end{align*}
  and hence
  \[
    \bigg|
      \| F \|_{L^2}
      - \bigg(
          \sum_{i \in I}
            \big|
              \big\langle
                F,
                \sqrt{\vphantom{h} \smash{\haarMeasure(U_i)}} \, K_{\lambda_i}
              \big\rangle
            \big|^2
        \bigg)^{1/2}
    \bigg|
    \leq \| H \|_{L^2}
    \leq \sqrt{\delta} \, \| F \|_{L^2}
    .
  \]
  By the choice of $\delta$, this easily implies
  \[
    (1 - \eps)  \| F \|_{L^2}^2
    \leq \sum_{i \in I}
           |\langle F, \sqrt{\vphantom{h} \smash{\haarMeasure(U_i)}} \, K_{\lambda_i} \rangle|^2
    \leq (1 + \eps)  \| F \|_{L^2}^2
  \]
  for all $F \in \CalK_g$, as required.
\end{proof}

We aim to apply \Cref{thm:CDIntegralOperatorsLocalSpectralInvariance} to prove the existence
of a dual frame for the family $(K_{\lambda_i})_{i \in I}$ that forms a system of molecules.
To this end, we need to construct a frame of molecules $(\tau_i \, K_{\lambda_i})_{i \in I}$
with a fixed envelope $\Phi$ but such that the associated frame operator $\frameop$ satisfies
$\| \identity_{\CalK_g} - \frameop \|_{\CalK_g \to \CalK_g} \leq \eps = \eps(\Phi)$.
In combination with \Cref{prop:AlmostTightFrames},
the following lemma shows that this can indeed be done.

\begin{lemma}\label{lem:AlmostTightFrameEnvelope}
  Let $U \subseteq Q$ be a unit neighborhood and assume
  that $\Lambda = (\lambda_i)_{i \in I} \subseteq \group$ is relatively separated
  and that $(U_i)_{i \in I}$ is a family of measurable sets $U_i \subseteq \lambda_i U$
  satisfying $\group = \bigcupdot_{i \in I} U_i$.

  If $(\tau_i)_{i \in I} \subseteq [0,\infty)$ satisfies $\tau_i \leq C  \haarMeasure(U_i)$
  for all $i \in I$ and some $C > 0$, then the family
  $\bigl(\tau_i^{1/2}  K_{\lambda_i}\bigr)_{i \in I}$
  is a system of $L_w^p$-molecules in $\CalK_g$ (indexed by $\Lambda$)
  and the kernel $H$ in \Cref{eq:FrameOperatorKernel}
  associated to $\bigl(\tau_i^{1/2}  K_{\lambda_i}\bigr)_{i \in I}$ satisfies
  \[
    |H(y,x)|
    = |H(x,y)|
    \leq C  \bigl[\maxL (V_g g) \ast \maxR (V_g g)\bigr](y^{-1} x)
  \]
  for all $x,y \in G$.
\end{lemma}

\begin{proof}
  Set $\Phi := |V_g g| \in \wienerStC{L_w^p}$ and note that $\Phi$ satisfies
  $|K_x(y)| \leq \Phi(x^{-1} y) = \Phi(y^{-1} x)$ for all $x,y \in \group$.
  Since $\tau_i \leq C  \haarMeasure(U_i) \leq C  \haarMeasure(U)$,
  this implies that $\bigl(\tau_i^{1/2}  K_{\lambda_i}\bigr)_{i \in I}$
  is a system  of $L_w^p$-molecules in $\CalK_g$
  with envelope $\sqrt{C  \haarMeasure(U)} \cdot \Phi$.

  Let $x,y \in \group$.
  For $i \in I$ and $z \in U_i \subseteq \lambda_i Q$,
  we have $\lambda_i^{-1} x = \lambda_i^{-1} z z^{-1} x \in Q z^{-1} x$
  and $y^{-1} \lambda_i = y^{-1} z (\lambda_i^{-1} z)^{-1} \in y^{-1} z Q^{-1} = y^{-1} z Q$.
  Since $Q$ is open and $\Phi$ is continuous, this implies
  $|K_{\lambda_i} (x)| \leq \Phi(\lambda_i^{-1} x) \leq \maxR \Phi (z^{-1} x)$
  and $|K_{\lambda_i} (y)| \leq \Phi(y^{-1} \lambda_i) \leq \maxL \Phi (y^{-1} z)$
  for all $i \in I$ and $z \in U_i$.
  Hence, by definition of $H$ in \Cref{eq:FrameOperatorKernel},
  \begin{align*}
    |H(x,y)|
    & \leq C \sum_{i \in I}
             \Big[
               \haarMeasure(U_i)  |K_{\lambda_i} (x)|  |K_{\lambda_i} (y)|
             \Big]
      \leq C \sum_{i \in I}
               \int_{U_i}
                 \maxR \Phi (z^{-1} x)
                 \maxL \Phi (y^{-1} z)
               \dd{z} \\
    & =    C \int_{\group}
               \maxL \Phi(w)
               \maxR \Phi(w^{-1} y^{-1} x)
             \dd{w}
      =    C  \big( \maxL \Phi \ast \maxR \Phi \big) (y^{-1} x).
  \end{align*}
  Since $H(y,x) = \overline{H(x,y)}$, this completes the proof.
\end{proof}

Using the previous two results, we can prove the main result of this subsection.

\begin{theorem}\label{thm:dual_frame_RKHS}
  There exists a compact unit neighborhood $U \subseteq Q$ with the following property:
  If $\Lambda = (\lambda_i)_{i \in I}$ is relatively separated and $U$-dense in $G$,
  the following assertions hold:
  \begin{enumerate}[label=(\roman*)]
      \item The family $(K_{\lambda_i})_{i \in I}$ is a frame for $\CalK_g$
            and admits a dual frame $(H_i)_{i \in I}$ of $L_w^p$-molecules in $\CalK_g$.

      \item There exists a Parseval frame $(F_i)_{i \in I}$ for $\CalK_g$
            that is a system of $L_w^p$-molecules in $\CalK_g$.
  \end{enumerate}
\end{theorem}

\begin{proof}
  Set $C := 1$ and $\Theta := C  \maxL (V_g g) \ast \maxR (V_g g)$.
  Since $|V_g g| \in \wienerStC{L_w^p}$
  and thus also $\maxL (V_g g), \maxR(V_g g) \in \wienerStC{L_w^p}$,
  \Cref{cor:UserFriendlyConvolutionBounds} shows that $\Theta \in \wienerStC{L_w^p}$ as well.
  An application of \Cref{thm:CDIntegralOperatorsLocalSpectralInvariance}
  yields $\eps = \eps(g,Q,w,p) \in (0,1)$ such that
  whenever $H : \group \times \group \to \CC$ is $L_w^p$-localized in $\CalK_g$
  with $H \dominated \Theta$ and with $\| T_H - \identity_{\CalK_g} \|_{\CalK_g \to L^2} \leq \eps$,
  then there exist kernels $H_1,H_2 : \group \times \group \to \CC$
  that are $L_w^p$-localized in $\CalK_g$ and such that
  $(T_H |_{\CalK_g})^{-1} = T_{H_1} |_{\CalK_g}$
  and $(T_H |_{\CalK_g})^{-1/2} = T_{H_2} |_{\CalK_g}$,
  where the operators on the left-hand side are defined by the holomorphic functional calculus.

  For this choice of $\eps$, let the compact unit neighborhood $U \subseteq Q$
  be as provided by \Cref{prop:AlmostTightFrames}.
  Let $\Lambda = (\lambda_i)_{i \in I} \subset\! \group$
  be any relatively separated and $U$-dense family.
  Then $I$ is countable and there exists a disjoint cover $(U_i)_{i \in I}$
  associated to $\Lambda$ and $U$ (cf.\ Section~\ref{sec:discrete}), i.e.,
  pairwise disjoint Borel sets $U_i \subseteq \lambda_i U$
  satisfying $\group = \biguplus_{i \in I} U_i$.
  Set $\tau_i := \haarMeasure(U_i)$ and note
  $\tau_i \leq \haarMeasure(U_i) \leq C  \haarMeasure(U_i)$.
  Therefore, \Cref{lem:AlmostTightFrameEnvelope} shows that the family
  $\bigl(\tau_i^{1/2} \, K_{\lambda_i}\bigr)_{i \in I}$ is a system of $L_w^p$-molecules
  in $\CalK_g$  and that the kernel $H$ in \Cref{eq:FrameOperatorKernel} is $L_w^p$-localized
  in $\CalK_g$ with $H \dominated \Theta$.
  Furthermore, \Cref{lem:MoleculesFrameOperatorExpression} shows that
  $\bigl(\tau_i^{1/2} \, K_{\lambda_i}\bigr)_{i \in I}$ is a Bessel sequence
  and that the associated frame operator $\frameop : \CalK_g \to \CalK_g$ satisfies
  $\frameop = T_H |_{\CalK_g}$.
  Moreover, the choice of $U$ (cf.\ \Cref{prop:AlmostTightFrames}) ensures that
  \[
    - \eps \, \| F \|_{L^2}^2
    \leq \langle (\frameop - \identity_{\CalK_g}) F, F \rangle
    \leq \eps \, \| F \|_{L^2}^2
    \quad \text{for all} \quad \, F \in \CalK_g
    .
  \]
  Since $\frameop$ is self-adjoint and $\frameop = T_H |_{\CalK_g}$, this yields
  \(
    \| T_H - \identity_{\CalK_g} \|_{\CalK_g \to L^2}
    = \| \frameop - \identity_{\CalK_g} \|_{\CalK_g \to L^2}
    \leq \eps
    .
  \)
  By the choice of $\eps$, this implies that
  \[
    \frameop^{-1} = (T_H |_{\CalK_g})^{-1} = T_{H_1} |_{\CalK_g}
    \qquad \text{and} \qquad
    \frameop^{-1/2} = (T_H |_{\CalK_g})^{-1/2} = T_{H_2} |_{\CalK_g}
  \]
  for suitable kernels $H_1, H_2 : \group \times \group \to \CC$
  that are $L_w^p$-localized in $\CalK_g$.

 (i)
  Since $|V_g g| \in \wienerStC{L_w^p}$ by assumption, it follows
  directly from the definitions that $(K_{\lambda_i})_{i \in I}$ is a
  system of $L_w^p$-molecules in $\CalK_g$.
  By \Cref{lem:MoleculesFrameOperatorExpression}, this implies that $(K_{\lambda_i})_{i \in I}$
  is a Bessel family.
  Similarly, since $0 \leq \tau_i \leq \haarMeasure(\lambda_i U) = \haarMeasure(U)$
  for all $i \in I$, it follows  that $\bigl(\tau_i \, K_{\lambda_i}\bigr)_{i \in I}$
  is also a system of $L_w^p$-molecules in $\CalK_g$.
  By the above, $\frameop^{-1} = T_{H_1}|_{\CalK_g}$,
  and thus \Cref{lem:CDMapsMoleculesToMolecules} shows that
  $(H_i)_{i \in I} := \big( \frameop^{-1} [\tau_i \, K_{\lambda_i}] \big)_{i \in I}$
  is also a system of $L_w^p$-molecules in $\CalK_g$, and hence a Bessel sequence
  by \Cref{lem:MoleculesFrameOperatorExpression}.
  Moreover, for arbitrary $F \in \CalK_g$,
  \[
    F
    = \frameop^{-1}
      \bigg(
        \sum_{i \in I}
          \big\langle F, \tau_i^{1/2} \, K_{\lambda_i} \big\rangle \,\,
          \tau_i^{1/2} \, K_{\lambda_i}
      \bigg)
    = \sum_{i \in I}
        \langle F, K_{\lambda_i} \rangle
        \frameop^{-1} (\tau_i \, K_{\lambda_i})
    = \sum_{i \in I}
        \langle F, K_{\lambda_i} \rangle
        H_i .
  \]
  Since $(K_{\lambda_i})_{i \in I}$ and $(H_i)_{i \in I}$ are both Bessel families,
  this implies that they form a pair of dual frames.

  \smallskip{}

  (ii)
  Since $\bigl(\tau_i^{1/2} \, K_{\lambda_i}\bigr)_{i \in I}$ is a system
  of $L_w^p$-molecules in $\CalK_g$, an application of \Cref{lem:CDMapsMoleculesToMolecules}
  shows that the same is true for
  \(
    (F_i)_{i \in I}
    := \bigl(
        \frameop^{-1/2} \, [\tau_i^{1/2} \, K_{\lambda_i} \,]
      \bigr)_{i \in I}
    = \big(
        T_{H_2} [\tau_i^{1/2} \, K_{\lambda_i} \,]
      \big)_{i \in I}
    .
  \)
  The system $(F_i)_{i \in I}$ is a Parseval frame for $\CalK_g$.
\end{proof}

\section{Dual Riesz sequences of molecules}

With notation as in Section~\ref{sec:dual_frame_molecule},
the aim of this section is to establish the existence of
a Riesz sequence $(K_{\lambda_i})_{i \in I}$ of reproducing kernels
\[
  K_{\lambda_i}
  = \twisttranslationL{\lambda_i} [V_g g]
\]
whose biorthogonal system also forms a family of $L^p_w$-molecules in $\CalK_g := V_g (\Hpi)$.

The first result shows the existence of almost tight Riesz sequences of reproducing kernels.
This fact could be cited from the general result \cite[Lemma~6.1]{MoleculePaper},
but the proof is included for the sake of completeness.

\begin{proposition}\label{prop:eps_riesz}
  For every $\eps \in (0,1)$, there exists a compact unit neighborhood $U \supset Q$
  in $G$ with the following property:

  If $\Lambda = (\lambda_i)_{i \in I}$ is any countable family in $G$ that is $U$-separated,
  then the system $(K_{\lambda_i})_{i \in I}$ forms a Riesz sequence in $\CalK_g$
  with lower Riesz bound $(1-\eps)^2 \| V_g g \|_{L^2}^2$
  and upper Riesz bound $(1+\eps)^2 \| V_g g \|^2_{L^2}$.
\end{proposition}

\begin{proof}
Throughout the proof, let $\Phi := |V_g g| \in \wienerStC{L^p_w}$
and define the normalized reproducing kernels $\widetilde{K}_{y} := K_{y} / \| V_g g \|_{L^2}$
for $y \in G$.
We note that
\begin{equation}
  |K_{y} (x)| = |\twisttranslationL{y} [V_g g]| = L_{y} |V_g g|(x) = (\translationL{y} \Phi)(x)
  \quad \text{and hence} \quad
  \| \widetilde{K}_y \|_{L^2} = \frac{\| \translationL{y} \Phi \|_{L^2}}{\| V_g g \|_{L^2}} = 1 .
  \label{eq:RieszProofKernelNormalization}
\end{equation}
The proof is split into three steps.

\medskip{}

\textbf{Step 1.} \emph{(Localized norm estimate).}
The group $G$ being second-countable, there exists an increasing sequence
$(U_n)_{n \in \mathbb{N}}$ of compact sets $U_n \subseteq G$
such that $G = \bigcup_{n \in \mathbb{N}} U_n$.
By Lemma~\ref{lem:AmalgamWeightedLInftyEmbedding},
it follows that $\Phi \in \wienerSt{L^p_w} \hookrightarrow L^2_w \hookrightarrow L^2$.
In addition, $\indicator_{U_n^c} \cdot \Phi \to 0$ as $n \to \infty$,
with pointwise convergence, and $|\indicator_{U_n^c} \cdot \Phi|^2 \leq |\Phi|^2 \in L^1$.
Therefore, an application of Lebesgue's dominated convergence theorem
yields the existence of $n_0 \in \mathbb{N}$ such that
 \(
   \| \indicator_{U_{n_0}^c} \cdot \Phi \|_{L^2}
   \leq \frac{\eps}{2}  \| V_g g \|_{L^2}
   = \frac{\eps}{2}  \| \Phi \|_{L^2}.
 \)

Using \Cref{eq:RieszProofKernelNormalization}, we see for any measurable $U \supset U_{n_0}$ in $G$
and any $y \in G$ that
\begin{align*}
        \Big|  \| K_{y} \|_{L^2}^2 - \| K_{y} \cdot \indicator_{y U} \|_{L^2}^2 \Big|
  &\leq \big\| K_{y} \cdot \indicator_{(y U_{n_0})^c} \big\|_{L^2}^2
  \leq  \big\| (L_{y} \Phi) \cdot (L_{y} \indicator_{U^c_{n_0}} ) \|_{L^2}^2
  =     \big\|  \Phi \cdot  \indicator_{U^c_{n_0}}  \|_{L^2}^2.
\end{align*}
This, combined with the estimate
$\| \indicator_{U_{n_0}^c} \cdot \Phi \|_{L^2} \leq \frac{\eps}{2}  \| V_g g \|_{L^2}$,
gives
\begin{align*}
  \Big| \| \widetilde{K}_{y} \cdot \indicator_{y U} \|_{L^2}^2 - 1 \Big|
  &= \| V_g g \|_{L^2}^{-2}
      \Big|
             \| K_{y} \|_{L^2}^2
             - \| K_{y} \cdot \indicator_{y U} \|_{L^2}^2
           \Big|
  \leq  \bigg(\frac{\eps}{2} \bigg)^2
  .
\end{align*}
Using that $\eps \in (0,1)$, this easily implies
\begin{align}\label{eq:final_estimate_step1}
  \bigg( 1 - \frac{\eps}{2} \bigg)^2
  \leq 1 - \bigg( \frac{\eps}{2} \bigg)^2
  \leq \big\| \widetilde{K}_{y} \cdot \indicator_{y U} \big\|_{L^2}^2
  \leq 1 + \bigg( \frac{\eps}{2} \bigg)^2
  \leq \bigg( 1 + \frac{\eps}{2} \bigg)^2
\end{align}
for arbitrary $y \in G$ and any Borel set $U \supset U_{n_0}$.

\medskip{}

\textbf{Step 2.} \emph{(Construction of compact set).}
For the sequence $(U_n)_{n \in \mathbb{N}}$ of Step~1, let $(\widetilde{U}_n)_{n \in \mathbb{N}}$
be defined by $\widetilde{U}_n := (U_n \overline{Q})^{-1}$.
Since $\Phi$ is symmetric, it follows that if $x \in G$ is such that
\[
  0
  \neq \maxL \big( (\Phi \cdot \indicator_{\widetilde{U}_n^c} )^{\vee} \big) (x)
  =    \big\| (\Phi \cdot \indicator_{\widetilde{U}_n^c} )^{\vee} \big\|_{L^{\infty} (xQ)}
  =    \big\| \Phi \cdot \indicator_{(U_n \overline{Q})^c}  \big\|_{L^{\infty} (xQ)},
\]
then $(U_n \overline{Q})^c \cap x Q \neq \emptyset$.
Therefore, there exists $q \in Q$ such that $xq \in (U_n \overline{Q})^c$,
which implies that $x \notin U_n$.
Hence,
\[
  0
  \leq \maxL \big( (\Phi \cdot \indicator_{\widetilde{U}_n^c} )^{\vee} \big)
  \leq \indicator_{U_n^c} \cdot \maxL (\Phi^{\vee})
  =    \indicator_{U_n^c} \cdot \maxL \Phi
  \to 0
  \quad \text{as} \quad n \to \infty,
\]
with pointwise convergence.
An application of Lebesgue's dominated convergence theorem yields
therefore some $m \in \mathbb{N}$ such that
\(
  \|\Phi \|_{L^1} \| (\Phi \cdot \indicator_{\widetilde{U}_m^c})^{\vee} \|_{\wienerL{L^1}}
  < \| \Phi \|_{L^2}  \mu_G (Q)  (\eps / 2)^2.
\)
Choose a symmetric $\varphi \in C_c (G)$ satisfying $0 \leq \varphi \leq 1$
and $\varphi \equiv 1$ on $\widetilde{U}_{m}$.
Set $U' := \supp \varphi$. For $U_{n_0}$ as in Step~1,
define $U := \overline{Q} \cup U_{n_0} \cup U'$ and $\Theta := \Phi \cdot (1 - \varphi)$.
Then $0 \leq \Theta \leq \Phi \cdot \indicator_{\widetilde{U_{m}^c}}$, and thus
\begin{align} \label{eq:final-estimate_step2}
    \| \Theta \|_{L^1} \| \Theta \|_{\wienerL{L^1}}
  < \| \Phi \|_{L^2}  \mu_G (Q)  \bigg( \frac{\eps}{2} \bigg)^2
\end{align}
by construction.

\medskip{}

\textbf{Step 3.} \emph{($\eps$-Riesz sequences).}
Let $\Lambda = (\lambda_i)_{i \in I}$ be a $U$-separated family in $G$.
To ease notation, set $H_{\lambda_i} := \widetilde{K}_{\lambda_i} \cdot \indicator_{\lambda_i U}$
for $i \in I$.
By construction of $U$, if $x \in (\lambda_i U)^c$,
then $\lambda_i^{-1} x \notin U \supset \supp \varphi$,
so that $1 - \varphi (\lambda_i^{-1} x) = 1$.
Hence,
\[
  \indicator_{(\lambda_i U)^c}
  \leq L_{\lambda_i} (1 - \varphi)
  \quad \text{and} \quad
  |\widetilde{K}_{\lambda_i} - H_{\lambda_i}|
  = \|V_g g \|_{L^2}^{-1}  |K_{\lambda_i}| \cdot \indicator_{(\lambda_i U)^c}
  \leq \|V_g g \|_{L^2}^{-1}  L_{\lambda_i} \Theta.
\]

For showing the Riesz inequalities, let $c = (c_i)_{i \in I} \in c_{00} (I)$ be arbitrary.
Then the above estimate yields
\[
  \Bigg|\,
    \bigg\| \sum_{i \in I} c_i \widetilde{K}_{\lambda_i} \bigg\|_{L^2}
    - \bigg\| \sum_{i \in I} c_i H_{\lambda_i} \bigg\|_{L^2}
    \,
  \Bigg|
  \leq \bigg\|
         \sum_{i \in I}
           c_i  \big( \widetilde{K}_{\lambda_i} - H_{\lambda_i} \big)
       \bigg\|_{L^2}
  \leq \| V_g g \|_{L^2}^{-1}
       \bigg\|
         \sum_{i \in I}
           |c_i | L_{\lambda_i} \Theta
       \bigg\|_{L^2}.
\]
Since $\Lambda$ is $U$-separated and $U \supset Q$, it follows that $\rel(\Lambda) \leq 1$.
Therefore, by \Cref{lem:StandardShiftedSeriesEstimates} and \Cref{eq:final-estimate_step2}
and since $\Theta^{\vee} = \Theta$, it follows that
\[
  \bigg\|
    \sum_{i \in I}
      |c_i| L_{\lambda_i} \Theta
  \bigg\|_{L^2}^2
  \leq (\mu_G (Q))^{-1}
        \| \Theta^{\vee} \|_{\wienerL{L^1}}
        \| \Theta \|_{L^1}
        \| c \|_{\ell^2}^2
  \leq \| \Phi \|_{L^2}
        \bigg(\frac{\eps}{2} \bigg)^2
        \| c \|_{\ell^2}^2
  .
\]
Combining the obtained inequalities thus gives
\[
  \Bigg| \,
    \bigg\|
      \sum_{i \in I}
        c_i \widetilde{K}_{\lambda_i}
    \bigg\|_{L^2}
    - \bigg\|
        \sum_{i \in I}
          c_i H_{\lambda_i}
      \bigg\|_{L^2}
    \,
  \Bigg|
  \leq  \frac{\eps}{2}  \| c \|_{\ell^2}.
\]
The family $(H_{\lambda_i})_{i \in I}$ of vectors
$H_{\lambda_i} = \widetilde{K}_{\lambda_i} \cdot \indicator_{\lambda_i U}$ is orthogonal,
since $(\lambda_i U)_{i \in I}$ is pairwise disjoint.
Therefore, \Cref{eq:final_estimate_step1} implies that
\[
  \bigg( 1 - \frac{\eps}{2} \bigg)  \| c \|_{\ell^2}
  \leq \bigg\| \sum_{i \in I} c_i H_{\lambda_i} \bigg\|_{L^2}
  \leq \bigg( 1 + \frac{\eps}{2} \bigg)  \| c \|_{\ell^2}.
\]
Therefore, an application of the triangle inequality easily yields that
\[
  ( 1 - \eps )  \| c \|_{\ell^2}
  \leq \bigg\| \sum_{i \in I} c_i \widetilde{K}_{\lambda_i} \bigg\|_{L^2}
  \leq  ( 1 + \eps )  \| c \|_{\ell^2},
\]
which easily completes the proof.
\end{proof}

The following theorem establishes the existence of Riesz sequences
$(K_{\lambda_i})_{i \in I}$ that admit a biorthogonal system of $L_w^p$ molecules.

\begin{theorem}\label{thm:dual_riesz_RKHS}
There exists a compact unit neighborhood $U \supset Q$ with the following property:
  If $\Lambda = (\lambda_i)_{i \in I}$ is $U$-separated in $G$,
  the following assertions hold:
  \begin{enumerate}[label=(\roman*)]
    \item The family $(K_{\lambda_i})_{i \in I}$ is a Riesz sequence in $\CalK_g$
          whose unique biorthogonal system $(H_i)_{i \in I}$
          in $\overline{\Span} \{ K_{\lambda_i} : i \in I\}$
          is a family of $L^p_w$-molecules in $\CalK_g$.

    \item There exists an orthonormal sequence $(F_i)_{i \in I}$
          in $\overline{\Span} \{ K_{\lambda_i} : i \in I\}$
          that is a system of $L_w^p$-molecules in $\CalK_g$.
  \end{enumerate}
\end{theorem}

\begin{proof}
 Throughout, let $\Theta := \| V_g g \|_{L^2}^{-2} \cdot \big( |V_g g| \ast |V_g g|)$.
 Then $\Theta \in \wienerStC{L^p_w}$ by Corollary~\ref{cor:UserFriendlyConvolutionBounds}.

 By Theorem~\ref{thm:CDMatricesLocalSpectralInvariance},
 there exists $\eps = \eps(g, p, w, Q) \in (0,1)$ such that,
 for any $Q$-separated family $\Lambda = (\lambda_i)_{i \in I}$ in $G$
 and any $A \in \goodMatrices (\Lambda)$ satisfying $A \dominated \Theta$
 and $\| A - \identity_{\ell^2 (I)} \|_{\ell^2 \to \ell^2} \leq \eps$,
 it holds that $A^{-1} \in \goodMatrices (\Lambda)$ and $A^{-1/2} \in \goodMatrices(\Lambda)$.
 Using \Cref{prop:eps_riesz}, let $U \subseteq G$ be a compact set such that $Q \subseteq U$
 and such that for every $U$-separated family $\Lambda = (\lambda_i)_{i \in I}$ in $G$,
 the family $(\widetilde{K}_{\lambda_i})_{i \in I}$
 of normalized kernels $\widetilde{K}_{\lambda_i} := \| V_g g \|_{L^2}^{-1} \cdot K_{\lambda_i}$
 satisfies the Riesz inequalities
 \begin{align}\label{eq:dual_riesz_auxiliary}
   \bigg( 1 - \frac{\eps}{3} \bigg)^2 \| c \|_{\ell^2}^2
   \leq \bigg\| \sum_{i \in I} c_i \widetilde{K}_{\lambda_i} \bigg\|_{L^2}^2
   \leq \bigg( 1 + \frac{\eps}{3} \bigg)^2 \| c \|_{\ell^2}^2
 \end{align}
 for all $c \in \ell^2 (I)$.

 Let $\Lambda = (\lambda_i)_{i \in I}$ be an arbitrary $U$-separated family in $G$.
 Clearly, $\Lambda$ is $Q$-separated.
 An application of Lemma~\ref{lem:MoleculesFrameOperatorExpression} yields that the Gramian matrix
  \(
    \widetilde{\gramian}
    = \big(
        \langle \widetilde{K}_{\lambda_{i'}}, \widetilde{K}_{\lambda_{i}} \rangle
      \big)_{i, i' \in I}
  \)
 of $(\widetilde{K}_{\lambda_{i}})_{i \in I}$
 satisfies $\widetilde{\gramian} \in \goodMatrices(\Lambda)$
 with $\widetilde{\gramian} \dominated \Theta$.
 Note that $\widetilde{\gramian} = \widetilde{\analysis} \circ \widetilde{\synthesis}$,
 where $\widetilde{\analysis} = \widetilde{\synthesis}^\ast$ and $\widetilde{\synthesis}$ are,
 respectively, the analysis and synthesis operators
 associated to $(\widetilde{K}_{\lambda_i})_{i \in I}$.
 Using \Cref{eq:dual_riesz_auxiliary}, we see
 \begin{align*}
   \big|
    \langle (\widetilde{\gramian} - \identity_{\ell^2(I)}) c, c \rangle
   \big|
   &= \big|
      \langle \widetilde{\synthesis} c, \widetilde{\synthesis} c \rangle
      - \| c \|_{\ell^2}^2
     \big|
   \leq \max
        \big\{
          (1 + \tfrac{\eps}{3})^2 - 1, \quad
          1 - (1 - \tfrac{\eps}{3})^2
        \big\}
         \| c \|_{\ell^2}^2 \\
   &\leq \eps  \| c \|_{\ell^2}^2
 \end{align*}
 and hence $\| \identity_{\ell^2 (I)} - \widetilde{\gramian} \|_{\ell^2 \to \ell^2} \leq \eps$;
 here, we used that $\identity_{\ell^2(I)} - \widetilde{\gramian}$ is self-adjoint.
 Thus, the choice of $\eps$ via \Cref{thm:CDMatricesLocalSpectralInvariance}
 yields that $\widetilde{\gramian}^{-1} \in \goodMatrices(\Lambda)$
 and $\widetilde{\gramian}^{-1/2} \in \goodMatrices(\Lambda)$.

 \medskip{}

(i)
 Let $\overline{\widetilde{\gramian}^{-1}} \in \goodMatrices(\Lambda)$
 be the element-wise conjugate matrix of $\widetilde{\gramian}^{-1}$.
 By \Cref{lem:CDMapsMoleculesToMolecules}, the family
  \(
    (H_{i})_{{i} \in I}
    := \overline{\widetilde{\gramian}^{-1}}
       (\widetilde{K}_{\lambda_{i'}})_{i' \in I}
  \)
 is a system of $L^p_w$-molecules.
 Moreover, \Cref{lem:CD_matrix_basic} yields that the series
 \(
   \widetilde{H}_{i}
   := \sum_{i' \in I}
       (\overline{\widetilde{\gramian}^{-1}})_{i, i'}
       \widetilde{K}_{\lambda_{i'}}
 \)
 is norm convergent in $L^2$,
 and thus 
 \[ \widetilde{H}_{i} \in \overline{\Span} \{ K_{\lambda_{i'}} : i' \in I \}.\]
 For $i \in I$, let $H_{i} := \| V_g g \|_{L^2}^{-1}  \widetilde{H}_{i}$.
 Then also $H_i \in \overline{\Span} \{ K_{\lambda_{i'}} : i' \in I \}$,
 and a direct calculation entails
 \begin{align}
   \langle H_{i'}, K_{\lambda_i} \rangle
   &= \| V_g g \|_{L^2}^{-1}
      \bigg\langle
        \sum_{\ell \in I}
          (\overline{\widetilde{\gramian}^{-1}})_{i', \ell}
          \widetilde{K}_{\lambda_{\ell}}, K_{\lambda_i}
      \bigg\rangle
    = \sum_{i \in I}
        (\overline{\widetilde{\gramian}^{-1}})_{i', \ell}
        \langle
          \widetilde{K}_{\lambda_{\ell}},
          \widetilde{K}_{\lambda_i}
        \rangle \\
   &= \overline{
        \sum_{\ell \in I}
          (\widetilde{\gramian}^{-1})_{i', \ell} (\widetilde{\gramian})_{\ell, i}
      }
    = \overline{
        (\identity_{\ell^2(I)})_{i', i}
      }
    = \delta_{i', i},
 \end{align}
 which shows the desired biorthogonality.

 \medskip{}

(ii)
 By similar arguments as in (i), it follows that the system
 $(F_{i})_{i \in I} := \overline{\widetilde{\gramian}^{-1/2}} (\widetilde{K}_{\lambda_{i'}})_{i' \in I}$
 is $L^p_w$-localized and that $F_{i} \in \overline{\Span} \{ K_{\lambda_{i'}} : i' \in I \}$.
 Since $\widetilde{\gramian} = \widetilde{\gramian}^*$,
 it follows that also $(\widetilde{\gramian}^{-1/2})^* = \widetilde{\gramian}^{-1/2}$,
 and thus $(\overline{\widetilde{\gramian}^{-1/2}})_{i, i'} = (\widetilde{\gramian}^{-1/2})_{i', i}$
 for all $i, i' \in I$.
 Using this, together with
 \(
   \widetilde{\gramian} \widetilde{\gramian}^{-1/2}
   = \widetilde{\gramian}^{-1/2} \widetilde{\gramian}
 \),
 it follows that
 \begin{align*}
   \langle F_{i'}, F_i \rangle
   &= \sum_{j, j' \in I}
        \bigg\langle
          (\overline{\widetilde{\gramian}^{-1/2}})_{i', j} \widetilde{K}_{\lambda_j},
          (\overline{\widetilde{\gramian}^{-1/2}})_{i, j'} \widetilde{K}_{\lambda_{j'}}
        \bigg\rangle
    = \sum_{j,j' \in I}
        (\widetilde{\gramian}^{-1/2})_{j, i'}
        (\widetilde{\gramian}^{-1/2})_{i, j'}
        (\widetilde{\gramian})_{j', j} \\
   &= \sum_{j \in I}
        (\widetilde{\gramian}^{-1/2} \widetilde{\gramian})_{i, j}
        (\widetilde{\gramian}^{-1/2})_{j, i'}
    = (\widetilde{\gramian}^{-1/2} \widetilde{\gramian} \widetilde{\gramian}^{-1/2})_{i, i'}
    = \delta_{i, i'}
 \end{align*}
 for all $i, i' \in I$.
 Hence, the sequence $(F_i)_{i \in I}$ is orthonormal.
\end{proof}

\section{Dual molecular systems for coorbit spaces}

The purpose of this section is to show that the canonical reproducing properties of molecular frames
and Riesz sequences on $\Hpi$ and $\ell^2 (\Lambda)$ extend to coorbit spaces $\Co(Y)$
and associated sequence spaces $Y_d (\Lambda)$.

The first result shows that the analysis (resp.\ synthesis) operator
associated to a system of molecules acts boundedly from (resp.\ into) coorbit spaces.

\begin{proposition}\label{prop:molecule_bounded}
Let $w : G \to [1,\infty)$ be a $p$-weight for some $p \in (0,1]$
and let $g \in \Bwp$ be admissible.
Let $Y$ be an $L_w^p$-compatible quasi-Banach function space on $G$.
Suppose that $(g_{i})_{i \in I}$ is a family of $L^p_w$-molecules in $\Hpi$
indexed by the relatively separated family $\Lambda = (\lambda_i)_{i \in I}$,
with $(g_i)_{i \in I} \dominated_{\Lambda} \Phi$ for a symmetric $\Phi \in \wienerStC{L_w^p}$.
Then the following hold:
\begin{enumerate}[label=(\roman*)]
 \item The coefficient operator
       \[
         \analysis : \quad
         \Co_g(Y) \to Y_d (\Lambda), \quad
         f \mapsto (\langle f, g_{i}\rangle)_{i \in I}
       \]
       is well-defined and bounded with $\| \analysis \| \lesssim \| \Phi \|_{\wienerSt{L_w^p}}$,
       with implied constant depending only on $g,Q,w,p,\rel(\Lambda)$.

\item It holds
      \(
        Y_d (\Lambda)
        \hookrightarrow (\wienerL{Y})_d (\Lambda)
        \hookrightarrow \ell_{1/w}^\infty (\Lambda)
        ,
      \)
      and the reconstruction operator
      \begin{align}\label{eq:coorbit_synthesis}
        \synthesis : \quad
        Y_d (\Lambda) \to \Co_g(Y), \quad
        c = (c_{i})_{i \in I} \mapsto \sum_{i \in I} c_{i} g_{i}
      \end{align}
      is well-defined and bounded,
      with the series \eqref{eq:coorbit_synthesis} converging unconditionally in
      the weak-$\ast$-topology on $\Reservoir_w = \Reservoir_w(g)$
      and with $\| \synthesis \| \lesssim \| \Phi \|_{\wienerSt{L_w^p}}$,
      where the implied constant only depends on $g,Q,w,p,Y$.
\end{enumerate}
\end{proposition}

\begin{proof}

(i)
Since $\wienerSt{L^p_w} \hookrightarrow \wienerL{L^p_w} \hookrightarrow L^1_w (G)$
by \Cref{lem:AmalgamWeightedLInftyEmbedding}, it follows from
\eqref{eq:molecule_envelope} that $g_{i} \in \Hw(g)$.
Therefore, for every $f \in \Reservoir_w$ and $i \in I$,
the pairing $\langle f, g_{i} \rangle_{\Reservoir_w, \Hw} \in \mathbb{C}$ is well-defined.
Using Part~(iii) of \Cref{lem:basic_Rw} and noting that $\Phi$ is symmetric,
it follows from \eqref{eq:molecule_envelope} that
\begin{align*}
  |\langle f, g_{i} \rangle |
  & = |\langle V_g f, V_g g_{i} \rangle_{L^{\infty}_{1/w}, L^1_w} |
    \leq \int_G |V_g f(x)|  \Phi(\lambda_i^{-1} x) \; d\mu_G (x) \\
  & = (|V_g f| \ast \Phi^{\vee})(\lambda_i)
    = (|V_g f| \ast \Phi)(\lambda_i)
\end{align*}
for every $i \in I$.
If $x \in \lambda_i Q$, then $\lambda_i = x q$ for some $q \in Q$.
Since $Q$ is open and $\Phi$ is continuous, this implies
$\Phi(y^{-1} \lambda_i) = \Phi(y^{-1} xq) \leq \maxL \Phi(y^{-1} x)$ for all $y \in G$.
Thus,
\begin{align*}
  |\langle f, g_{i} \rangle |
  & \leq [|V_g f| \ast \Phi] (\lambda_i)
    =    \int_G  |V_g f(y)|  \Phi (y^{-1} \lambda_i) \; d\mu_G (y) \\
  & \leq \int_G  |V_g f(y)| \maxL \Phi (y^{-1} x) \; d\mu_G (y)
    =    [|V_g f| \ast \maxL \Phi] (x)
\end{align*}
for every $x \in \lambda_i Q$.
Therefore,
\begin{align*}
  \big\| (\langle f, g_{i} \rangle )_{i \in I} \big\|_{Y_d}
  \!= \bigg\|
        \sum_{i \in I}
          |\langle f, g_{i} \rangle | \indicator_{\lambda_i Q}
      \bigg\|_{Y} \!\!
  \leq \bigg\|
         (|V_g f| \ast \maxL \Phi)
         \, \sum_{i \in I}
              \indicator_{\lambda_i Q}
       \bigg\|_{Y} \!\!
  \leq \rel(\Lambda) \, \big\|\, |V_g f | \ast \maxL \Phi \big\|_{Y}
  .
\end{align*}
Since $\Phi \in \wienerSt{L^p_w}$, it also holds $\maxL \Phi \in \wienerSt{L^p_w}$,
with $\| \maxL \Phi \|_{\wienerSt{L_w^p}} \lesssim \| \Phi \|_{\wienerSt{L_w^p}}$,
where the implied constant only depends on $p,w,Q$.
The convolution relation $\wienerL{Y} \ast \wienerSt{L^p_w} \!\hookrightarrow\! \wienerL{Y}$
from condition \ref{enu:CompatibleConvolution} of \Cref{def:compatible} therefore yields that
\begin{align*}
  \big\| (\langle f, g_{i} \rangle )_{i \in I} \big\|_{Y_d}
  & \leq \rel(\Lambda)  \big\|\, |V_g f | \ast \maxL \Phi \big\|_{Y} \\
  & \lesssim_{Y,w,p,Q}\,\,
      \rel(\Lambda)
       \| V_g f \|_{\wienerL{Y}}
       \| \maxL \Phi \|_{\wienerSt{L^p_w}} \\
  & \lesssim_{w,p,Q} \,\,
      \rel(\Lambda)
       \| \Phi \|_{\wienerSt{L_w^p}}
       \| f \|_{\Co_g (Y)},
\end{align*}
and thus $\analysis : \Co_g(Y) \to Y_d (\Lambda)$ is well-defined and bounded,
with $\| \analysis \| \lesssim \rel(\Lambda)  \| \Phi \|_{\wienerSt{L_w^p}}$.

\medskip{}

(ii)
To verify the embedding $Y_d(\Lambda) \hookrightarrow (\wienerL{Y})_d (\Lambda)$,
let $c = (c_i)_{i \in I} \in Y_d(\Lambda)$.
Note for $x,y \in G$ and $i \in I$ satisfying
$\indicator_{xQ}(y) \neq 0 \neq \indicator_{\lambda_i Q}(y)$ that
$x \in y Q^{-1} \subseteq \lambda_i  Q Q$.
Thus,
\[
  \maxL
  \bigg(
    \sum_{i \in I}
      |c_{i} | \indicator_{\lambda_i Q}
  \bigg) (x)
  = \bigg\|
      \indicator_{xQ} \sum_{i \in I} |c_{i} | \indicator_{\lambda_i Q}
    \bigg\|_{L^{\infty}}
  \leq \sum_{i \in I} |c_{i}| \indicator_{\lambda_i QQ} (x).
\]
Since $Y_d (\Lambda, Q) = Y_d (\Lambda, QQ)$ with equivalent quasi-norms
(where the implied constant only depends on $Y$ and $Q$, but not on $\Lambda$;
see \Cref{eq:SequenceSpaceIndependentOfBaseSet}),
this shows that $\sum_{i \in I} |c_{i} | \indicator_{\lambda_i Q} \in \wienerL{Y}$ with
\(
  \big\| \sum_{i \in I} |c_i| \, \indicator_{\lambda_i Q} \big\|_{\wienerL{Y}}
  \leq \big\| \sum_{i \in I} |c_i| \, \indicator_{\lambda_i Q Q} \big\|_Y
  \lesssim \| c \|_{Y_d(\Lambda)}
  .
\)
In addition, since $Y$ is $L_w^p$-com\-pa\-ti\-ble, condition \ref{enu:CompatibleLInfty} of \Cref{def:compatible}
implies that $\wienerL{Y} \hookrightarrow L_{1/w}^\infty$, which then implies the embedding
$(\wienerL{Y})_d (\Lambda) \hookrightarrow (L_{1/w}^\infty)_d (\Lambda) = \ell_{1/w}^\infty (\Lambda)$.

To prove the claim regarding the reconstruction operator $\synthesis$,
note that if $F : G \to \mathbb{C}$ is continuous, then since $Q$ is open
it holds that $|F(x)| = |F(q q^{-1} x)| \leq (\maxR F)(q^{-1} x)$ for all $q \in Q$.
Since $(g_i)_{i \in I} \dominated \Phi$, we thus see
\(
  |V_g g_i (x)|
  \leq \Phi(\lambda_i^{-1} x)
  \leq (\maxR \Phi)(q^{-1} \lambda_i^{-1} x)
  =    (\maxR \Phi) ( (\lambda_i q)^{-1} x)
\)
for all $q \in Q$.
Integrating this estimate over $q \in Q$, we obtain
\begin{align*}
  \mu_G (Q)  |V_g g_{i}(x)|
  %\leq \mu_G (Q) \cdot (\translationL{\lambda_i} \Phi) (x)
  & \leq \int_G
           \indicator_{Q} (q)
           (\maxR \Phi) ( (\lambda_i q)^{-1} x)
           %(\maxR [\translationL{\lambda_i} \Phi])(q^{-1} x)
         \; d\mu_G (q)
    =    \int_G
           \indicator_{Q} (\lambda_i^{-1} \widetilde{q})
           (\maxR \Phi) ( \widetilde{q}^{-1} x)
           %(\maxR [\translationL{\lambda_i} \Phi])(q^{-1} x)
         \; d\mu_G (\widetilde{q}) \\
    %=    L_{\lambda_i} (\indicator_Q \ast \maxR \Phi) (x) \\
  & =    (\indicator_{\lambda_i Q} \ast \maxR \Phi)(x)
\end{align*}
for every $x \in G$.
For $c = (c_{i})_{i \in I} \in Y_d (\Lambda)$, this implies
\[
  \mu_G (Q)
   \sum_{i \in I}
          |c_{i}| \, |V_g g_{i}|
  \leq \bigg(
         \sum_{i \in I}
           |c_{i} | \indicator_{\lambda_i Q}
       \bigg)
       \ast \maxR \Phi.
\]
%For applying the convolution relation
%$\wienerL{Y} \ast \wienerSt{L^p_w} \hookrightarrow \wienerL{Y}$
%of Corollary~\ref{cor:UserFriendlyConvolutionBounds},
An application of the convolution relation
$\wienerL{Y} \ast \wienerSt{L_w^p} \hookrightarrow \wienerL{Y}$
from condition \ref{enu:CompatibleConvolution} of \Cref{def:compatible}
and of the embedding $Y_d (\Lambda) \hookrightarrow (\wienerL{Y})_d (\Lambda)$ from above
shows that
\begin{align*}
  \bigg\| \sum_{i \in I} |c_{i}| \, |V_g g_{i}| \bigg\|_{\wienerL{Y}}
  & \lesssim_Q \bigg\|
                 \bigg(
                   \sum_{i \in I}
                     |c_{i}| \, \indicator_{\lambda_i Q}
                 \bigg)
                 \ast \maxR \Phi
               \bigg\|_{\wienerL{Y}} \\
  & \lesssim_{Y,p,w,Q} \bigg\|
                         \sum_{i \in I}
                           |c_{i} | \indicator_{\lambda_i Q}
                       \bigg\|_{\wienerL{Y}}
                       \| \maxR \Phi \|_{\wienerSt{L^p_w}} \\
  & \lesssim_{p,w,Q} \| c \|_{(\wienerL{Y})_d}  \| \Phi \|_{\wienerSt{L_w^p}} \\
    & \lesssim_{Q,Y} \| c \|_{Y_d(\Lambda)}
                   \| \Phi \|_{\wienerSt{L^p_w}}.
  \numberthis \label{eq:D0_bounded}
\end{align*}
By the solidity of $\wienerL{Y}$, it follows that the map
\(
  \synthesis_0 :
  Y_d (\Lambda) \to \wienerL{Y}, \;
  (c_{i})_{i \in I} \mapsto \sum_{i \in I} c_{i} V_g g_{i}
\)
is well-defined and bounded.

For showing that $\synthesis$ is well-defined and bounded, let $f \in \Hw \subseteq \Hpi$.
Since $V_g : \Hpi \to L^2 (G)$ is an isometry,
\[
  |\langle g_{i}, f \rangle_{\Hpi} |
  = |\langle V_g g_i, V_g f \rangle_{L^2(G)} |
  \leq \int_G |V_g g_{i} (x) |  | V_g f (x) | \; d\mu_G (x).
\]
This, combined with the embedding $\wienerL{Y} \hookrightarrow L^{\infty}_{1/w} (G)$
(see condition \ref{enu:CompatibleLInfty} of \Cref{def:compatible})
and the estimate \eqref{eq:D0_bounded}, gives
\begin{align*}
  \sum_{i \in I}
    |c_{i} | \, | \langle g_{i}, f \rangle_{\Hpi} |
  & \leq \| V_g f \|_{L^1_w}
         \bigg\|
           \sum_{i \in I}
             |c_{i}|  |V_g g_{i}|
         \bigg\|_{L^{\infty}_{1/w}}
    \lesssim \| f \|_{\Hw}
             \bigg\| \sum_{i \in I} |c_{i}|  |V_g g_{i} | \bigg \|_{\wienerL{Y}} \\
  & \lesssim \| c \|_{Y_d (\Lambda)} \| f \|_{\Hw}.
\end{align*}
The absolute convergence of $\sum_{i \in I} c_{i} \langle g_{i}, f \rangle$
for $f \in \Hpi$ implies that $\synthesis c = \sum_{i \in I} c_{i} g_{i} \in \Reservoir_w$
is well-defined and that the series defining $\synthesis c$ converges unconditionally
in the weak-$\ast$-topology on $\Reservoir_w$.
The identity $V_g (\synthesis c) = \sum_{i \in I} c_{i} V_g g_{i} = \synthesis_0 c$
and the estimate \eqref{eq:D0_bounded} yield that
\(
  \| \synthesis c \|_{\Co_g (Y)}
  = \| V_g (\synthesis c) \|_{\wienerL{Y}}
  = \| \synthesis_0 c \|_{\wienerL{Y}}
  \lesssim \| \Phi \|_{\wienerSt{L_w^p}} \| c \|_{Y_d (\Lambda)},
\)
which completes the proof.
\end{proof}

The method of proof used for \Cref{prop:molecule_bounded}
resembles the one of \cite[Lemma~A.3]{MoleculePaper},
with several modifications to apply the convolution relation
\ref{enu:CompatibleConvolution} of \Cref{def:compatible}.

The next result will be helpful for extending certain identities from $\Hpi$
to $\Reservoir_w$ by density.

\begin{lemma}\label{lem:wstar_continuity}
  Let $(g_{i})_{i \in I} \subseteq \Hpi$ be a system of $L_w^p$-molecules
  indexed by the relatively separated family $\Lambda = (\lambda_i)_{i \in I}$.
  Then the associated reconstruction and coefficient operators
  $\synthesis : \ell^{\infty}_{1/w} (I) \to \Reservoir_w(g)$
  and $\analysis : \Reservoir_w(g) \to \ell^{\infty}_{1/w} (I)$
  are well-defined, bounded, and weak-$*$-continuous.
\end{lemma}

\begin{proof}
Throughout the proof, it will be used that $L_w^1$ and $L_{1/w}^\infty$ are
$L_w^p$-compatible and that $\Co(L^1_w) = \Hw$ and $\Co(L^{\infty}_{1/w}) = \Reservoir_w$
(cf.\ \Cref{lem:coincidence}) and that $(L^1_w)_d (\Lambda) = \ell^1_w (I)$
and $(L^{\infty}_{1/w})_d (\Lambda) = \ell^{\infty}_{1/w} (I)$,
with the interpretation $w(i) = w(\lambda_i)$ for $i \in I$ and similarly for $1/w$.
(cf.\ Section~\ref{sec:discrete}).

An application of \Cref{prop:molecule_bounded} with $Y = L^{\infty}_{1/w}$
(resp.\ $Y = L^{1}_w$) yields that the synthesis operator
$\synthesis : \ell^{\infty}_{1/w} (I) \to \Reservoir_w$ and the coefficient operator
$\analysis_0 : \Hw \to \ell_w^1(I), f \mapsto (\langle f, g_i \rangle)_{i \in I}$
are well-defined and bounded.
Therefore, if $f \in \Hw$ and $c \in \ell^{\infty}_{1/w} (I)$, then
\[
  \langle \synthesis c, f \rangle_{\Reservoir_w, \Hw}
  = \sum_{i \in I}
      c_i \langle g_{i}, f \rangle_{\Hpi}
  = \sum_{i \in I}
      c_i \, \overline{\langle f, g_{i} \rangle_{\Reservoir_w, \Hw}}
  = \langle c, \analysis_0 f \rangle.
\]
This easily implies that $\synthesis : \ell^{\infty}_{1/w} \to \Reservoir_w$
is continuous if the domain and co-domain are equipped with the weak-$*$-topology.

Similarly, \Cref{prop:molecule_bounded} implies that
$\analysis : \Reservoir_w \to \ell_{1/w}^\infty (I)$ and the reconstruction operator
$\synthesis_0 : \ell^1_w (I) \to \Hw, (c_i)_{i \in I} \mapsto \sum_{i \in I} c_i \, g_i$
are well-defined and bounded.
Since $c = \sum_{i \in I} c_i \delta_i$ with unconditional convergence in $\ell_w^1$
for $c = (c_i)_{i \in I} \in \ell_w^1$,
this implies that $\synthesis_0 c = \sum_{i \in I} c_i g_i$ converges unconditionally in $\Hw$.
Therefore, if $c \in \ell^1_w (I)$ and $f \in \Reservoir_w$, then
\[
  \langle f, \synthesis_0 c \rangle
  = \sum_{i \in I}
      \langle f, g_i \rangle \, \overline{c_i}
  = \langle \analysis f, c \rangle.
\]
This easily implies that the map
$\analysis : \Reservoir_w \to \ell^{\infty}_{1/w} (I)$ is weak-$*$-continuous.
\end{proof}

As a consequence, we can now show that certain reproducing formulas involving families of molecules
that are valid on $\Hpi$ (or even only on $\Hw$) extend to $\Reservoir_w$
(and hence to all coorbit spaces $\Co (Y)$).

\begin{corollary}\label{cor:DensityExtension}
  Let $(g_i)_{i \in I}, (h_i)_{i \in I} \subseteq \Hpi$ be two systems of
  $L_w^p$-molecules index by the relatively separated family $\Lambda = (\lambda_i)_{i \in I}$.
  Then the following hold:
  \begin{enumerate}[label=(\roman*)]
    \item If $f = \sum_{i \in I} \langle f,h_i \rangle g_i$ for all $f \in \Hw$,
          then the same holds for all $f \in \Reservoir_w$.

    \item If $c_i = \langle \sum_{\ell \in I} c_\ell \, g_\ell, h_i \rangle_{\Hpi}$ holds for all
          finitely supported sequences $c \in c_{00}(I)$ and all $i \in I$, then
          \[
            c_i
            = \bigg\langle
                \sum_{\ell \in I}
                  c_\ell \, g_\ell,
                h_i
              \bigg\rangle_{\Reservoir_w, \Hw}
          \]
          holds for all $c \in \ell_{1/w}^\infty (I)$ and all $i \in I$.
  \end{enumerate}
\end{corollary}

\begin{proof}
(i)
\Cref{lem:basic_Hw} shows that $\pi(G) g \subseteq \Hw$ is complete in $\Hw$;
therefore, it follows that $\Hw \hookrightarrow \Hpi \hookrightarrow \Reservoir_w$
separates the points of $\Hw$; i.e., if $g \in \Hw$ satisfies
$f (g) = \langle f, g \rangle_{\Hpi} = 0$ for all $f \in \Hw$, then $g = 0$.
In particular, this implies that $\Hw$ is weak-$*$-dense in $\Reservoir_w$;
see, e.g., \cite[Corollary~5.108]{AliprantisBorderHitchhiker}.

Let $\analysis_g$ (resp.\ $\analysis_h$) and $\synthesis_g$ (resp.\ $\synthesis_h$)
denote the coefficient and reconstruction operators
associated to $(g_i)_{i \in I}$ (resp.\ $(h_i)_{i \in I}$).
By assumption, it holds $f = \synthesis_g \analysis_h f$ for all $f \in \Hw$.
Since these operators are bounded between $\Reservoir_w$ and $\ell_{1/w}^\infty (I)$
and continuous with respect to the weak-$\ast$-topologies on $\Reservoir_w$
and $\ell_{1/w}^\infty$ (cf.\ \Cref{lem:wstar_continuity}),
it follows that $f = \synthesis_g \analysis_h f$ for all $f \in \Reservoir_w$.

\medskip{}

(ii)
Let $\analysis_g$ (resp.\ $\analysis_h$) and $\synthesis_g$ (resp.\ $\synthesis_h$)
denote the coefficient and reconstruction operators
associated to $(g_i)_{i \in I}$ (resp.\ $(h_i)_{i \in I}$).
By \Cref{lem:wstar_continuity}, these operators are bounded between $\Reservoir_w$
and $\ell_{1/w}^\infty (I)$ and continuous with respect to the weak-$\ast$-topologies
on $\Reservoir_w$ and $\ell_{1/w}^\infty$.
By assumption, we have $c = \analysis_h \synthesis_g$
for all $c \in c_{00}(I)$.
Since $c_{00}(I)$ is weak-$\ast$-dense in $\ell_{1/w}^\infty (I)$,
this implies the claim.
\end{proof}

The following result establishes an atomic decomposition of the coorbit space
$\Co(Y)$ in terms of two families of molecules.

\begin{theorem}\label{thm:frame_extension}
Let $w$ be a $p$-weight for some $p \in (0,1]$ and let $Y$ be a solid,
translation-invariant quasi-Banach function space
such that $Y$ is $L^p_w$-compatible.
Suppose $g \in \Bwp$ is admissible.

There exists a compact unit neighborhood $U \subseteq G$
such that for every relatively separated, $U$-dense family $\Lambda = (\lambda_i)_{i \in I}$ in $G$,
the following properties hold:
\begin{enumerate}[label=(\roman*)]
  \item There exists a family $(h_i)_{i \in I}$
        of $L^p_w$-localized molecules in $\Hpi$
        indexed by $\Lambda$
        such that any $f \in \Co_g(Y)$ can be represented as
        \[
          f
          = \sum_{i \in I}
             \langle f, \pi(\lambda_i) g \rangle h_i
          = \sum_{i \in I}
              \langle f, h_i \rangle \pi(\lambda_i) g
        \]
        with unconditional convergence of the series
        in the weak-$*$-topology on $\Reservoir_w(g)$.

  \item There exists a family $(g_i)_{i \in I}$ of $L^p_w$-localized molecules in $\Hpi$
        such that any vector $f \in \Co_g(Y)$ can be represented as
        \[
          f = \sum_{i \in I} \langle f, g_i \rangle g_i
        \]
        with unconditional convergence of the series
        in the weak-$*$-topology on $\Reservoir_w(g)$.
\end{enumerate}
\end{theorem}

\begin{proof}
By \Cref{thm:dual_frame_RKHS} and since $V_g : \Hpi \to \CalK_g $ is unitary
and $K_\lambda = V_g [\pi(\lambda) g]$, there exists a compact unit neighborhood $U \subseteq G$
such that for every relatively separated, $U$-dense family $\Lambda$ in $G$,
there exists a system of $(L^p_w,g)$-localized molecules  $(h_i)_{i \in I}$
(namely, $h_i = V_g^{-1} [H_i]$ with $(H_i)_{i \in I}$ as in \Cref{thm:dual_frame_RKHS})
that forms a dual frame of $\bigl(\pi(\lambda_i) g \bigr)_{i \in I}$.
The claim then follows by recalling that $\Co_g (Y) \subseteq \Reservoir_w$
and by combining Part (i) of \Cref{cor:DensityExtension}
with \Cref{prop:molecule_bounded}.
This shows (i).

The proof of (ii) is similar, using Part (ii) of \Cref{thm:dual_frame_RKHS}
and Part (ii) of \Cref{cor:DensityExtension}.
\end{proof}

Dual to \Cref{thm:frame_extension}, we show the existence of dual Riesz sequences of molecules for coorbit spaces.

\begin{theorem} \label{cor:riesz_extension}
Let $w$ be a $p$-weight for some $p \in (0,1]$ and let $Y$ be a solid, translation-invariant
quasi-Banach function space such that $Y$ is $L^p_w$-compatible.
Suppose $g \in \Bwp$ is admissible.

There exists a compact unit neighborhood $U \subseteq G$
such that for every $U$-separated family $\Lambda = (\lambda_i)_{i \in I}$ in $G$, there exists
a family $(h_i)_{i \in I} \subseteq \overline{\Span} \{ \pi(\lambda_i) g : i \in I \} \subseteq \Hpi$
which forms a system of $L^p_w$-localized molecules
and such that the moment problem
\[
  \langle f, \pi(\lambda_i) g \rangle = c_i, \quad i \in I,
\]
admits the solution $f = \sum_{i \in I} c_i h_i \in \Co_g(Y)$
for any given sequence $(c_i)_{i \in I} \in Y_d (\Lambda)$.
\end{theorem}

\begin{proof}
  By \Cref{thm:dual_riesz_RKHS} and since $V_g : \Hpi \to \CalK_g$ is unitary
  and $K_\lambda = V_g [\pi(\lambda) g]$, there exists a compact unit neighborhood $U \subseteq G$
  such that for any $U$-separated family $\Lambda = (\lambda_i)_{i \in I}$ in $G$,
  the family $\bigl(\pi(\lambda_i) g\bigr)_{i \in I}$ is a Riesz sequence in $\Hpi$
  with unique biorthogonal sequence $(h_i)_{i \in I}$
  in $\overline{\Span} \{ \pi(\lambda_i) g : i \in I \}$ forming a family of $(L_w^p, g)$ molecules
  in $\Hpi$ indexed by $\Lambda$.
  Let $\analysis_g$ (resp.\ $\analysis_h$) and $\synthesis_g$ (resp.\ $\synthesis_h$)
  denote the coefficient and reconstruction operators associated to $(\pi(\lambda_i) g)_{i \in I}$
  (resp.\ $(h_i)_{i \in I}$).
  Then $c = \analysis_g \synthesis_h c$ for all $c \in c_{00} (\Lambda)$.
  By Part~(ii) of \Cref{cor:DensityExtension}, the same then holds for all
  $c \in \ell_{1/w}^\infty (I)$.
  Since $Y_d (\Lambda) \hookrightarrow \ell_{1/w}^\infty (I)$ by \Cref{prop:molecule_bounded},
  this implies the claim.
\end{proof}

\chapter{Applications of molecular decompositions}
\label{sec:applications}

This chapter provides three applications of the existence of dual frames and Riesz sequences of molecules.

\section{Boundedness of operators}
A powerful principle for showing that an operator is bounded on distribution or function spaces
is to show that is maps adequate ``atoms'' into ``molecules'', see, e.g.,
\cite{taibleson1980the, torres1991boundedness, frazier1990discrete}
for various classical versions of this principle.
An abstract version of this principle for coorbit spaces was obtained in \cite[Theorem 3.5]{groechenig2009molecules}.
The following theorem provides an extension of this result to quasi-Banach spaces.

\begin{theorem}\label{thm:bounded_extension}
  Let $w$ be a $p$-weight for some $p \in (0,1]$
  and let $Y$ be a solid, translation-invariant quasi-Banach function space
  such that $Y$ is $L^p_w$-compatible.
  For an admissible $g \in \Bwp$, let $\Lambda = (\lambda_i)_{i \in I}$
  be a relatively separated family in $G$ such that
  $\bigl(\pi(\lambda_i) g \bigr)_{i \in I}$ is a frame for $\Hpi$
  with a dual frame $(h_i)_{i \in I}$ forming a family of $L_w^p$-molecules
  in $\Hpi$.

  If $T : \Hw \to \Reservoir_w$ is a bounded linear operator
  such that the vectors
  \[
     m_i
     := T [\pi(\lambda_i) g],
     \quad i \in I,
  \]
  define a system of $L^p_w$-molecules with envelope $\Phi \in \wienerStC{L^p_w}$,
  then $T$ extends to a well-defined bounded linear operator $T : \Co(Y) \to \Co(Y)$ with
  \begin{align} \label{eq:operator_extension}
    \| T \|_{\Co(Y) \to \Co(Y)}
    \lesssim \| \Phi \|_{\wienerSt{L^p_w}},
  \end{align}
  where the implicit constant only depends on $g,Q,w,p,Y,\Lambda$, and $(h_i)_{i \in I}$.
\end{theorem}

\begin{proof}
  Since $\big(\pi(\lambda_i) g \big)_{i \in I}$ and $(h_i)_{i \in I}$
  are dual frames of molecules, it follows by \Cref{cor:DensityExtension}
  and because of $\Co(Y) \subseteq \Reservoir_w$ that any vector $f \in \Co(Y)$ can be represented as
  $f = \sum_{i \in I} \langle f, h_i \rangle \pi(\lambda_i) g$.
  In addition, Proposition~\ref{prop:molecule_bounded} provides the estimate
  \[ \| \analysis_h \|_{Y_d (\Lambda)} = \| (\langle f, h_i \rangle)_{i \in I} \|_{Y_d (\Lambda)} \lesssim \| f \|_{\Co(Y)},\]
  where the implicit constant is independent of $f$
  (in fact, only depending on $Y,Q,\Lambda,(h_i)_{i \in I}$).
  With the family $(m_i)_{i \in I} = \big( T[\pi(\lambda_i) g] \big)_{i \in I}$
  from the statement of the theorem, \Cref{prop:molecule_bounded} also shows
  that the reconstruction operator
  $\synthesis_m : Y_d(\Lambda) \to \Co_g (Y), (c_i)_{i \in I} \mapsto \sum_{i \in I} c_i \, m_i$
  is well-defined and bounded, with unconditional convergence of the series in the
  weak-$\ast$-topology of $\Reservoir_w \supset \Co_g (Y)$
  and with $\| \synthesis_m \|_{Y_d \to \Co(Y)} \lesssim \| \Phi \|_{\wienerSt{L_w^p}}$,
  where the implied constant only depends on $g,Q,w,p,Y$.

  %Given $f \in \Co(Y)$, write $f = \sum_{i \in I} \langle f, h_i \rangle \pi(\lambda_i) g$
  %and define $\widetilde{T} f$ by
  Now, define $\widetilde{T} : \Co_g (Y) \to \Co_g (Y)$ by
  \[
    \widetilde{T} := \synthesis_m \circ \analysis_h,
    \quad \text{that is}, \quad
    \widetilde{T} f
    = \sum_{i \in I} \langle f, h_i \rangle m_i
    = \sum_{i \in I} \langle f, h_i \rangle T [\pi(\lambda_i) g]
    .
  \]
  As a composition of bounded linear operators, $\widetilde{T} : \Co(Y) \to \Co(Y)$
  is itself a well-defined bounded linear operator, with
  \[
    \| \widetilde{T} \|_{\Co(Y) \to \Co(Y)}
    \leq \| \analysis_h \| \cdot \| \synthesis_m \|
    \lesssim \| \Phi \|_{\wienerSt{L_w^p}}
    ,
  \]
  where the implied constant only depends on $g,Q,w,p,Y,\Lambda$, and $(h_i)_{i \in I}$.
  This shows \eqref{eq:operator_extension}.

  To prove that $\widetilde{T}$ is an extension of $T$,
  let $\synthesis_g$ denote the reconstruction operator associated to the family
  $\bigl(\pi(\lambda_i) g\bigr)_{i \in I}$.
  Since $\Hw = \Co_g (L_w^1)$ by \Cref{lem:coincidence}
  and since $(L_w^1)_d(\Lambda) = \ell_w^1 (I)$,
  \Cref{prop:molecule_bounded} shows that $\synthesis_g : \ell_{w}^1 (I) \to \Hw$
  is well-defined and bounded.
  Since we have $c = \sum_{i \in I} c_i \, \delta_i$ with unconditional convergence
  in $\ell_w^1 (I)$ for all $c \in \ell_w^1(I)$, it follows that the series defining
  $\synthesis_g c = \sum_{i \in I} c_i \, \pi(\lambda_i) g$ converges unconditionally
  in $\Hw$ for any $c \in \ell_w^1 (I)$.
  Since \Cref{prop:molecule_bounded} also implies that
  $(\langle f, h_i \rangle)_{i \in I} \in \ell_w^1(I)$ for any $f \in \Hw = \Co_g (L_w^1)$,
  and since $(h_i)_{i \in I}$ is a dual frame for $\big( \pi(\lambda_i) g \big)_{i \in I}$,
  we thus see $f = \sum_{i \in I} \langle f, h_i \rangle \pi(\lambda_i) g$
  for all $f \in \Hw \subseteq \Hpi$, where the series converges unconditionally in $\Hw$.
  Because $T : \Hw \to \Reservoir_w$ is a bounded linear operator, this implies
  \[
    T f
    = T \bigg( \sum_{i \in I} \langle f, h_i \rangle \pi(\lambda_i) g \bigg)
    = \sum_{i \in I} \langle f, h_i \rangle T [\pi(\lambda_i) g]
    = \synthesis_m \analysis_h f
    = \widetilde{T} f
  \]
  for all $f \in \Hw$.
  Thus, $\widetilde{T} : \Co_g (Y) \to \Co_g(Y)$ is indeed a bounded extension of $T$.
\end{proof}

The proof of \Cref{thm:bounded_extension} follows the arguments in \cite{groechenig2009molecules} closely.

\begin{remark} With notation as in \Cref{thm:bounded_extension}, the following hold:
 \begin{enumerate}[label=(\alph*)]
  \item The existence of a frame $(\pi(\lambda_i) g )_{i \in I}$ for $\Hpi$
        with a dual frame $(h_i)_{i \in I}$ forming a family of $(L^p_w, g)$-molecules
        is guaranteed by \Cref{thm:dual_frame_RKHS},
        for $\Lambda$ ``sufficiently dense''.

  \item The assumption that $T : \Hw \to \Reservoir_w$ is bounded is satisfied,
        in particular, whenever $T : \Hpi \to \Hpi$ is bounded,
        since $\Hw \hookrightarrow \Hpi \hookrightarrow \Reservoir_w$
        by Section~\ref{sec:integrable}.
 \end{enumerate}

\end{remark}

\section{Embeddings}

In this section, we show that the embedding property $\Co_g(Y) \hookrightarrow \Co_g(Z)$
can be decided at the level of sequence spaces; namely, this embedding is valid if and only if
$Y_d(\Lambda) \hookrightarrow Z_d(\Lambda)$.
For Banach spaces, this result is already known;
see \cite[Theorem~8.4]{feichtinger1989banach2}.

We start by showing that the boundedness of the embedding $Y_d(\Lambda) \hookrightarrow Z_d(\Lambda)$
does not depend on the choice of the relatively separated and relatively dense family $\Lambda$.

\begin{lemma}\label{lem:SequenceSpaceEmbeddingIndependence}
  Let $Y,Z$ be translation-invariant, solid quasi-Banach function spaces on $G$.
  Let $\Lambda = (\lambda_i)_{i \in I}$ be relatively separated and relatively dense in $G$. Assume that
  $Y_d(\Lambda) \hookrightarrow Z_d(\Lambda)$.
  Then, if $\Gamma = (\gamma_j)_{j \in J} \subseteq G$ is any relatively separated family,
  the embedding $Y_d (\Gamma) \hookrightarrow Z_d (\Gamma)$ holds.
\end{lemma}

\begin{proof}
  Since $\Lambda$ is relatively dense, there exists a relatively compact
  unit neighborhood $U \subseteq G$ satisfying $\sum_{i \in I} \indicator_{\lambda_i U} (x) \geq 1$
  for all $x \in G$.
  Let $c = (c_j)_{j \in J} \in Y_d(\Gamma)$ be arbitrary and note for
  $e_i := \sum_{j \in J} |c_j| \indicator_{\gamma_j Q \cap \lambda_i U \neq \emptyset}$ that
  \begin{equation}
    0
    \leq \sum_{j \in J}
           |c_j| \, \indicator_{\gamma_j Q} (x)
    \leq \sum_{i \in I}
         \bigg[
           \indicator_{\lambda_i U} (x)
           \sum_{j \in J}
             |c_j| \,
             \indicator_{\gamma_j Q} (x) \,
             \indicator_{\lambda_i U} (x)
         \bigg]
    \leq \sum_{i \in I}
           e_i \, \indicator_{\lambda_i U} (x)
    .
    \label{eq:SequenceSpaceEmbeddingIndependenceProof1}
  \end{equation}

  Let $V := Q U^{-1} Q$; we claim for arbitrary $x \in G$, $i \in I$, and $j \in J$ that
  \[
    \indicator_{\lambda_i Q} (x)
    \indicator_{\gamma_j Q \cap \lambda_i U \neq \emptyset}
    \leq \indicator_{\lambda_i Q} (x)
         \indicator_{\gamma_j V} (x)
    .
    %\label{eq:SequenceSpaceEmbeddingIndependenceProof1}
  \]
  Indeed, this is trivial if the left-hand side vanishes.
  If not, then $x = \lambda_i q$ for some $q \in Q$ and there exists some
  $y = \lambda_i u \in \gamma_j Q \cap \lambda_i U$.
  This implies
  \[
    x
    = \lambda_i q
    = \lambda_i u u^{-1} q
    = y u^{-1} q
    \in \gamma_j Q U^{-1} Q
    = \gamma_j V
    ,
  \]
  from which the claimed estimate follows easily.
  Combining this estimate with the definition of $e_i$, it follows that
  \begin{align*}
    0
    \leq \sum_{i \in I}
           e_i \indicator_{\lambda_i Q} (x)
    & =    \sum_{j \in J}
           \bigg[
             |c_j|
             \sum_{i \in I}
               \indicator_{\lambda_i Q} (x)
               \indicator_{\gamma_j Q \cap \lambda_i U \neq \emptyset}
           \bigg] \\
    & \leq \sup_{x \in G}
           \bigg[
             \sum_{i \in I}
               \indicator_{\lambda_i Q} (x)
           \bigg]
            \sum_{j \in J}
                   |c_j|
                   \indicator_{\gamma_j V} (x)
      = \rel(\Lambda)
         \sum_{j \in J}
                |c_j|
                \indicator_{\gamma_j V} (x)
      ;
  \end{align*}
  for arbitrary $x \in G$.

  Lastly, let us write $\| \cdot \|_{Y_d (\Lambda, U)}$ to denote the (quasi)-norm on the
  discrete sequence space $Y_d(\Lambda)$, but using the set $U$ instead of $Q$
  in \Cref{eq:SequenceSpaceNormDefinition}; similar notation will also be used for $Z_d$.
  As noted in \Cref{sub:SequenceSpaces}, since $Y,Z$ are translation-invariant,
  different choices of relatively compact unit neighborhoods for the set $Q$
  yield equivalent (quasi)-norms.
  By combining \Cref{eq:SequenceSpaceEmbeddingIndependenceProof1} with the solidity of $Z$,
  then using the embedding $Y_d (\Lambda) \hookrightarrow Z_d (\Lambda)$
  and finally combining the previous estimate with the solidity of $Y$, we thus see
  \begin{align*}
    \| c \|_{Z_d (\Gamma)}
    & = \bigg\|
          \sum_{j \in J}
            |c_j| \, \indicator_{\gamma_j Q}
        \bigg\|_Z
      \leq \bigg\|
             \sum_{i \in I}
               e_i \indicator_{\lambda_i U}
           \bigg\|_Z \\
    & = \| e \|_{Z_d (\Lambda, U)}
      \lesssim \| e \|_{Z_d (\Lambda)}
      \lesssim \| e \|_{Y_d (\Lambda)}
      = \bigg\|
          \sum_{i \in I}
            e_i \, \indicator_{\lambda_i Q}
        \bigg\|_Y \\
    & \leq \rel(\Lambda)
           \bigg\|
             \sum_{j \in J}
               |c_j| \,
               \indicator_{\gamma_j V}
           \bigg\|_Y
      =    \rel(\Lambda) \, \| c \|_{Y_d(\Gamma, V)}
      \lesssim \| c \|_{Y_d(\Gamma)}
      .
  \end{align*}
  This shows that $Y_d (\Gamma) \hookrightarrow Z_d (\Gamma)$ and thus completes the proof.
\end{proof}

Using the previous lemma, we can now state and prove the main result of this subsection.

\begin{theorem}\label{thm:CoorbitEmbeddingViaSequenceSpaces}
  Let $w : G \to [1,\infty)$ be a $p$-weight for some $p \in (0,1]$
  and let $g \in \Bwp$ be admissible.
  Let $Y,Z$ be solid, translation-invariant quasi-Banach spaces on $G$
  that are both $L_w^p$-compatible.

  Then the embedding $\Co_g (Y) \hookrightarrow \Co_g (Z)$ holds if and only if
  $Y_d (\Lambda) \hookrightarrow Z_d(\Lambda)$ holds for some (every)
  relatively separated and relatively dense family $\Lambda = (\lambda_i)_{i \in I} \subseteq G$.
\end{theorem}

\begin{proof}
  First, suppose that the identity map $\iota : \Co_g (Y) \to \Co_g (Z), f \mapsto f$
  is well-defined and bounded.
  Choose a compact unit neighborhood $U \subseteq G$ as provided by \Cref{cor:riesz_extension}.
  Let $\Lambda = (\lambda_i)_{i \in I}$ be a maximal $U$-separated family in $G$.
  For every $x \in G$ we then either have $x = \lambda_i$ for some $i \in I$
  (and hence $x \in \bigcup_{i \in I} \lambda_i U U^{-1}$)
  or $x \notin \{ \lambda_i \colon i \in I \}$ and then $x U \cap \lambda_i U \neq \emptyset$
  for some $i \in I$, by maximality; hence, $x \in \lambda_i U U^{-1}$.
  Overall, this shows that $G = \bigcup_{i \in I} \lambda_i U U^{-1}$,
  which easily implies that $\Lambda$ is relatively separated and relatively dense in $G$.

  \Cref{cor:riesz_extension} yields a family $(h_i)_{i \in I} \subseteq \Hpi$
  that forms a system of $L_w^p$-localized molecules and such that
  $c = \analysis_g \synthesis_h c$ for every $c \in Y_d(\Lambda)$,
  where $\analysis_g$ and $\synthesis_h$ are the coefficient and reconstruction operators associated to
  $\bigl(\pi(\lambda_i) g\bigr)_{i \in I}$ and $(h_i)_{i \in I}$, respectively.
  By \Cref{prop:molecule_bounded}, we know that $\synthesis_h : Y_d (\Lambda) \to \Co_g(Y)$
  and $\analysis_g : \Co_g (Z) \to Z_d(\Lambda)$ are bounded.
  Hence, for every $c \in Y_d (\Lambda)$,
  \[
    \| c \|_{Z_d (\Lambda)}
    = \| \analysis_g \iota \synthesis_h c \|_{Z_d(\Lambda)}
    \leq \| \analysis_g \|_{\Co_g (Z) \to Z_d}
          \| \iota \|_{\Co_g (Y) \to \Co_g (Z)}
          \| \synthesis_h \|_{Y_d(\Lambda) \to \Co_g (Y)}
          \| c \|_{Y_d (\Lambda)}
    \!<\! \infty .
  \]
  This proves that $Y_d (\Lambda) \hookrightarrow Z_d (\Lambda)$.
  Since $\Lambda$ is relatively separated and relatively dense in $G$, \Cref{lem:SequenceSpaceEmbeddingIndependence}
  shows that in fact $Y_d (\Gamma) \hookrightarrow Z_d(\Gamma)$
  for every relatively separated family $\Gamma$ in $G$.

  \medskip{}

  Conversely, suppose that the identity map $\iota : Y_d (\Lambda) \to Z_d (\Lambda)$ is well-defined and bounded
  for some (hence, every)
  relatively separated and relatively dense family $\Lambda$ in $G$ by \Cref{lem:SequenceSpaceEmbeddingIndependence}.
  \Cref{thm:frame_extension} shows that there exists a relatively separated and relatively dense family
  $\Lambda = (\lambda_i)_{i \in I}$ in $G$ and a family $(g_i)_{i \in I}$
  of $L_w^p$-localized molecules in $\Hpi$ indexed by $\Lambda$ such that
  $f = \synthesis_g \analysis_g f$ for all $f \in \Co_g (Y)$.
  By \Cref{prop:molecule_bounded}, it follows that that $\analysis_g : \Co_g (Y) \to Y_d (\Lambda)$
  and $\synthesis_g : Z_d (\Lambda) \to \Co_g (Z)$ are well-defined and bounded.
  Hence, we see for any $f \in \Co_g (Y)$ that $f = \synthesis_g \iota \analysis_g f \in \Co_g (Z)$
  with
  \[
    \| f \|_{\Co_g (Z)}
    = \| \synthesis_g \iota \analysis_g f \|_{\Co_g (Z)}
    \leq \| \synthesis_g \|_{Z_d (\Lambda) \to \Co_g (Z)}
         \| \iota \|_{Y_d(\Lambda) \to Z_d(\Lambda)}
         \| \analysis_g \|_{\Co_g(y) \to Y_d(\Lambda)}
    .
  \]
  This shows that $\Co_g (Y) \hookrightarrow \Co_g (Z)$.
\end{proof}

\section{Local behavior of matrix coefficients}

\Cref{prop:selfimproving} provides a sufficient condition on a quasi-Banach function space $Y$
under which the associated coorbit space $\Co(Y)$ can be characterized as
\begin{align} \label{eq:local_wavelet}
 \Co(Y) := \{ f \in \Reservoir_w : V_g f \in \wienerL{Y} \} = \{ f \in \Reservoir_w : V_g f \in Y \},
\end{align}
i.e., when $V_g f \in Y$ automatically implies $V_g f \in \wienerL{Y}$.
In particular, this is the case for genuine Banach spaces, cf.\ \cite[Theorem 8.1]{feichtinger1989banach2}.
In addition, the characterization \eqref{eq:local_wavelet} holds for certain settings
in which there exists a specific vector $g_0$ such that $\|V_{g_0} f \|_{\wienerL{Y}} \lesssim \| V_{g_0} f \|_Y$,
see, e.g., \cite[Theorem 3.3]{galperin2004time}.
The following proposition provides an abstract result to extend the characterization
\eqref{eq:local_wavelet} from one specific vector to more general analyzing vectors.

\begin{proposition}\label{thm:local_wavelet}
%Let $Y$ be a solid, translation-invariant quasi-Banach function space with $p$-norm $\| \cdot \|_Y$,
%let $w$ be a control weight for $Y$ and let $v$ be a control weight for $L^p_w$.

Let $w$ be a $p$-weight for some $p \in (0,1]$ and let $Y$ be a solid,
translation-invariant quasi-Banach function space such that $Y$ is $L^p_w$-compatible.
Moreover, assume that $w$ is a control weight for $Y$, i.e.,
$\| \translationR{x} \|_{Y \to Y} \leq w(x)$ for all $x \in \group$.

Suppose that $g \in \Bw^p$ is an  admissible vector and that there exists $C>0$ such that
\[
  \| V_g f \|_{\wienerL{Y}} \leq C \| V_g f \|_{Y}
\]
for all $f \in \Reservoir_w(g)$.
Then, for any admissible vector $h \in \Bw^p$ satisfying $V_h g \in \wienerSt{L^p_w}$,
it holds that
\[
  \| V_h f \|_{\wienerL{Y}} \asymp \| V_h f \|_{Y}
\]
for all $f \in \Reservoir_w(h)$.
Moreover,
\[
  \Co_g (Y)
  = \{ f \in \Reservoir_w (g) : V_g f \in Y \}
  = \{ f \in \Reservoir_w(h) : V_h f \in Y \}
  = \Co_h (Y).
\]
\end{proposition}

\begin{proof}
By assumption, the vectors $g, h \in \Bw^p$ are admissible and $V_h g \in \wienerSt{L^p_w}$.
By \Cref{prop:coorbit_basic}, it follows that $\Hw(g) = \Hw(h)$ and
$\Reservoir_w := \Reservoir_w (g) = \Reservoir_w (h)$, and furthermore
\[
  \| V_h f \|_{\wienerL{Y}} \asymp \| V_g f \|_{\wienerL{Y}}
\]
for all $f \in  \Reservoir_w$.
Since $\wienerL{Y} \hookrightarrow Y$ by definition
and since $\| V_g f \|_{\wienerL{Y}} \lesssim \| V_g f \|_Y$ by assumption,
it remains to show that $\| V_g f \|_Y \lesssim \|V_h f \|_Y$ for all $f \in \Reservoir_w$.

The condition $V_h g \in \wienerSt{L_w^p}$ implies that $g \in  \Co_h(L^p_w)$.
Since moreover $L_w^p$ is $L_w^p$-compatible by \Cref{cor:Lpw_nostrong} and since
$h \in \mathcal{B}_w^p$ is admissible, \Cref{thm:frame_extension} (applied to $L_w^p$)
shows that there exists a family $\Lambda = (\lambda_i)_{i \in I}$ in $G$
and a sequence $(c_i)_{i \in I} \in (L^p_w)_d(\Lambda) = \ell^p_v(I)$
(with $v_i = w(\lambda_i)$) such that
\[
  g = \sum_{i \in I} c_i \pi(\lambda_i) h,
\]
where the series converges in $\Co_h (L^p_w)$. This holds since $\| \cdot \|_{\Co_h (L^p_w)}$ is a $p$-norm by \Cref{prop:coorbit_basic} and
\(
  \| \pi(\lambda_i) h \|_{\Co_h (L_w^p)}
  = \| \twisttranslationL{\lambda_i} [V_h h] \|_{\wienerL{L_w^p}}
  \leq w(\lambda_i) \| V_h h \|_{\wienerL{L_w^p}}
  = v_i \, \| V_h h \|_{\wienerL{L_w^p}}
\),
by \Cref{lem:basic_Rw} and \Cref{max_commute}.
Combining \Cref{lem:AmalgamWeightedLInftyEmbedding} and \Cref{lem:coincidence}, it follows that $g = \sum_{i \in I} c_i \pi(\lambda_i) h$ also converges in $\Co_h (L^1_w) = \Hw(h) = \Hw(g)$.
Using this, a direct calculation gives
\[
  V_g f
  = \bigg\langle f, \pi(\cdot) \sum_{i \in I} c_i \pi(\lambda_i) h \bigg \rangle
  = \sum_{i \in I} \overline{c_i}  R_{\lambda_i}^{\sigma} [V_h f],
\]
where $R_{\lambda_i}^{\sigma}$ denotes the twisted right-translation operator (see also \Cref{lem:basic_cocycle}).
Therefore,
\[
  \| V_g f \|_Y^p
  = \bigg\| \sum_{i \in I} \overline{c_i}  R_{\lambda_i}^{\sigma} [V_h f] \bigg\|_Y^p
  \leq \sum_{i \in I} |c_i|^p \|R_{\lambda_i}^{\sigma} \|_{Y \to Y}^p \| V_h f \|_Y^p
  \lesssim \| (c_i)_{i \in I} \|^p_{\ell^p_v} \| V_h f \|^p_Y
  .
\]
Here, the last step used the assumption $\| \translationR{x} \|_{Y \to Y} \leq w(x)$ for all $x \in \group$.
This completes the proof.
\end{proof}

\Cref{thm:local_wavelet} is an adaptation of \cite[Theorem 6.2]{rauhut2007coorbit}
to the setting of the present paper.

\chapter{Main results for irreducible, square-integrable representations}
\label{sec:discrete_series}

This chapter provides an overview of the key results
of this paper for the case of irreducible, square-integrable representations.
Several of the results obtained in the main text
have simplified statements for such representations,
which recover known results for Banach spaces.
We include these statements here to allow for an easy reference.

In what follows, we will always assume that the following assumptions are met.

\begin{assumption}\label{assu:IrreducibleAssumption}
  Assume that the following hold:
  \begin{enumerate}[label=(a\arabic*)]
    \item $G$ is a second countable, locally compact group.

    \item $\pi : G \to \CalU(\Hpi)$ is an irreducible unitary representation
          of $G$ on a separable Hilbert space $\Hpi \neq \{ 0 \}$.

    \item  $w : G \to [1,\infty)$ is a $p$-weight for $p \in (0,1]$ (see \Cref{def:PWeight}).

    \item The Wiener amalgam spaces $W^L, W^R, W$ (see \Cref{sec:AmalgamPrelims})
          are defined relatively to a fixed open, symmetric, relatively compact
          unit neighborhood $Q \subseteq G$.

    \item With the coefficient transform $V_g f$ as defined in \Cref{sec:admissible},
          it is assumed that $\Bwp \neq \{ 0 \}$, where
          \[
            \Bwp
            := \big\{ g \in \Hpi \colon V_g g \in \wienerSt{L_w^p} \big\}
            ,
          \]
          and we fix some $g \in \Bwp \setminus \{ 0 \}$.

    \item The space $Y \subseteq L^0(G)$ is assumed to be a solid, translation-invariant quasi-Banach
          function space with $p$-norm $\| \cdot \|_Y$. In addition, $Y$ is assumed to be $L_w^p$-compatible
          (see \Cref{def:compatible}).
  \end{enumerate}
\end{assumption}

\section{Admissibility}%
\label{sub:IrreducibleAdmissibility}

Since $w$ is a $p$-weight (in particular, $w \geq 1$),
it follows from an application of \Cref{lem:AmalgamWeightedLInftyEmbedding} that
\(
  \wienerSt{L^p_w}
  \hookrightarrow \wienerL{L_w^p}
  \hookrightarrow L^1_w \cap L_w^{\infty}
  \hookrightarrow L^2
  ,
\)
so that every $h \in \Bwp$ satisfies 
\[ \int_G |\langle h, \pi(x) h \rangle|^2 \, d \mu_G(x) < \infty,\]
i.e., the representation $\pi$ is \emph{square-integrable}.
Then, by \Cref{thm:ortho}, there exists a unique, self-adjoint, positive operator
$C_{\pi} : \dom(C_\pi) \to \Hpi$ such that
\begin{align}\label{eq:ds_admissible}
  \int_G
    | \langle f, \pi(x) h \rangle |^2
  \; d\mu_G (x)
  = \| C_{\pi} h\|_{\Hpi}^2 \| f \|_{\Hpi}^2
\end{align}
for all vectors $f, h \in \Hpi$.
The domain of $C_{\pi}$ is given by $\dom(C_{\pi}) = \{ h \in \Hpi : V_h h \in L^2 (G) \}$.
This implies that $\Bwp \subseteq \dom(C_\pi)$.
Furthermore, \Cref{eq:ds_admissible} implies that any vector in $\dom(C_{\pi}) \setminus \{0\}$
can be normalized to obtain an admissible vector $h$ for $\pi$,
i.e., such that $V_h : \Hpi \to L^2 (G)$ is an isometry.

\section{Coorbit spaces}
With the fixed vector $g \in \Bwp \setminus \{0\}$ from \Cref{assu:IrreducibleAssumption}, define
\[
  \Hw
  := \Hw (g)
  := \bigl\{ f \in \Hpi \,\,\colon\,\, V_g f \in L^1_w (G) \bigr\}
\]
and equip it with the norm $\| f \|_{\Hw} = \| f \|_{\Hw (g)} = \|V_g f \|_{L^1_w}$.
Note that $W(L^p_w) \hookrightarrow L^1_w$ by \Cref{lem:AmalgamWeightedLInftyEmbedding},
so that $\Bwp \subseteq \Hw$; in particular, $g \in \Hw$.

Although the space $\Hw(g)$ is defined relative to the fixed vector $g \in \Bwp \setminus \{0\}$,
the following lemma shows that the space $\Hw$ is actually independent of the choice
of $g \in \Bwp \setminus \{ 0 \}$.
This crucially uses that $\pi$ is an \emph{irreducible} representation; see
also \Cref{rem:reducible_vgh}.

\begin{lemma}\label{lem:reservoir_ds}
Under \Cref{assu:IrreducibleAssumption}, the following hold:
\begin{enumerate}[label=(\roman*)]
    \item If $h \in \Bwp \setminus \{0\}$,
          then $\Hw(g) = \Hw(h)$  with $\| \cdot \|_{\Hw (g)} \asymp \| \cdot \|_{\Hw (h)}$.

    \item The space $\Hw$ is a $\pi$-invariant Banach space
          satisfying $\Hw \hookrightarrow \Hpi$.
\end{enumerate}
\end{lemma}

\begin{proof}
As explained in \Cref{sub:IrreducibleAdmissibility}, we have $\Bwp \!\subset\! \dom(C_{\pi})$,
and the normalized vector $\widetilde{g} = g / \| C_{\pi} g \|_{\Hpi}$ is admissible
for $\pi$.
Since clearly $\Hw (g) = \Hw(\widetilde{g})$, we can assume that $g$ is admissible
and similarly (for proving Part~(i)) also that $h$ is admissible.

(i)
Since $L^1_w \ast L^1_w \hookrightarrow L^1_w$
(see, e.g., \cite[Section 3.7]{ReiterClassicalHarmonicAnalysis}),
it follows by an application of \Cref{lem:cross_irreducible}
that the collection $\mathcal{C}_{L^1_w} = \{ f \in \Hpi : V_f f \in L^1_w\}$ is a vector space
and that $V_{f_2} f_1 \in L^1_w$ for all $f_1, f_2 \in \mathcal{C}_{L^1_w}$.
In particular, this implies that $V_g h, V_h g \in L^1_w (G)$,
so that the conclusion follows from \Cref{lem:basic_Hw}.

Assertion~(ii) follows directly from \Cref{lem:basic_Hw}.
\end{proof}

Let $\Reservoir_w = (\Hw)^*$ be the \emph{anti}-dual space of $\Hw(g)$.
Denote the associated conjugate-linear pairing by
\[
  \langle f, h \rangle
  := f(h),
  \quad f \in \Reservoir_w, \; h \in \Hw.
\]
The action $\pi$ of $G$ on $\Hw$ can be naturally extended
to act on an element $f \in \Reservoir_w$ by
\[
  \pi (x) f : \quad
  \Hw \to \mathbb{C},
  \quad x \mapsto f([\pi(x)]^{\ast} h), \quad h \in \Hw,
\]
for $x \in G$.
The associated matrix coefficients are defined as
$V_h f : G \to \mathbb{C}, \; x \mapsto  \langle f, \pi(x) h \rangle$.

For $g \in \Bwp \setminus \{0\}$, the \emph{coorbit space} of $Y$ is the collection
\[
  \Co (Y)
  = \Co_g (Y)
  = \big\{ f \in \Reservoir_w : V_g f \in \wienerL{Y} \big\},
\]
where $\wienerL{Y}$ is the left-sided Wiener amalgam space of $Y$ (cf.\ \Cref{sec:AmalgamPrelims}).
The space is equipped with the norm $\| f \|_{\Co(Y)} = \|V_g f \|_{\wienerL{Y}}$.

As for the space $\Hw$ (cf. \Cref{lem:reservoir_ds}), also the coorbit spaces are independent of the choice of the
analyzing vector $g \in \Bwp \setminus \{ 0 \}$,
whenever $\pi$ is irreducible.

\begin{proposition} \label{prop:independence_irreducible}
Under \Cref{assu:IrreducibleAssumption}, the following assertions hold:
\begin{enumerate}[label=(\roman*)]
  \item If $h \in \Bwp \setminus \{0\}$, then $\Co_g (Y) = \Co_h (Y)$
        with $\| \cdot \|_{\Co_g (Y)} \asymp \| \cdot \|_{\Co_h (Y)}$.
  \item The space $\Co(Y)$ is a $\pi$-invariant quasi-Banach space
        with $p$-norm $\| \cdot \|_{\Co(Y)}$
        and satisfies $\Co(Y) \hookrightarrow \Reservoir_w$.
\end{enumerate}
\end{proposition}

\begin{proof}
As in the proof of \Cref{lem:reservoir_ds},
the normalized vector $\widetilde{g} := g / \| C_{\pi} g \|_{\Hpi}$ is admissible for $\pi$,
and (for proving Part~(i)) the same holds for the normalized version of $h$.
Furthermore, we clearly have $\Co_{g}(Y) = \Co_{\widetilde{g}}(Y)$
and similarly for $h$.
Hence, we can assume that $g$ and $h$ are admissible.

(i)
By \Cref{cor:UserFriendlyConvolutionBounds},
it holds that $\wienerSt{L^p_w} \ast \wienerSt{L^p_w} \hookrightarrow \wienerSt{L^p_w}$.
Since $\Bwp = \mathcal{C}_{\wienerSt{L^p_w}}$,
where $\mathcal{C}_{\wienerSt{L^p_w}}$ is as defined in \Cref{lem:cross_irreducible},
an application of that result yields that $V_h g \in \wienerSt{L^p_w}$.
The equivalence $\| \cdot \|_{\Co_g (Y)} \asymp \| \cdot \|_{\Co_h (Y)}$
follows therefore from \Cref{prop:coorbit_basic}.

Assertion~(ii) is a direct consequence of \Cref{prop:coorbit_basic}.
\end{proof}

\section{Molecular decompositions}

This subsection provides self-contained statements of the main results on molecular decompositions.

A family $(h_i)_{i \in I}$ of vector $h_i \in \Hpi$,
indexed by a relatively separated $\Lambda = (\lambda_i)_{i \in I} \subseteq G$,
is called a \emph{system of  $L_w^p$-molecules},
if there exists an envelope $\Phi \in \wienerStC{L_w^p}$ satisfying
\begin{equation}
  |V_g h_i (x)| \leq \Phi(\lambda_i^{-1} x)
  \label{eq:IrreducibleMoleculeDefinition}
\end{equation}
for all $i \in I$ and $x \in G$.

The following result summarizes the main properties of a system of molecules.

\begin{lemma}\label{lem:IrreducibleMoleculeProperties}
  Under \Cref{assu:IrreducibleAssumption}, the following hold:
  \begin{enumerate}[label=(\roman*)]
    \item The notion of a system of $L_w^p$-molecules is independent of the choice of analyzing vector
          $g \in \Bwp \setminus \{ 0 \}$, i.e., if $(h_i)_{i \in I} \subseteq \Hpi$
          satisfies condition \eqref{eq:IrreducibleMoleculeDefinition}
          for some $\Phi \in \wienerStC{L_w^p}$ and if $h \in \Bwp \setminus \{ 0 \}$ is arbitrary,
          then there exists $\Theta \in \wienerStC{L_w^p}$ satisfying
          \[
            |V_h h_i (x)| \leq \Theta (\lambda_i^{-1} x)
          \]
          for all $i \in I$ and $x \in G$.

    \item If $(h_i)_{i \in I} \subseteq \Hpi$ is a system of $L_w^p$-molecules
          indexed by the relatively separated family $\Lambda = (\lambda_i)_{i \in I}$,
          then $(h_i)_{i \in I} \subseteq \Hw$ and
          the associated coefficient  operator
          \[
            \quad \,\,\,\,
            \analysis : \,\,
            \Reservoir_w \to \ell_{1/w}^\infty (I), \,\,
            f  \mapsto \bigl(\langle f, h_i \rangle_{\Reservoir_w, \Hw}\bigr)_{i \in I}
            \]
          and reconstruction operator
          \[
            \synthesis : \,\,
            \ell_{1/w}^\infty (I) \to \Reservoir_w, \,\,
            (c_i)_{i \in I} \mapsto \sum_{i \in I} c_i h_i
          \]
          are well-defined and bounded, with unconditional convergence of the series
          in the weak-$\ast$-topology on $\Reservoir_w$.

    \item The following restrictions of the operators $\analysis$ and $\synthesis$ are
          well-defined and bounded:
          \begin{enumerate}[label=(\alph*)]
            \item $\analysis : \Hpi \to \ell^2(I)$,
                  $\analysis : \Hw \to \ell_w^1(I)$,
                  and $\analysis : \Co(Y) \to Y_d (\Lambda)$.

            \item $\synthesis : \ell^2(I) \to \Hpi$,
                  $\synthesis : \ell_w^1 (I) \to \Hw$,
                  and $\synthesis : Y_d (\Lambda) \to \Co(Y)$.
          \end{enumerate}
          Here, the space $Y_d(\Lambda) \hookrightarrow \ell_{1/w}^\infty (I)$
          is as defined in \Cref{sub:SequenceSpaces}.
  \end{enumerate}
\end{lemma}

\begin{proof}
(i)
If $(h_i)_{i \in I}$ is a family of $L_w^p$-molecules with respect to the window $g$,
then the same holds with respect to the normalized (admissible) window
$\widetilde{g} = g / \| C_\pi g \|_{\Hpi}$, and vice versa.
Therefore, we can assume that both $g,h \in \Bwp \setminus \{ 0 \}$ are admissible.
By \Cref{cor:UserFriendlyConvolutionBounds},
it holds that $\wienerSt{L^p_w} \ast \wienerSt{L^p_w} \hookrightarrow \wienerSt{L^p_w}$.
Since $\Bwp = \mathcal{C}_{\wienerSt{L^p_w}}$,
where $\mathcal{C}_{\wienerSt{L^p_w}}$ is as defined in \Cref{lem:cross_irreducible},
an application of that result yields that $V_h g, V_g h \in \wienerSt{L^p_w}$.
Therefore, \Cref{lem:MoleculeConditionIndependence} yields the claimed independence.

\medskip{}

Assertions (ii) and (iii) follow from \Cref{prop:molecule_bounded}. For this,
note that
\Cref{lem:AmalgamWeightedLInftyEmbedding} shows that
$\Phi \in \wienerSt{L_w^p} \hookrightarrow \wienerL{L_w^p} \hookrightarrow L_w^1$.
Because of \Cref{eq:IrreducibleMoleculeDefinition} and the translation-invariance of $L_w^1$,
this easily implies $V_g h_i \in L_w^1$ and hence $h_i \in \Hw$ for all $i \in I$.
 \Cref{lem:coincidence} shows that $L^2(G)$, $L_w^1 (G)$ and $L_{1/w}^\infty$
are all $L_w^p$-compatible and that, in addition, $\Co(L^2) = \Hpi$, $\Co(L_w^1) = \Hw$,
and $\Co(L_{1/w}^\infty) = \Reservoir_w$.
Therefore, \Cref{prop:molecule_bounded} implies all the stated properties.
\end{proof}

The next result shows that the orbit $\pi(G) g$ always contains discrete subsystems
that form a frame for $\Hpi$ and admit a dual system consisting of $L^p_w$-molecules.

\begin{theorem}
Under \Cref{assu:IrreducibleAssumption}, there exists a compact unit neighborhood
$U \subseteq G$
such that for every relatively separated, $U$-dense family $\Lambda = (\lambda_i)_{i \in I}$ in $G$
there exists a family $(h_i)_{i \in I}$ of vectors $h_i \in \Hw$ with the following properties:
\begin{enumerate}[label=(\roman*)]
  \item The system $(h_i)_{i \in I}$ forms a system of $L_w^p$-molecules.

  \item Any $f \in \Co(Y)$ can be represented as
        \[
          f
          = \sum_{i \in I}
              \langle f, h_i \rangle \, \pi(\lambda_i) g
          = \sum_{i \in I}
              \langle f, \pi(\lambda_i) g \rangle \, h_i,
        \]
  with unconditional convergence of the series in the weak-*-topology on $\Reservoir_w$.
\end{enumerate}
\end{theorem}

\begin{proof}
  Both claims follow from \Cref{thm:frame_extension}, strictly speaking after replacing
  $g$ by the admissible vector $\widetilde{g} = g / \| C_\pi g \|_{\Hpi}$.
\end{proof}

A complementing result on dual Riesz sequences of molecules in $\Hpi$ is provided by the following.

\begin{theorem}
Under \Cref{assu:IrreducibleAssumption}, there exists a compact unit neighborhood $U \subseteq G$ such that for every $U$-separated family
$\Lambda = (\lambda_i)_{i \in I}$ in $G$ there exists a family $(h_i)_{i \in I} \subseteq \Hw$
of vectors $h_i \in \overline{\Span \{\pi(\lambda_i) g : i \in I \}} \subseteq \Hpi$
with the following properties:
\begin{enumerate}[label=(\roman*)]
  \item The system $(h_i)_{i \in I}$ forms a system of $L_w^p$-molecules.

  \item For any sequence $(c_i)_{i \in I} \in Y_d (\Lambda)$,
        the vector $f = \sum_{i \in I} c_i h_i$ is an element of $\Co(Y)$
        that solves the moment problem
        \[
          \langle f, \pi(\lambda_i) g \rangle = c_i, \quad i \in I.
        \]
\end{enumerate}
\end{theorem}

\begin{proof}
  Both claims follow from \Cref{cor:riesz_extension}, strictly speaking after replacing
  $g$ by the admissible vector $\widetilde{g} = g / \| C_\pi g \|_{\Hpi}$.
\end{proof}

\appendix

\chapter{Miscellany on quasi-Banach spaces}
\label{sec:QuasiBanachFunctionSpaceAppendix}

\addtocontents{toc}{\protect\setcounter{tocdepth}{0}}

This appendix contains two auxiliary results on quasi-Banach spaces used throughout the main text.
As no reference could be found in the literature, their proofs are provided.

\begin{lemma}\label{lem:QuasiBanachAbsoluteConvergence}
  Let $Y$ be a quasi-normed vector space with $p$-norm $\| \cdot \|$ for some $p \in (0,1]$.
  If $Y$ is complete, then for every countable family $(f_i)_{i \in I}$
  in $Y$ satisfying $\sum_{i \in I} \| f_i \|^p < \infty$,
  the series $\sum_{i \in I} f_i$ is unconditionally convergent in $Y$.

  Conversely, if for every sequence $(f_i)_{i \in \N}$ with $\sum_{i \in \N} \| f_i \|^p < \infty$
  the series $\sum_{i=1}^\infty f_i$ converges in $Y$, then $Y$ is complete.
\end{lemma}

\begin{proof}
  Suppose first that $Y$ is complete.
  It clearly suffices to consider the case $I = \N$.
  Thus, assume that $\sum_{i \in \N} \| f_i \|^p < \infty$.
  Let $F_n := \sum_{i=1}^n f_i$.
  For $n \geq m \geq n_0$, the $p$-norm property gives
  \[
    \| F_n - F_m \|^p
    = \bigg\|
        \sum_{i=m+1}^n f_i
      \bigg\|^p
    \leq \sum_{i=m+1}^n
           \| f_i \|^p
    \leq \sum_{i=n_0+1}^\infty
           \| f_i \|^p
    \to 0 \quad \text{as} \quad n_0 \to \infty.
    .
  \]
  Hence, $(F_n)_{n \in \N} \subseteq Y$ is Cauchy.
  Since $Y$ is complete, it follows that $(F_n)_{n \in \N}$ converges
  to some $F \in Y$; i.e., $F = \sum_{i=1}^\infty f_i$.
 It remains to show that unconditional convergence of the defining series.
  For this, let $\eps > 0$ be arbitrary and choose $N_0 \in \N$
  such that $\sum_{i = N_0 + 1}^\infty \| f_i \|^p < \eps^p$.
  Let $J \subseteq \N$ be any finite set with $J \supseteq \{ 1,\dots,N_0 \}$.
  Since $\| \cdot \|_Y$ is a $p$-norm, $\| \cdot \|^p$ is a metric and hence continuous; therefore,
  \begin{align*}
    \bigg\|
      F - \sum_{j \in J} f_j
    \bigg\|^p
    & = \lim_{n \to \infty}
          \bigg\|
            F_n - \sum_{j \in J} f_j
          \bigg\|^p
      = \lim_{\substack{n \to \infty, \\ n \geq N_0}} \,\,
          \bigg\|
            \sum_{i \in \{ 1,\dots,n \} \setminus J}
              f_i
          \bigg\|^p \\
    & \leq \lim_{\substack{n \to \infty, \\ n \geq N_0}} \,\,
             \sum_{i \in \{ 1,\dots,n \} \setminus J}
               \| f_i \|^p
      \leq \sum_{i = N_0 + 1}^\infty
             \| f_i \|^p
      <    \eps^p .
  \end{align*}
  Since this holds for any finite set $J \subseteq \N$ with $J \supseteq \{ 1,\dots,N_0 \}$,
  the series converges unconditionally.

  \medskip{}

  For the remaining implication, let $(g_n)_{n \in \N}$ be a Cauchy sequence in $Y$.
  Choose a strictly increasing sequence $(n_i)_{i \in \N}$ such that
  $\| g_n - g_m \| \leq 2^{-i}$ for all $n,m \geq n_i$.
  Note that $n_i \geq i$ for all $i \in \N$.
  Define $f_i := g_{n_{i+1}} - g_{n_i}$ and note $\| f_i \| \leq 2^{-i}$,
  so that $\sum_{i \in \N} \| f_i \|^p < \infty$.
  By assumption, $F := \sum_{i=1}^\infty f_i \in Y$, and
  \[
    F
    = \lim_{N \to \infty}
        \sum_{i =1}^N
          f_i
    = \lim_{N \to \infty}
        \sum_{i=1}^N
          (g_{n_{i+1}} - g_{n_i})
    = \lim_{N \to \infty}
        (g_{n_{N+1}} - g_{n_1})
    .
  \]
 For fixed, but arbitrary $\eps > 0$, choose $N_0 \in \N$ such that
  $\| g_n - g_m \|^p < \frac{\eps^p}{2}$ for all $n,m \geq N_0$
  and such that $\| F - (g_{n_{N_0 + 1}} - g_{n_1}) \|^p < \frac{\eps^p}{2}$.
  Then, it holds for all $n \geq N_0$ that
  \begin{align*}
    \| g_n - (F + g_{n_1}) \|^p
    & \leq \| g_n - g_{n_{N_0+1}} \|^p
           + \| F - (g_{n_{N_0+1}} - g_{n_1}) \|^p
      <    \eps^p
      .
  \end{align*}
  Thus, $(g_n)_{n \in \N}$ has the limit
  $\lim_{n \to \infty} g_n = F + g_{n_1} \in Y$,
  showing that $Y$ is complete.
\end{proof}

\begin{lemma}\label{lem:QuasiBanachSolidAbsoluteConvergence}
  Let $(Y, \| \cdot \|_Y)$ be a solid quasi-Banach function space
  on a locally compact group $\group$.
  Suppose that $\| \cdot \|_Y$ is a $p$-norm, with $p \in (0,1]$.

  Let $I \neq \emptyset$ be countable and let $(F_i)_{i \in I} \subseteq Y$
  with $\sum_{i \in I} \| F_i \|_Y^p < \infty$.
  Then the series defining $F := \sum_{i \in I} F_i$ is almost everywhere
  absolutely convergent, and $F \in Y$
  with \[ \| F \|_Y \leq \biggl(\sum_{i \in I} \| F_i \|_Y^p\biggr)^{1/p}.\]
\end{lemma}

\begin{proof}
  The claim is trivial if $I$ is finite; hence, we can assume that $I$
  is countably infinite and then without loss of generality that $I = \N$.
The proof consists of three steps.
  \medskip{}

  \textbf{Step~1.} Let $H_n := \sum_{i = 1}^n |F_i|$ for $n \in \N$.
  In this step, we show that $(H_n)_{n \in \N}$ converges to some $H \in Y$.
  For this, note for $N \in \N$ and $n \geq m \geq N$ that
  \[
    \| H_n - H_m \|_Y^p
    = \bigg\|
        \sum_{i=m+1}^n |F_i|
      \bigg\|_Y^p
    \leq \sum_{i=m+1}^n
           \big\| \, |F_i| \, \big\|_Y^p
    \leq \sum_{i=N}^\infty
           \| F_i \|_Y^p
  \]
  and hence
  \[
    \sup_{n,m \geq N}
      \| H_n - H_m \|_Y
    \leq \bigg(
           \sum_{i=N}^\infty
             \| F_i \|_Y^p
         \bigg)^{1/p}
    \to 0 \quad \text{as} \quad n \to \infty.
  \]
  Hence, $(H_n)_{n \in \N}$ is a Cauchy sequence in $Y$.
  Since $Y$ is complete, the claim follows.

  \medskip{}

  \textbf{Step~2.} We show that the function $H$ from Step~1
  satisfies $H = \sum_{i=1}^\infty |F_i|$ almost everywhere.
  For $n \in \N$, define
  \[
    \Xi_n
    := \bigg|
         H - \sum_{i=1}^n |F_i|
       \bigg|
    =  |H - H_n|
  \]
  Then $\| \Xi_n \|_Y = \| H - H_n \|_Y \to 0$ as $n \to \infty$.
  Setting $\Xi^{(n)} := \inf_{i \geq n} \Xi_i$ yields $0 \leq \Xi^{(n)}  \leq \Xi_i$
  for all $i \geq n$, and hence
  \(
    \| \Xi^{(n)}  \|_Y
    \leq \| \Xi_i \|_Y
    \to 0 \quad \text{as} \quad n \to \infty
  \)
  by solidity of $Y$. Therefore, $\Xi^{(n)}  = 0$ almost everywhere for all $n \in \N$,
  showing that there exists a null-set $N \subseteq \group$ satisfying
  $\Xi^{(n)}  (x) = 0$ for all $n \in \N$ and $x \in \group \setminus N$.
  For $x \in \group \setminus N$, it follows therefore that
  \[
    0
    = \lim_{n \to \infty} \Xi^{(n)}  (x)
    = \liminf_{n \to \infty}
        \Xi_n (x)
    = \liminf_{n \to \infty}
        \bigg|
          H(x)
          - \sum_{i=1}^n
              |F_i (x)|
        \bigg|
    ,
  \]
  so that there exists a subsequence $(n_\ell)_{\ell \in \N}$
  (possibly dependent on $x$) such that
  \[
    0
    = \lim_{\ell \to \infty}
        \bigg|
          H(x) - \sum_{i=1}^{n_\ell} |F_i (x)|
        \bigg|.
  \]
  Hence,
  \(
    H(x)
    = \lim_{\ell \to \infty}
        \sum_{i=1}^{n_\ell}
          |F_i(x)|
    = \sum_{i=1}^\infty
        |F_i (x)|
    .
  \)
  Since this holds for all $x \in \group \setminus N$, we conclude
  $H = \sum_{i=1}^\infty |F_i|$ almost everywhere.

  \medskip{}

  \textbf{Step~3.}
  First, note for $\Omega := \{ x \in \group \colon |H(x)| = \infty \}$ that
  $n \cdot \indicator_{\Omega} \leq |H|$ for arbitrary $n \in \N$. Hence
  $\| \indicator_{\Omega} \|_Y \leq \frac{1}{n} \| H \|_Y \to 0$ as $n\to\infty$,
  so that $\indicator_\Omega = 0$ almost everywhere and thus
  $|H(x)| < \infty$ almost everywhere.
  Combined with Step~2, this shows $\sum_{n=1}^\infty |F_n(x)| < \infty$ almost everywhere.
  Hence, the series defining $F := \sum_{n=1}^\infty F_n$ converges absolutely
  almost everywhere, with $|F(x)| \leq \sum_{n=1}^\infty |F_n(x)|  \leq |H(x)|$
  for a.e. $x \in G$.

  Second, note that $d(f,g) := \| f - g \|_Y^p$ is a metric
  and that $d(H_n, H) \to 0$ as $n \to \infty$. Therefore,
  \(
    \| H_n \|_Y^p
    = d(H_n, 0)
    \to d(H,0)
    = \| H \|_Y^p
    .
  \)
  This means
  \[
    \| F \|_Y^p
    \leq \| H \|_Y^p
    = \lim_{n \to \infty} \| H_n \|_Y^p
    = \lim_{n \to \infty}
        \bigg\|
          \sum_{i=1}^n
            |F_i|
        \bigg\|_Y^p
    \leq \lim_{n \to \infty}
           \sum_{i=1}^n
             \big\| \,|F_i|\, \big\|_Y^p
    = \sum_{i=1}^\infty
        \| F_i \|_Y^p
    ,
  \]
  which  implies $\| F \|_Y \leq \big( \sum_{i=1}^\infty \| F_i \|_Y^p \big)^{1/p}$,
  as claimed.
\end{proof}

\backmatter

\end{document}